\documentclass[11pt]{article} 
\usepackage{geometry}
\usepackage{setspace}
\usepackage{amsmath}
\allowdisplaybreaks
\usepackage{graphicx}
\usepackage{amsthm}
\usepackage{thmtools}
\graphicspath{{}}
\usepackage{caption}
\usepackage{hyperref}
\newtheorem{thm}{Theorem}
\newtheorem{lem}{Lemma}

\usepackage{array}
\usepackage{multirow}
\usepackage{algorithm}
\usepackage{algpseudocode}

\begin{document}

\title{\textbf{\LARGE{High-Dimension, Low Sample Size Asymptotics \\ of Canonical Correlation Analysis}}}
\author{Sungwon Lee, Statistics Department, University of Pittsburgh}
\date{}
\maketitle

\begin{abstract}
An asymptotic behavior of canonical correlation analysis is studied when dimension d grows and the sample size n is fxed. In particular, we are interested in the conditions for which CCA works or fails in the HDLSS situation. This technical report investigates those conditions in a rather simplified setting where there exists one pair of directions in two sets of random variables with non-zero correlation between two sets of scores on them. Proofs and an extensive simulation study supports the findings.
\end{abstract}

\section{Introduction}

Canonical correlation analysis (CCA) introduced in~\cite{23} is a standard statistical tool to explore the relationship between two sets of random variables. Consider $d_{X}$- and $d_{Y}$-dimensional random vectors $X^{(d_{X})}$ and $Y^{(d_{Y})}$,
\begin{align*}
\left(X^{(d_{X})}\right)^{T}=
\begin{bmatrix}
X_{1}, &X_{2}, &\dots, &X_{d_{X}}
\end{bmatrix}, \,\, \left(Y^{(d_{Y})}\right)^{T}=
\begin{bmatrix}
Y_{1}, &Y_{2}, &\dots, &Y_{d_{Y}}
\end{bmatrix}.
\end{align*}
CCA first seeks a pair of $d_{X}$- and $d_{Y}$-dimensional weights vectors $\psi^{(d_{X})}_{X1}$ and $\psi^{(d_{Y})}_{Y1}$ such that two random variables, one being the linear combination of $X_{1}, X_{2}, \dots, X_{d_{X}}$ weighted by the elements of $\psi^{(d_{X})}_{X1}$ and the other being that of $Y_{1}, Y_{2}, \dots, Y_{d_{X}}$ weighted by the elements of $\psi^{(d_{Y})}_{Y1}$, have a maximal correlation,
\begin{align}
\label{sec2:equ1}
(\psi^{(d_{X})}_{X1}, \psi^{(d_{Y})}_{Y1})=\underset{\text{Var}(\langle \psi^{(d_{X})}_{X1}, X^{(d_{X})} \rangle)=\text{Var}(\langle \psi^{(d_{Y})}_{Y1}, Y^{(d_{Y})} \rangle)=1}{\text{argmax}} \text{Cov}(\langle \psi^{(d_{X})}_{X1}, X^{(d_{X})} \rangle, \langle \psi^{(d_{Y})}_{Y1}, Y^{(d_{Y})} \rangle).
\end{align}
Requiring the norms of the weight vectors $\psi^{(d_{X})}_{X1}$ and $\psi^{(d_{Y})}_{Y1}$ to be one, the equation~(\ref{sec2:equ1}) can be written as an equivalent form of,
\begin{align}
\label{sec2:equ2}
(\psi^{(d_{X})}_{X1}, \psi^{(d_{Y})}_{Y1})=\underset{\| \psi^{(d_{X})}_{X1} \|_{2}=\| \psi^{(d_{Y})}_{Y1}\|_{2}=1}{\text{argmax}} \frac{\text{Cov}(\langle \psi^{(d_{X})}_{X1}, X^{(d_{X})} \rangle, \langle \psi^{(d_{Y})}_{Y1}, Y^{(d_{Y})} \rangle)}{\sqrt{\text{Var}(\langle \psi^{(d_{X})}_{X1}, X^{(d_{X})} \rangle)} \sqrt{\text{Var}(\langle \psi^{(d_{Y})}_{Y1}, Y^{(d_{Y})} \rangle)}}.
\end{align}
For convenience, denote the objective function in the right hand side of~(\ref{sec2:equ2}) by $\rho_{P}(\psi^{(d_{X})}, \psi^{(d_{Y})})$,
\begin{align*}
&\rho: R^{d_{X}} \times R^{d_{Y}} \mapsto R\\
&\rho_{P}(\psi^{(d_{X})}, \psi^{(d_{Y})})=\frac{\text{Cov}(\langle \psi^{(d_{X})}, X^{(d_{X})} \rangle, \langle \psi^{(d_{Y})}, Y^{(d_{Y})} \rangle)}{\sqrt{\text{Var}(\langle \psi^{(d_{X})}, X^{(d_{X})} \rangle)} \sqrt{\text{Var}(\langle \psi^{(d_{Y})}, Y^{(d_{Y})} \rangle)}}.
\end{align*}
Subsequent weights vectors $\psi^{(d_{X})}_{Xi}$ and $\psi^{(d_{Y})}_{Yi}$, for $i=1,2,\dots,\text{min}(d_{X}, d_{Y})$, are found by maximizing the objective function $\rho_{P}(\psi^{(d_{X})}, \psi^{(d_{Y})})$,
\begin{align*}
(\psi^{(d_{X})}_{Xi}, \psi^{(d_{Y})}_{Yi})=\underset{\| \psi^{(d_{X})}_{Xi} \|_{2}=\| \psi^{(d_{Y})}_{Yi}\|_{2}=1}{\text{argmax}} \rho_{P}(\psi^{(d_{X})}_{Xi}, \psi^{(d_{Y})}_{Yi}), \,\, i=1,2,\dots,\text{min}(d_{X}, d_{Y}),
\end{align*}
under the constraint that,
\begin{align*}
\text{Cov}(\langle \psi^{(d_{X})}_{Xi}, X^{(d_{X})} \rangle, \langle \psi^{(d_{X})}_{Xj}, X^{(d_{X})} \rangle) &= \text{Cov}(\langle \psi^{(d_{Y})}_{Yi}, Y^{(d_{Y})} \rangle, \langle \psi^{(d_{Y})}_{Yj}, Y^{(d_{Y})} \rangle)\\
&= \text{Cov}(\langle \psi^{(d_{X})}_{Xi}, X^{(d_{X})} \rangle, \langle \psi^{(d_{Y})}_{Yj}, Y^{(d_{Y})} \rangle)\\
&= \text{Cov}(\langle \psi^{(d_{Y})}_{Yi}, Y^{(d_{Y})} \rangle, \langle \psi^{(d_{X})}_{Xj}, X^{(d_{X})} \rangle)\\
&= 0, \,\, i=1,2,\dots,\text{min}(d_{X}, d_{Y}), \,\, j = 1,2,.., i-1 \,\, \text{for each} \,\, i.
\end{align*}
The $i$th pair of weight vectors $\psi^{(d_{X})}_{Xi}$ and $\psi^{(d_{Y})}_{Yi}$ are usually called the $i$th pair of canonical weight vectors (or canonical loadings). The correlation $\rho$ evaluated at the $i$th pair $\psi^{(d_{X})}_{Xi}$ and $\psi^{(d_{Y})}_{Yi}$, denoted by $\rho_{i}^{(d_{X}, d_{Y})}$, is called the $i$th canonical correlation coefficient, that is, $\rho_{i}^{(d_{X}, d_{Y})}=\rho_{P}(\psi^{(d_{X})}_{Xi}, \psi^{(d_{Y})}_{Yi})$.

In practice, we collect two sets of obervations of $d_{X}$- and $d_{Y}$-dimensional random vectors $X^{(d_{X})}$ and $Y^{(d_{Y})}$ on a common set of samples in a $d_{X} \times n$ matrix $\mathbf{X}^{(d_{X})}$ and a $d_{Y} \times n$ matrix $\mathbf{Y}^{(d_{Y})}$, respectively. We row-center $\mathbf{X}^{(d_{X})}$ and $\mathbf{X}^{(d_{X})}$ and let $\hat{\mathbf{\Sigma}}^{(d_{X})}_{X}, \hat{\mathbf{\Sigma}}^{(d_{Y})}_{Y}$ and $\hat{\mathbf{\Sigma}}^{(d_{X},d_{Y})}_{XY}$ be a covariance matrix of $X^{(d_{X})}$, a covariance matrix of $Y^{(d_{Y})}$ and a cross-covariance matrix of $X^{(d_{X})}$ and $Y^{(d_{Y})}$,
\begin{align*}
\hat{\mathbf{\Sigma}}^{(d_{X})}_{X}=\frac{1}{n}\mathbf{X}^{(d_{X})} \left(  \mathbf{X}^{(d_{X})} \right)^{T}, \,\, \hat{\mathbf{\Sigma}}^{(d_{Y})}_{Y}=\frac{1}{n}\mathbf{Y}^{(d_{Y})} \left( \mathbf{Y}^{(d_{Y})} \right)^{T}, \,\, \hat{\mathbf{\Sigma}}^{(d_{X},d_{Y})}_{XY}=\frac{1}{n}\mathbf{X}^{(d_{X})} \left( \mathbf{Y}^{(d_{Y})} \right)^{T}.
\end{align*}
For the case where the sample size $n$ is greater than $d_{X}$ and $d_{Y}$, the estimation of sample canonical weight vectors $(\hat{\psi}^{(d_{X})}_{Xi}, \hat{\psi}^{(d_{Y})}_{Yi})$ and sample canonical correlation coefficients $\hat{\rho}_{i}^{(d_{X}, d_{Y})}$ are done through singular value decomposition of the matrix $\hat{\mathbf{R}}^{(d_{X}, d_{Y})}$, 
\begin{align}
\label{sec2:equ3}
\begin{aligned}
\hat{\mathbf{R}}^{(d_{X},d_{Y})}&= \left(\hat{\mathbf{\Sigma}}^{(d_{X})}_{X}\right)^{-\frac{1}{2}} \hat{\mathbf{\Sigma}}^{(d_{X},d_{Y})}_{XY} \left(\hat{\mathbf{\Sigma}}^{(d_{Y})}_{Y}\right)^{-\frac{1}{2}},\\
\text{SVD}(\hat{\mathbf{R}}^{(d_{X},d_{Y})})&=\sum_{i=1}^{\text{min}(d_{X}, d_{Y})} \hat{\lambda}_{Ri}^{(d_{X}, d_{Y})} \hat{\eta}_{RXi}^{(d_{X})} \left ( \hat{\eta}_{RYi}^{(d_{Y})} \right )^{T},
\end{aligned}
\end{align}
where $\hat{\lambda}_{Ri}^{(d)}$ is a sample singular value with $\hat{\lambda}_{R1}^{(d)} \ge \hat{\lambda}_{R2}^{(d)} \ge \dots \ge \hat{\lambda}_{R\text{min}(d_{X}, d_{Y})}^{(d)} \ge 0$, and $(\hat{\eta}_{RXi}^{(d_{X})}, \hat{\eta}_{RYi}^{(d_{Y})})$ is a pair of left and right sample singular vectors corresponding to $\hat{\lambda}_{Ri}^{(d)}$. Then, the $i$th sample canonical correlation coefficient $\hat{\rho}_{i}^{(d)}$ is found to be,
\begin{align*}
\hat{\rho}_{i}^{(d_{X}, d_{Y})}=\hat{\lambda}_{Ri}^{(d_{X}, d_{Y})}.
\end{align*}
The $i$th pair of canonical weight vectors $\hat{\psi}_{Xi}^{(d_{X})}$ and $\hat{\psi}_{Yi}^{(d_{Y})}$ are obtained by unscaling and normalzing the $i$th pair of sample singular vectors $\hat{\eta}_{RXi}^{(d_{X})}$ and $\hat{\eta}_{RYi}^{(d_{Y})}$,
\begin{align}
\label{sec2:equ4}
\hat{\psi}_{Xi}^{(d_{X})}=\frac{\left(\hat{\mathbf{\Sigma}}^{(d_{X})}_{X}\right)^{-\frac{1}{2}}\hat{\eta}_{RXi}^{(d_{X})}}{\left\| \left(\hat{\mathbf{\Sigma}}^{(d_{X})}_{X}\right)^{-\frac{1}{2}}\hat{\eta}_{RXi}^{(d_{X})} \right\|_{2}}, \,\, \hat{\psi}_{Yi}^{(d_{Y})}=\frac{\left(\hat{\mathbf{\Sigma}}^{(d_{Y})}_{Y}\right)^{-\frac{1}{2}}\hat{\eta}_{RYi}^{(d_{Y})}}{\left\| \left(\hat{\mathbf{\Sigma}}^{(d_{Y})}_{Y}\right)^{-\frac{1}{2}}\hat{\eta}_{RYi}^{(d_{Y})} \right\|_{2}}.
\end{align}
The projection of the data matrix $\mathbf{X}^{(d_{X})}$ onto the $i$th sample canonical weight vector $\hat{\psi}_{Xi}^{(d_{X})}$ gives the canonical scores (or canonical variables) of $\mathbf{X}^{(d_{X})}$ with respect to $\hat{\psi}_{Xi}^{(d_{X})}$ and similarly for $\mathbf{X}^{(d_{X})}$. Although powerful, CCA has several disadvantages. first, use of CCA is practically restricted to the case of two sets of data even if there is an attempt to generalize it to more than two sets of data~\cite{28}. Second, CCA components are estimable only if the sample size $n$ is greater than $d_{X}$ and $d_{Y}$. It is well know that, when $n < \text{max}(d_{X}, d_{Y})$, one can construct an infinite number of sample canonical weight vector pairs with their correlation of one. Moreover, overfitting is often a problem even when $n > d_{X}$ and $d_{Y}$. Hence, CCA is often considered not relible in high-dimensional data sets. We, however, will show that, even in the case where sample size $n$ is less than $d_{X}$ or $d_{Y}$, some sample canonical weight vectors is estimable and furthremore consistent under a certain condition. 

As high-dimensional data are increasingly common these days, where a large number of variables are measured for each object, there is a strong need to investigate the behavior of estimates resulting from the application of standard statistical tools such as CCA to a high-dimensional case (that is, scalability of those tools). Literature in the HDLSS asymptotic study of CCA is very limited, while the behavior of PCA components under the similar high-dimensional condition is well-studied in~\cite{37,38}. This might be in part because CCA is not as widely used as PCA, which is almost an indispensible tool for dimension reduction of high-dimensional data prevalent these days, and in part due to the complicated estimation steps involving an inverse operator as in~(\ref{sec2:equ3}), which makes the analysis not straightforward. 

A relevant work is first addressed in~\cite{36}, where the asymptotic behavior of sample singular vectors and singular values are analysed under a HDLSS setting. In~\cite{35}, the similar study of CCA is elaborated on, but their proof should have considered the fact that an infinite sum of quantities converging to zero does not neccessarily approach to zero. The HDLSS asymptotic behavior of CCA components in this chapter will be studied in relatively a simple population structure and serves as a groundwork for further analysis.

\section{Assumptions and Definitions}

Without loss of generality for the case where the dimensions of two random vectors $X^{(d_{X})}$ and $Y^{(d_{Y})}$ grow in a sense that $d_{X}/d_{Y} \to 1$, we set $d_{X}=d_{Y}$ and consider two random vectors $X^{(d)}$ and $Y^{(d)}$ of a same dimension with mean zero. We assume that covariance structure of $X^{(d)}$ and $Y^{(d)}$ follows a simple spiked model as in~\cite{32}, where the leading eigenvalues of their covariance matrix is considerably larger than the rest. In specific, let $\mathbf{\Sigma}^{(d)}_{X}$ and $\mathbf{\Sigma}^{(d)}_{Y}$ be the covariance matrices of $X^{(d)}$ and $Y^{(d)}$. Then, a spiked model can be easily understood via eigendecomposition of $\mathbf{\Sigma}^{(d)}_{X}$ and $\mathbf{\Sigma}^{(d)}_{Y}$, 
\begin{align}
\label{sec3:equ6}
\mathbf{\Sigma}^{(d)}_{X}=\sum_{i=1}^{d} \lambda_{Xi}^{(d)} \xi_{Xi}^{(d)} \left ( \xi_{Xi}^{(d)} \right )^{T}, \,\, \mathbf{\Sigma}^{(d)}_{Y}=\sum_{j=1}^{d} \lambda_{Yj}^{(d)} \xi_{Yj}^{(d)} \left ( \xi_{Yj}^{(d)} \right )^{T},
\end{align}
where $\lambda_{Xi}^{(d)}$ is an polpulation eigenvalue (or population PC variance) with $\lambda_{X1}^{(d)} \ge \lambda_{X2}^{(d)} \ge \dots \ge \lambda_{Xd}^{(d)} \ge 0$, $\xi_{Xi}^{(d)}$ is an population eigenvector (or population PC direction) with $\| \xi_{Xi}^{(d)} \|_{2}=1$ and $\langle \xi_{Xi}^{(d)}, \xi_{Xj}^{(d)} \rangle=0$ for $i \not= j$ and similarly for $\lambda_{Yj}^{(d)}$ and $\xi_{Yj}^{(d)}$. Here, we set,
\begin{align}
\label{sec3:equ7}
\begin{aligned}
\lambda_{X1}^{(d)}&=\sigma_{X}^{2}d^{\alpha} \,\, \text{and} \,\, \lambda_{Xi}^{(d)}=\tau_{X}^{2} \,\, \text{for} \,\, i=2,3,\dots, d,\\
\lambda_{Y1}^{(d)}&=\sigma_{Y}^{2}d^{\alpha} \,\, \text{and} \,\, \lambda_{Yj}^{(d)}=\tau_{Y}^{2} \,\, \text{for} \,\, j=2,3,\dots, d,
\end{aligned}
\end{align}
where one sees that the leading eigenvalues $\lambda_{X1}^{(d)}$ and $\lambda_{Y1}^{(d)}$ become dominating the rest as $d \to \infty$. We now set up the population canonical components. We assume that the two random vector is related by a pair of canonical weight vectors with its canonical correlation coefficient of $\rho$. The population canonical weight vector $\psi_{X}^{(d)}$ in the $X^{(d)}$ part is a linear combination of two eigenvectors $\xi_{X1}^{(d)}$ and $\xi_{X2}^{(d)}$ without loss of generality ($\xi_{X2}^{(d)}$ can be replaced with $\xi_{Xi}^{(d)}$ for any $i$) and similarly for the other population canonical weight vector $\psi_{Y}^{(d)}$ in the $Y^{(d)}$ part,
\begin{align}
\label{sec3:equ13}
\psi_{X}^{(d)}=\cos \theta_{X} \xi_{X1}^{(d)} + \sin \theta_{X} \xi_{X2}^{(d)}, \,\, \psi_{Y}^{(d)}=\cos \theta_{Y} \xi_{Y1}^{(d)} + \sin \theta_{Y} \xi_{Y2}^{(d)}.
\end{align}
Note that the angle between $\psi_{X}^{(d)}$ and $\xi_{X1}^{(d)}$ is $\theta_{X}$ and that the angle between $\psi_{Y}^{(d)}$ and $\xi_{Y1}^{(d)}$ is $\theta_{Y}$ as $\langle \psi_{X}^{(d)}, \xi_{X1}^{(d)} \rangle=\text{cos}\theta_{X}$ and $\langle \psi_{Y}^{(d)}, \xi_{Y1}^{(d)} \rangle=\text{cos}\theta_{Y}$. At this point, we apply the change of basis to the spaces of $X^{(d)}$ and $Y^{(d)}$ so that the eigenvectors $\{ \xi_{Xi}^{(d)} \}_{i=1}^{d}$ and $\{ \xi_{Yi}^{(d)} \}_{j=1}^{d}$ are represented by the standard basis $\{e_{i}^{(d)}\}_{k=1}^{d}$. Then, the canonical weight vectors $(\psi_{X}^{(d)}, \psi_{Y}^{(d)})$ given in~(\ref{sec3:equ13}) is rewritten as,
\begin{align*}
\psi_{X}^{(d)}=\cos \theta_{X} e_{1}^{(d)} + \sin \theta_{X} e_{2}^{(d)}, \,\, \psi_{Y}^{(d)}=\cos \theta_{Y} e_{1}^{(d)} + \sin \theta_{Y} e_{2}^{(d)},
\end{align*}
and the covariance structures given in~(\ref{sec3:equ6}) and~(\ref{sec3:equ7}) are described as,
\begin{align}
\label{sec3:equ12}
\mathbf{\Sigma}_{X}^{(d)}=\underset{d \times d}{\text{diag}}(\sigma_{X}^{2}d^{\alpha},\,\, \tau_{X}^{2},\,\, \tau_{X}^{2},\,\, \dots,\,\, \tau_{X}^{2}),\,\, \mathbf{\Sigma}_{Y}^{(d)}=\underset{d \times d}{\text{diag}}(\sigma_{Y}^{2}d^{\alpha},\,\, \tau_{Y}^{2},\,\, \tau_{Y}^{2},\,\, \dots,\,\, \tau_{Y}^{2}),
\end{align}
where $\text{diag}(\bullet)$ is a square matrix with entries of $\bullet$ in the main diagonal and 0 off of it. With these population covariance structures and canonical components, the multivariate version of the corallory~\ref{cor1} gives the cross-covariance structure of $X^{(d)}$ and $Y^{(d)}$ as follows,
\begin{align}
\label{sec3:equ8}
\mathbf{\Sigma}_{XY}^{(d)} =
\begin{bmatrix}
\frac{\rho \sigma_{X}^{2} \sigma_{Y}^{2} d^{2\alpha} \text{cos}\theta_{X} \text{cos}\theta_{Y}}{AB} &\frac{\rho \sigma_{X}^{2} d^{\alpha} \tau_{Y}^{2} \text{cos}\theta_{X} \text{sin}\theta_{Y}}{AB} &\underset{1 \times (d-2)}{\mathbf{0}}\\
\frac{\rho \tau_{X}^{2} \sigma_{Y}^{2} d^{\alpha} \text{sin}\theta_{X} \text{cos}\theta_{Y}}{AB} &\frac{\rho \tau_{X}^{2} \tau_{Y}^{2} \text{sin}\theta_{X} \text{sin}\theta_{Y}}{AB} &\underset{1 \times (d-2)}{\mathbf{0}}\\
\underset{(d-2) \times 1}{\mathbf{0}} &\underset{(d-2) \times 1}{\mathbf{0}} &\underset{(d-2) \times (d-2)}{\mathbf{0}}
\end{bmatrix},
\end{align}
where
\begin{align*}
A=\sqrt{\sigma_{X}^{2} d^{\alpha} \text{cos}^{2} \theta_{X} + \tau_{X}^{2} \text{sin}^{2} \theta_{X}}, \,\, B=\sqrt{\sigma_{Y}^{2} d^{\alpha} \text{cos}^{2} \theta_{Y} + \tau_{Y}^{2} \text{sin}^{2} \theta_{Y}}.
\end{align*}

Then the covariance and cross-covariance structure of $X^{(d)}$ and $Y^{(d)}$ is succintly described by the covariance structure of the concatenated random vector $T^{(2d)}$,
\begin{align}
\label{sec3:equ9}
T^{(2d)}=
\begin{bmatrix}
X^{(d)}\\
Y^{(d)}
\end{bmatrix}, \,\,
\mathbf{\Sigma}_{T}^{(2d)}=
\begin{bmatrix}
\mathbf{\Sigma}_{X}^{(d)} &\mathbf{\Sigma}_{XY}^{(d)}\\
\left( \mathbf{\Sigma}_{XY}^{(d)} \right)^{T} &\mathbf{\Sigma}_{Y}^{(d)}
\end{bmatrix}.
\end{align}
To make the analysis a bit easy, we are going to work with a different representation of $X^{(d)}$ and $Y^{(d)}$. Let $Z^{(d)}$ be the $2d$-dimensional standard normal random vector. Then, $T^{(2d)}$ can be expressed as,
\begin{align}
\label{sec3:equ10}
T^{(2d)}=
\begin{bmatrix}
X^{(d)}\\
Y^{(d)}
\end{bmatrix}=\left(\mathbf{\Sigma}_{T}^{(2d)} \right)^{\frac{1}{2}} Z^{(2d)}, \,\, Z^{(2d)} \sim N \left( \underset{2d \times 1}{0}, \underset{2d \times 2d}{\mathbf{I}} \right).
\end{align}

We state some definitions used in the estimation. Since the dimensionality $d$ is much larger than the sample size $n$ in the HDLSS setting, the estimation step~(\ref{sec2:equ3}) of canonical components is problematic as the sample covariance matrices $\hat{\mathbf{\Sigma}}^{(d)}_{X}$ and $\hat{\mathbf{\Sigma}}^{(d)}_{Y}$ are singular. There are two ways to handle this singularity situation. The first one is to add a minute perturbation of $\epsilon \mathbf{I}$ for a small $\epsilon > 0$ to $\hat{\mathbf{\Sigma}}^{(d)}_{X}$ and $\hat{\mathbf{\Sigma}}^{(d)}_{Y}$ and the second is to use a pseudoinverse such as Moore-Penrose pseudoinverse. We use the pseudoinverse obtained from the eigendecomposition of the sample covariance matrices,
\begin{align}
\label{sec3:equ11}
\hat{\mathbf{\Sigma}}^{(d)}_{X}=\sum_{i=1}^{n} \hat{\lambda}_{Xi}^{(d)} \hat{\xi}_{Xi}^{(d)} \left ( \hat{\xi}_{Xi}^{(d)} \right )^{T}, \,\, \hat{\mathbf{\Sigma}}^{(d)}_{Y}=\sum_{j=1}^{n} \hat{\lambda}_{Yj}^{(d)} \hat{\xi}_{Yj}^{(d)} \left ( \hat{\xi}_{Yj}^{(d)} \right )^{T},
\end{align}
where $\hat{\lambda}_{Xi}^{(d)}$ is an sample eigenvalue (or sample PC variance) with $\hat{\lambda}_{X1}^{(d)} \ge \hat{\lambda}_{X2}^{(d)} \ge \dots \ge \hat{\lambda}_{Xd}^{(n)} \ge 0$, $\hat{\xi}_{Xi}^{(d)}$ is an sample eigenvector (or sample PC direction) with $\| \hat{\xi}_{Xi}^{(d)} \|_{2}=1$ and $\langle \hat{\xi}_{Xi}^{(d)}, \hat{\xi}_{Xj}^{(d)} \rangle=0$ for $i \not= j$ and similarly for $\hat{\lambda}_{Yj}^{(d)}$ and $\hat{\xi}_{Yj}^{(d)}$. The pseudoinverse we employ is defined as,
\begin{align}
\label{sec3:equ5}
\left(\hat{\mathbf{\Sigma}}^{(d)}_{X}\right)^{-1}=\sum_{i=1}^{n} \left(\hat{\lambda}_{Xi}^{(d)}\right)^{-1} \hat{\xi}_{Xi}^{(d)} \left ( \hat{\xi}_{Xi}^{(d)} \right )^{T}, \,\, \left(\hat{\mathbf{\Sigma}}^{(d)}_{Y}\right)^{-1}=\sum_{j=1}^{d} \left(\hat{\lambda}_{Yj}^{(d)}\right)^{-1} \hat{\xi}_{Yj}^{(d)} \left ( \hat{\xi}_{Yj}^{(d)} \right )^{T}.
\end{align}
Then, the sample canonical correlation coefficient $\hat{\rho}_{i}^{(d)}$ is found as an $i$th sample singular value from the SVD of the matrix $\hat{R}^{(d)}$ defined in~(\ref{sec2:equ4}). The sample canonical weight vectors $\hat{\psi}_{Xi}^{(d)}$ and $\hat{\psi}_{Yi}^{(d)}$ corresponding to $\hat{\rho}_{i}^{(d)}$ are obtained from~(\ref{sec2:equ4}) using the pseudoinverses~(\ref{sec3:equ5}).

The success and failure of CCA can be described by the consistency of the sample canonical weight vectors $\hat{\psi}_{X}^{(d)}$ and $\hat{\psi}_{Y}^{(d)}$ with their population counterpart $\psi_{X}^{(d)}$ and $\psi_{Y}^{(d)}$ under the limiting operation of $d \to \infty$ and $n$ fixed. Using the angle as a measure of consistency, we say that $\hat{\psi}_{X}^{(d)}$ (similarly $\hat{\psi}_{Y}^{(d)}$) is,
\begin{itemize}
\item Consistent with $\psi_{X}^{(d)}$ if angle($\hat{\psi}_{X}^{(d)}$, $\psi_{X}^{(d)}$) $\to 0$ as $d \to \infty$,
\item Inonsistent with $\psi_{X}^{(d)}$ if angle($\hat{\psi}_{X}^{(d)}$, $\psi_{X}^{(d)}$) $\to a$, for $0 < a < \pi/2$, as $d \to \infty$,
\item Strongly inonsistent with $\psi_{X}^{(d)}$ if angle($\hat{\psi}_{X}^{(d)}$, $\psi_{X}^{(d)}$) $\to \pi/2$ as $d \to \infty$.
\end{itemize}
Strong inconsistency implies that the estimate $\hat{\psi}_{X}^{(d)}$ and $\hat{\psi}_{Y}^{(d)}$ become completely oblivious of its population structure and reduce to arbitrary quantities, as indicated in the fact that $pi/2$ is indeed a largest angle possible between two vectors.

\section{Main Theorem and Interpretation}

\subsection{Main theorem}
Let $X^{(d)}$ and $Y^{(d)}$ be the $d$-dimensional random vectors from the multivariate Gaussian distributions with mean 0 and the simple spiked covariance matrices $\mathbf{\Sigma}_{X}^{(d)}$ and $\mathbf{\Sigma}_{Y}^{(d)}$ described in~(\ref{sec3:equ6}) and~(\ref{sec3:equ7}). With the population canonical correlation coefficient $\rho$ for $0 \le \rho \le 1$, define the population canonical weight vectors $\psi_{X}^{(d)}$ and $\psi_{Y}^{(d)}$ as,
\begin{align*}
\psi_{X}^{(d)}=\cos \theta_{X} \xi_{X1}^{(d)} + \sin \theta_{X} \xi_{X2}^{(d)}, \,\, \psi_{Y}^{(d)}=\cos \theta_{Y} \xi_{Y1}^{(d)} + \sin \theta_{Y} \xi_{Y2}^{(d)}
\end{align*}
so that the angle between $\psi_{X}^{(d)}$ and $\xi_{X1}^{(d)}$ is $\theta_{X}$, and the angle between $\psi_{Y}^{(d)}$ and $\xi_{Y1}^{(d)}$ is $\theta_{Y}$. Then, the cross-covariance matrix $\mathbf{\Sigma}_{XY}^{(d)}$ of $X^{(d)}$ and $Y^{(d)}$ is found as in~\ref{sec3:equ8}. The two random variables $X^{(d)}$ and $Y^{(d)}$ can be written in a equivalent form, 
\begin{align}
\label{main1}
\begin{bmatrix}
X^{(d)}\\
Y^{(d)}
\end{bmatrix}=
\begin{bmatrix}
\mathbf{\Sigma}_{X}^{(d)} &\mathbf{\Sigma}_{XY}^{(d)}\\
\left( \mathbf{\Sigma}_{XY}^{(d)} \right)^{T} &\mathbf{\Sigma}_{Y}^{(d)}
\end{bmatrix}Z^{(2d)},
\end{align}
where $Z^{(2d)}$ is a $2d$-dimensional standard normal random vector. The data matrix whose columns consist of $n$ i.i.d. samples from the distribution~\ref{main1} is written as,
\begin{align}
\label{main2}
\begin{bmatrix}
\mathbf{X}^{(d)}\\
\mathbf{Y}^{(d)}
\end{bmatrix}=
\begin{bmatrix}
\mathbf{\Sigma}_{X}^{(d)} &\mathbf{\Sigma}_{XY}^{(d)}\\
\left( \mathbf{\Sigma}_{XY}^{(d)} \right)^{T} &\mathbf{\Sigma}_{Y}^{(d)}
\end{bmatrix}\mathbf{Z}^{(2d)},
\end{align}
where the columns of $\mathbf{Z}^{(2d)}$ consist of $n$ i.i.d. samples from $2d$-dimensional standard normal ditribution. Denote by $z_{1}$ and $z_{2}$ the first and $(d+1)$th rows of $\mathbf{Z}^{(2d)}$ corresponding to the first rows of $\mathbf{X}^{(d)}$ and $\mathbf{Y}^{(d)}$ respectively. Then, as $d \to \infty$ with the sample size $n$ being fixed, the limiting behaviors of the sample canonical correlation coefficient $\hat{\rho}_{i}^{(d)}$ and its corresponding sample canonical weight vectors $\hat{\psi}_{Xi}^{(d)}$ and $\hat{\psi}_{Yi}^{(d)}$ obtained from the data~\ref{main2} are as follows,
\begin{thm}[Main result of the HDLSS asymptotic analysis of CCA.] 
\label{thm1}
(i) $\alpha > 1$
\begin{align*}
&\emph{\text{angle}} \left( \hat{\psi}_{X1}^{(d)}, \psi_{X}^{(d)} \right) \mathop{\longrightarrow}_{d \to \infty}^{P} \theta_{X}, \,\, \emph{\text{angle}} \left( \hat{\psi}_{Y1}^{(d_{Y})}, \psi_{Y}^{(d)} \right) \mathop{\longrightarrow}_{d \to \infty}^{P} \theta_{Y}, \,\, \hat{\rho}_{1}^{(d)} \mathop{\longrightarrow}_{d \to \infty}^{P} \frac{\langle m_{1}, m_{2} \rangle}{\| m_{1} \|_{2} \| m_{2} \|_{2}},\\
&\emph{\text{angle}} \left( \hat{\psi}_{Xi}^{(d)}, \psi_{X}^{(d)} \right) \mathop{\longrightarrow}_{d \to \infty}^{P} 0, \,\, \emph{\text{angle}} \left( \hat{\psi}_{Yi}^{(d)}, \psi_{Y}^{(d)} \right) \mathop{\longrightarrow}_{d \to \infty}^{P} 0, \,\, \hat{\rho}_{i}^{(d)} \mathop{\longrightarrow}_{d \to \infty}^{P} 0, \,\, i=2,3,\dots,n,
\end{align*}
where 
\begin{align*}
m_{1} &= (\sqrt{C_{1}}A_{1}^{2}+\sqrt{C_{2}}B_{1}^{2})z_{1}+(\sqrt{C_{1}}A_{1}A_{2}+\sqrt{C_{2}}B_{1}B_{2})z_{2},\\ 
m_{2} &= (\sqrt{C_{1}}A_{1}A_{2}+\sqrt{C_{2}}B_{1}B_{2})z_{1}+(\sqrt{C_{1}}A_{1}^{2}+\sqrt{C_{2}}B_{1}^{2})z_{2},
\end{align*}
where
\begin{align*}
&z_{1}, z_{2} \overset{i.i.d.}{\sim} N \left(\underset{n \times 1}{0}, \underset{n \times n}{\mathbf{I}} \right),\\
&C_{1}=\frac{\sigma_{X}^{2} + \sigma_{Y}^{2} + \sqrt{\left(\sigma_{X}^{2}\right)^{2} - 2\sigma_{X}^{2}\sigma_{Y}^{2} + 4\sigma_{X}^{2}\sigma_{Y}^{2}\rho^{2} + \left(\sigma_{Y}^{2}\right)^{2}}}{2},\\
&C_{2}=\frac{\sigma_{X}^{2} + \sigma_{Y}^{2} - \sqrt{\left(\sigma_{X}^{2}\right)^{2} - 2\sigma_{X}^{2}\sigma_{Y}^{2} + 4\sigma_{X}^{2}\sigma_{Y}^{2}\rho^{2} + \left(\sigma_{Y}^{2}\right)^{2}}}{2},\\
&A_{1}=\frac{C_{1}-\sigma_{Y}^{2}}{\rho\sigma_{X}\sigma_{Y}}/\sqrt{ \left(\frac{C_{1}-\sigma_{Y}^{2}}{\rho\sigma_{X}\sigma_{Y}}\right)^{2} +1 }, \,\, A_{2}=1/\sqrt{ \left(\frac{C_{1}-\sigma_{Y}^{2}}{\rho\sigma_{X}\sigma_{Y}}\right)^{2} +1},\\
&B_{1}=\frac{C_{2}-\sigma_{Y}^{2}}{\rho\sigma_{X}\sigma_{Y}}/\sqrt{ \left(\frac{C_{2}-\sigma_{Y}^{2}}{\rho\sigma_{X}\sigma_{Y}}\right)^{2} +1}, \,\, B_{2}=1/\sqrt{ \left(\frac{C_{2}-\sigma_{Y}^{2}}{\rho\sigma_{X}\sigma_{Y}}\right)^{2} +1}.
\end{align*}
(ii) $\alpha < 1$
\begin{align*}
\emph{\text{angle}} \left( \hat{\psi}_{Xi}^{(d)}, \psi_{X}^{(d)} \right) \mathop{\longrightarrow}_{d \to \infty}^{P} 0, \,\, \emph{\text{angle}} \left( \hat{\psi}_{Yi}^{(d)}, \psi_{Y}^{(d)} \right) \mathop{\longrightarrow}_{d \to \infty}^{P} 0, \,\, \hat{\rho}_{i}^{(d)} \mathop{\longrightarrow}_{d \to \infty}^{P} 1, \,\, i=1,2,\dots,n. 
\end{align*}
\end{thm}

\subsection{Interpretation}

The theorem~\ref{thm1} implies that where $\hat{\psi}_{X1}^{(d)}$ and $\hat{\psi}_{Y1}^{(d)}$ converge to depend heavily on the size of the variance $d^{\alpha}$ of the population eigenvector $\xi_{X1}^{(d)}$ and $\xi_{Y1}^{(d)}$. That is, the estimates $\hat{\psi}_{X1}^{(d)}$ and $\hat{\psi}_{Y1}^{(d)}$ tend to converge to the eigenvectors $\xi_{X1}^{(d)}$ and $\xi_{Y1}^{(d)}$ when their eigenvalues $\sigma_{X}^{2} d^{\alpha}$ and $\sigma_{Y}^{2} d^{\alpha}$ become strong enough ($\alpha > 1$) as $d \to \infty$. Briefly, we summarize results. The sample canonical weight vector $\hat{\psi}_{X1}^{(d)}$ (similarly $\hat{\psi}_{Y1}^{(d)}$) is, 
\begin{itemize}
\item Consistent with $\psi_{X}^{(d)}$ if $\alpha > 1$ and angle($\psi_{X}^{(d)}$, $\xi_{X1}^{(d)}$) $= 0$ as $d \to \infty$,
\item Inonsistent with $\psi_{X}^{(d)}$ if $\alpha > 1$ and angle($\psi_{X}^{(d)}$, $\xi_{X1}^{(d)}$) $= \theta_{X}$, for $0 < \theta_{X} < \pi/2$, as $d \to \infty$,
\item Strongly inonsistent with $\psi_{X}^{(d)}$ if $\alpha < 1$ or if $\alpha > 1$ and angle($\psi_{X}^{(d)}$, $\xi_{X1}^{(d)}$) $= \pi/2$ as $d \to \infty$.
\end{itemize}

The asymptotic behavior of the sample canonical correlation coefficient $\hat{\rho}_{1}^{(d)}$ is not straightforward to imagine. Let's take a simple example where $\sigma_{X}^{2}=1, \sigma_{X}^{2}=1, \tau_{X}^{2}=1$ and $\tau_{Y}^{2}=1$ in the spiked covariance structure in~(\ref{sec3:equ6}) and~(\ref{sec3:equ7}). In this case, referring to the theorem~\ref{thm1}, the sample canonical correlation coefficient $\hat{\rho}_{1}^{(d)}$ converges in probability to the following random quantity,
\begin{align*}
\hat{\rho}_{1}^{(d)} \mathop{\longrightarrow}_{d \to \infty}^{P} \frac{\langle m_{1}, m_{2} \rangle}{\| m_{1} \|_{2} \| m_{2} \|_{2}},
\end{align*}
where
\begin{align*}
&m_{1} = \left( \frac{\sqrt{1+\rho}+\sqrt{1-\rho}}{2} \right) z_{1} + \left( \frac{\sqrt{1+\rho}-\sqrt{1-\rho}}{2} \right) z_{2},\\
&m_{2} = \left( \frac{\sqrt{1+\rho}-\sqrt{1-\rho}}{2} \right) z_{1} + \left( \frac{\sqrt{1+\rho}+\sqrt{1-\rho}}{2} \right) z_{2}.
\end{align*}
Note that $z_{1}$ and $z_{2}$ are samples from $n$-dimensional multivariate standard normal distribution. It can be easily verified that each element $m_{1i}$ of $m_{1}$ (similarly for $m_{2i}$ of $m_{2}$) follows a standard normal distribution, 
\begin{equation*}
\begin{aligned}
m_{1,i}&=\left( \frac{\sqrt{1+\rho}+\sqrt{1-\rho}}{2} \right) z_{1i} + \left( \frac{\sqrt{1+\rho}-\sqrt{1-\rho}}{2} \right) z_{2i} \sim N(0,1),\\
m_{2,i}&=\left( \frac{\sqrt{1+\rho}+\sqrt{1-\rho}}{2} \right) z_{1i} + \left( \frac{\sqrt{1+\rho}-\sqrt{1-\rho}}{2} \right) z_{2i} \sim N(0,1),
\end{aligned}
\end{equation*}
which leads to,
\begin{align*}
\| m_{1} \|_{2} \sim \sqrt{\chi^{2}_{n}}, \,\, \| m_{2} \|_{2} \sim \sqrt{\chi^{2}_{n}},
\end{align*}
where $\chi^{2}_{n}$ denotes the chi-square distribution with degree of freedom of $n$. Since the numerator part $\langle m_{1}, m_{2} \rangle$ is not a degenerate random quantity, one sees that $\hat{\rho}_{1}^{(d)}$ does not converge to a trivial random variable such as 1.

Now increase the sample size $n$ to see which value the sample canonical correlation coefficient $\hat{\rho}_{1}^{(d)}$ converges to. By the law of large numbers and noting that the elements $m_{1,i}$ and $m_{2,i}$ are from i.i.d. standard normal distribution,
\begin{align*}
\frac{\| m_{1} \|_{2}^{2}}{n}=\sum_{i=1}^{n} \frac{m_{1i}^{2}}{n} \mathop{\longrightarrow}_{n \to \infty}^{P} 1, \,\, \frac{\| m_{2} \|_{2}^{2}}{n}=\sum_{j=1}^{n} \frac{m_{2j}^{2}}{n} \mathop{\longrightarrow}_{n \to \infty}^{P} 1.
\end{align*}
Furthurmore, noting that $m_{1}$ and $m_{2}$ are i.i.d. samples,
\begin{align*}
\frac{\langle m_{1}, m_{2} \rangle}{n}=&\left(\frac{\sqrt{1+\rho}+\sqrt{1-\rho}}{2}\right) \left(\frac{\sqrt{1+\rho}-\sqrt{1-\rho}}{2}\right) \sum_{i=1}^{n} \frac{z_{1i}^{2}}{n}\\ 
&+\left(\frac{\sqrt{1+\rho}+\sqrt{1-\rho}}{2}\right) \left(\frac{\sqrt{1+\rho}-\sqrt{1-\rho}}{2}\right) \sum_{i=1}^{n} \frac{z_{2i}^{2}}{n}\\
&+\left(\frac{\sqrt{1+\rho}+\sqrt{1-\rho}}{2}\right)^{2} \sum_{i=1}^{n} \frac{z_{1i}z_{2i}}{n}\\
&+\left(\frac{\sqrt{1+\rho}-\sqrt{1-\rho}}{2}\right)^{2} \sum_{i=1}^{n} \frac{z_{1i}z_{2i}}{n}\\
&\mathop{\longrightarrow}_{n \to \infty}^{P} 2\left(\frac{\sqrt{1+\rho}+\sqrt{1-\rho}}{2}\right) \left(\frac{\sqrt{1+\rho}-\sqrt{1-\rho}}{2}\right)=\rho,
\end{align*}
which confirms the conventional large sample asymptotic property of the statistic $\hat{\rho}_{1}^{(d)}$,
\begin{align*}
\hat{\rho}_{1}^{(d)} \mathop{\longrightarrow}_{d,n \to \infty}^{P} \rho.
\end{align*}

\section{Proof}

We state here the theorem on the HDLSS asymptotic behavior of the sample eigenvalues and vectors (or sample PC variances and directions) included in~\cite{38} as frequently referred in this chapter. 
\begin{thm}[HDLSS asymptotic result of engenvalues and engenvectors.]
\label{thm2}
Under the Gaussian assumption and the one spike case of~(\ref{sec3:equ6}) and~(\ref{sec3:equ7}),\\
(i) the limit of the first sample eigenvalue $\hat{\lambda}_{Xi}^{(d)}$ (similarly for $\hat{\lambda}_{Yi}^{(d)}$) described in~(\ref{sec3:equ11}) depends on $\alpha$,
\begin{align*}
\frac{\hat{\lambda}_{X1}^{(d)}}{\emph{\text{max}}(d^{\alpha},d)} \Longrightarrow
\begin{cases}
\sigma_{X}^{2}\frac{\chi_{n}^{2}}{n}, \,\, &\alpha > 1,\\
\sigma_{X}^{2}\frac{\chi_{n}^{2}}{n}+\frac{\tau_{X}^{2}}{n}, \,\, &\alpha = 1,\\ 
\frac{\tau_{X}^{2}}{n}, \,\, &\alpha < 1,
\end{cases}
\end{align*}
as $d \to \infty$, where $\Longrightarrow$ denotes the convergence in distribution, and $\chi_{n}^{2}$ denotes the chi-square distribution with degree of freedom $n$. The rest of eigenvalues converge to the same quantities when scaled, that is, for any $\alpha \in [0,\infty), i=2,3,\dots,n$,
\begin{align*}
\frac{\hat{\lambda}_{Xi}^{(d)}}{n} \to \frac{\tau_{X}^{2}}{n}, \,\, \text{as} \,\, d \to \infty,
\end{align*}
in probability.\\
(ii) The limit of the first eigenvectors $\hat{\xi}_{Xi}^{(d)}$ (similarly for $\hat{\xi}_{Yi}^{(d)}$) described in~(\ref{sec3:equ11}) depends on $\alpha$,
\begin{align*}
\langle \hat{\xi}_{Xi}^{(d)}, \xi_{Xi}^{(d)} \rangle \Longrightarrow
\begin{cases}
1, \,\, &\alpha > 1,\\
\left( 1+\frac{\tau_{X}^{2}}{\sigma_{X}^{2}\chi_{n}^{2}} \right)^{-\frac{1}{2}}, \,\, &\alpha = 1,\\
0, \,\, &\alpha < 1,
\end{cases}
\end{align*}
as $d \to \infty$. The rest of the eigenvectors are strongly inconsistent with their population counterpart, for any $\alpha \in [0,\infty), i=2,3,\dots,n$,
\begin{align*}
\langle \hat{\xi}_{Xi}^{(d)}, \xi_{Xi}^{(d)} \rangle \to 0, \,\, \text{as} \,\, \infty,
\end{align*}
in probability.
\end{thm}

\subsection{Settings}

Let $\mathbf{X}^{(d)}$ and $\mathbf{Y}^{(d)}$ be $d \times n$ matrices that collect $d$-dimensional vector observations measured on the common set of $n$ samples. Following the represectation~(\ref{sec3:equ10}) of the two random vectors $X^{(d)}$ and $Y^{(d)}$, we, using cavariance and cross-covariance structures given in~(\ref{sec3:equ12}),~(\ref{sec3:equ8}) and~(\ref{sec3:equ9}), write $\mathbf{X}^{(d)}$ and $\mathbf{Y}^{(d)}$ as,
\begin{align}
\label{sec3:equ14}
\begin{bmatrix}
\mathbf{X}^{(d)}\\
\mathbf{Y}^{(d)}
\end{bmatrix}=
\left(\mathbf{\Sigma}_{T}^{(2d)}\right)^{\frac{1}{2}}\mathbf{Z}^{(2d)}=
\begin{bmatrix}
\mathbf{\Sigma}_{X}^{(d)} &\mathbf{\Sigma}_{XY}^{(d)}\\
\left( \mathbf{\Sigma}_{XY}^{(d)} \right)^{T} &\mathbf{\Sigma}_{Y}^{(d)}
\end{bmatrix}^{\frac{1}{2}}\mathbf{Z}^{(2d)},
\end{align}
where $\mathbf{Z}^{(2d)}$ is a $2d \times n$ matrix with columns of $\mathbf{Z}^{(2d)}$ being i.i.d. observations from an $2d$-dimensional standard normal distribution. We introduce notations for elements of the matrix $\mathbf{Z}^{(2d)}$ for a later use in the proof,
 \begin{align}
\label{sec3:equ15}
\mathbf{Z}^{(2d)}=
\begin{bmatrix}
\mathbf{Z}^{(d)}_{X}\\
\mathbf{Z}^{(d)}_{Y}
\end{bmatrix}=
\begin{bmatrix}
z_{X11} &z_{X12} &\dots &z_{X1n}\\
\vdots &\vdots &\ddots &\vdots\\
z_{Xd1} &z_{Xd2} &\dots &z_{Xdn}\\
z_{Y11} &z_{Y12} &\dots &z_{Y1n}\\
\vdots &\vdots &\ddots &\vdots\\
z_{Yd1} &z_{Yd2} &\dots &z_{Ydn}
\end{bmatrix}=
\begin{bmatrix}
z_{X1\bullet}\\
\vdots\\
z_{Xd\bullet}\\
z_{Y1\bullet}\\
\vdots\\
z_{Yd\bullet}
\end{bmatrix},
\end{align}
where $\mathbf{Z}^{(d)}_{X}$ denotes the upper half of $\mathbf{Z}^{(2d)}$ correspondint to $\mathbf{X}^{(d)}$ (similarly for $\mathbf{Z}^{(d)}_{Y}$), $z_{Xij}$, for $i=1,2,\dots,d$ and $j=1,2,\dots,n$, represents the $i$th element of the $j$ observation in $\mathbf{Z}^{(d)}_{X}$ (similarly for $z_{Yij}$), and $z_{Xk\bullet}$ denotes the $k$th row of $\mathbf{Z}^{(2d)}$ (simiarly for $z_{Yk\bullet}$. To investigate the asymptotic behavior of the sample covariance and cross-coraviance of $\mathbf{X}^{(d)}$ and $\mathbf{Y}^{(d)}$, we want to expand the matrix in~(\ref{sec3:equ14}) to get an explicit expression elementwise. This can be down either by manual or using symbolic operations in Matlab. The results, however, is too long to be included in this page, so we are going to work with a big O, small o representation for the elements that have lengthy expressions. The covariance matrix $\mathbf{\Sigma}_{T}^{(2d)}$ takes the following eigendecompsition,
\begin{align}
\label{sec3:equ16}
\mathbf{\Sigma}^{(2d)}_{T}=\sum_{i=1}^{2d} \lambda_{Ti}^{(2d)} \xi_{Ti}^{(2d)} \left ( \xi_{Ti}^{(2d)} \right )^{T}, 
\end{align}
where
\begin{align*}
&\lambda_{T1}^{(2d)}=O_{P}(d^{\alpha}), \,\, \lambda_{T2}^{(2d)}=O_{P}(d^{\alpha}), \,\, \lambda_{T4}^{(2d)}=O_{P}(1), \,\, \lambda_{T4}^{(2d)}=O_{P}(1),\\
&\lambda_{Ti}^{(2d)}=\tau_{X}^{2}, \,\, i=5,6,\dots,d+2\\
&\lambda_{Tj}^{(2d)}=\tau_{Y}^{2}, \,\, j=d+3,d+4,\dots,2d,\\
&\xi_{T1}^{(2d)}=
\begin{bmatrix}
O_{P}(1) &O_{P}(\frac{1}{\sqrt{d^{\alpha}}}) &O_{P}(1) &\dots &O_{P}(1) &O_{P}(1) &O_{P}(\frac{1}{\sqrt{d^{\alpha}}}) &O_{P}(1) &\dots &O_{P}(1)
\end{bmatrix}^{T},\\
&\xi_{T2}^{(2d)}=
\begin{bmatrix}
O_{P}(1) &O_{P}(\frac{1}{\sqrt{d^{\alpha}}}) &O_{P}(1) &\dots &O_{P}(1) &O_{P}(1) &O_{P}(\frac{1}{\sqrt{d^{\alpha}}}) &O_{P}(1) &\dots &O_{P}(1)
\end{bmatrix}^{T},\\
&\xi_{T3}^{(2d)}=
\begin{bmatrix}
O_{P}(1) &O_{P}(1) &\dots &O_{P}(1) &O_{P}(1)
\end{bmatrix}^{T},\\
&\xi_{T4}^{(2d)}=
\begin{bmatrix}
O_{P}(1) &O_{P}(1) &\dots &O_{P}(1) &O_{P}(1)
\end{bmatrix}^{T},\\
&\xi_{Ti}^{(2d)}=e_{i}^{(2d)}, \,\, i=5,6,\dots,d+2,\\
&\xi_{Tj}^{(2d)}=e_{j}^{(2d)},  \,\, j=d+3,d+4,\dots,2d.
\end{align*}
Then, using~(\ref{sec3:equ14}),~(\ref{sec3:equ15}) and~(\ref{sec3:equ16}), the data matrix $\mathbf{X}^{(d)}$ and $\mathbf{Y}^{(d)}$ can be expressed as, 
\begin{align}
\label{sec3:equ17}
\begin{bmatrix}
\mathbf{X}^{(d)}\\
\mathbf{Y}^{(d)}
\end{bmatrix}=
\begin{bmatrix}
a_{1}z_{X1\bullet}+a_{2}z_{X2\bullet}+a_{3}z_{Y1\bullet}+a_{4}z_{Y2\bullet}\\
a_{5}z_{X1\bullet}+a_{6}z_{X2\bullet}+a_{7}z_{Y1\bullet}+a_{8}(1)z_{Y2\bullet}\\
\tau_{X}z_{X3\bullet}\\
\vdots\\
\tau_{X}z_{Xd\bullet}\\
b_{1}z_{X1\bullet}+b_{2}z_{X2\bullet}+b_{3}z_{Y1\bullet}+b_{4}z_{Y2\bullet}\\
b_{5}z_{X1\bullet}+b_{6}z_{X2\bullet}+b_{7}z_{Y1\bullet}+b_{8}z_{Y2\bullet}\\
\tau_{Y}z_{Y3\bullet}\\
\vdots\\
\tau_{Y}z_{Yd\bullet}
\end{bmatrix},
\end{align}
where $a_{1}, a_{3}, b_{1}$ and $b_{3}$ are random variable of magnitude of $O_{P}(\sqrt{d^{\alpha}})$, and the rest of $a_{i}$ and $b_{j}$ are of magnitude of $O_{P}(1)$. Let $c_{1}, c_{2}, d{1}$ and $d{2}$ denote the first, second, $(d+1)$th and $(d+2)$th row of the matrix~(\ref{sec3:equ17}), respectively. 
\begin{align}
\label{sec3:equ18}
\begin{aligned}
&c_{1}=a_{1}z_{X1\bullet}+a_{2}z_{X2\bullet}+a_{3}z_{Y1\bullet}+a_{4}z_{Y2\bullet},
\quad c_{2}=a_{5}z_{X1\bullet}+a_{6}z_{X2\bullet}+a_{7}z_{Y1\bullet}+a_{8}z_{Y2\bullet},\\
&d_{1}=b_{1}z_{X1\bullet}+b_{2}z_{X2\bullet}+b_{3}z_{Y1\bullet}+b_{4}z_{Y2\bullet},
\quad d_{2}=b_{5}z_{X1\bullet}+b_{6}z_{X2\bullet}+b_{7}z_{Y1\bullet}+b_{8}z_{Y2\bullet}.
\end{aligned}
\end{align}
The sample covariance and cross-covariance matrices $\hat{\mathbf{\Sigma}}^{(d)}_{X}$, $\hat{\mathbf{\Sigma}}^{(d)}_{Y}$ and $\hat{\mathbf{\Sigma}}^{(d)}_{XY}$ are found as blocks of the following matrix,
\begin{align}
\label{sec3:equ34}
\begin{bmatrix}
\hat{n\mathbf{\Sigma}}_{X}^{(d)} &\hat{n\mathbf{\Sigma}}_{XY}^{(d)}\\
\left(n\hat{\mathbf{\Sigma}}_{XY}^{(d)} \right)^{T} &n\hat{\mathbf{\Sigma}}_{Y}^{(d)}
\end{bmatrix}=
\begin{bmatrix}
\mathbf{X}^{(d)}\\
\mathbf{Y}^{(d)}
\end{bmatrix}
\begin{bmatrix}
\mathbf{X}^{(d)}\\
\mathbf{Y}^{(d)}
\end{bmatrix}^{T},
\end{align}
where elementwise-explicit matrices are given as,
\begin{align*}
&\hat{\mathbf{\Sigma}}^{(d)}_{X} =\\
&\frac{1}{n}\begin{bmatrix}
\langle c_{1}, c_{1} \rangle &\langle c_{1}, c_{2} \rangle &\tau_{X}\langle c_{1}, z_{X3\bullet} \rangle &\tau_{X}\langle c_{1}, z_{X4\bullet} \rangle &\ldots &\tau_{X}\langle c_{1}, z_{Xd\bullet} \rangle\\
\langle c_{2}, c_{1} \rangle &\langle c_{2}, c_{2} \rangle &\tau_{X}\langle c_{2}, z_{X3\bullet} \rangle &\tau_{X}\langle c_{2}, z_{X4\bullet} \rangle &\ldots &\tau_{X}\langle c_{2}, z_{Xd\bullet} \rangle\\
\tau_{X}\langle z_{X3\bullet}, c_{1} \rangle &\tau_{X}\langle z_{X3\bullet}, c_{2} \rangle &\tau_{X}^{2}\langle z_{X3\bullet}, z_{X3\bullet} \rangle &\tau_{X}^{2}\langle z_{X3\bullet}, z_{4,\bullet} \rangle &\ldots &\tau_{X}^{2}\langle z_{X3\bullet}, z_{Xd\bullet} \rangle\\
\tau_{X}\langle z_{X4\bullet}, c_{1} \rangle &\tau_{X}\langle z_{X4\bullet}, c_{2} \rangle &\tau_{X}^{2}\langle z_{X4\bullet}, z_{X3\bullet} \rangle &\tau_{X}^{2}\langle z_{X4\bullet}, z_{X4\bullet} \rangle &\ldots &\tau_{X}^{2}\langle z_{X4\bullet}, z_{Xd\bullet} \rangle\\
\vdots &\vdots &\vdots &\vdots &\ddots &\vdots\\
\tau_{X}\langle z_{Xd\bullet}, c_{1} \rangle &\tau_{X}\langle z_{Xd\bullet}, c_{2} \rangle &\tau_{X}^{2}\langle z_{Xd\bullet}, z_{3,\bullet} \rangle &\tau_{X}^{2}\langle z_{Xd\bullet}, z_{4,\bullet} \rangle &\ldots &\tau_{X}^{2}\langle z_{Xd\bullet}, z_{Xd\bullet} \rangle
\end{bmatrix},
\end{align*}
\begin{align}
\label{sec3:equ19}
\begin{aligned}
&\hat{\mathbf{\Sigma}}^{(d)}_{Y} =\\
&\frac{1}{n}\begin{bmatrix}
\langle d_{1}, d_{1} \rangle &\langle d_{1}, d_{2} \rangle &\tau_{Y}\langle d_{1}, z_{Y3\bullet} \rangle &\tau_{Y}\langle d_{1}, z_{Y4\bullet} \rangle &\ldots &\tau_{Y}\langle d_{1}, z_{Yd\bullet} \rangle\\
\langle d_{2}, d_{1} \rangle &\langle d_{2}, d_{2} \rangle &\tau_{Y}\langle d_{2}, z_{Y3\bullet} \rangle &\tau_{Y}\langle d_{2}, z_{Y4\bullet} \rangle &\ldots &\tau_{Y}\langle d_{2}, z_{Yd\bullet} \rangle\\
\tau_{Y}\langle z_{Y3\bullet}, d_{1} \rangle &\tau_{Y}\langle z_{Y3\bullet}, d_{2} \rangle &\tau_{Y}^{2}\langle z_{Y3\bullet}, z_{Y3\bullet} \rangle &\tau_{Y}^{2}\langle z_{Y3\bullet}, z_{Y4\bullet} \rangle &\ldots &\tau_{Y}^{2}\langle z_{Y3\bullet}, z_{Yd\bullet} \rangle\\
\tau_{Y}\langle z_{Y4\bullet}, d_{1} \rangle &\tau_{Y}\langle z_{Y4\bullet}, d_{2} \rangle &\tau_{Y}^{2}\langle z_{Y4\bullet}, z_{Y3\bullet} \rangle &\tau_{Y}^{2}\langle z_{Y4\bullet}, z_{Y4\bullet} \rangle &\ldots &\tau_{Y}^{2}\langle z_{Y4\bullet}, z_{Yd\bullet} \rangle\\
\vdots &\vdots &\vdots &\vdots &\ddots &\vdots\\
\tau_{Y}\langle z_{Yd\bullet}, d_{1} \rangle &\tau_{Y}\langle z_{Yd\bullet}, d_{2} \rangle &\tau_{Y}^{2}\langle z_{Yd\bullet}, z_{Y3\bullet} \rangle &\tau_{Y}^{2}\langle z_{Yd\bullet}, z_{Y4\bullet} \rangle &\ldots &\tau_{Y}^{2}\langle z_{Yd\bullet}, z_{Yd\bullet} \rangle
\end{bmatrix},
\end{aligned}
\end{align}
\begin{align*}
&\hat{\mathbf{\Sigma}}^{(d)}_{XY} =\\
&\frac{1}{n}\begin{bmatrix}
\langle c_{1}, d_{1} \rangle &\langle c_{1}, d_{2} \rangle &\tau_{Y}\langle c_{1}, z_{Y3\bullet} \rangle &\tau_{Y}\langle c_{1}, z_{Y4\bullet} \rangle &\ldots &\tau_{Y}\langle c_{1}, z_{Yd\bullet} \rangle\\
\langle c_{2}, d_{1} \rangle &\langle c_{2}, d_{2} \rangle &\tau_{Y}\langle c_{2}, z_{Y3\bullet} \rangle &\tau_{Y}\langle c_{2}, z_{Y4\bullet} \rangle &\ldots &\tau_{Y}\langle c_{2}, z_{Yd\bullet} \rangle\\
\tau_{X}\langle z_{X3\bullet}, d_{1} \rangle &\tau_{X}\langle z_{X3\bullet}, d_{2} \rangle &\tau_{X}\tau_{Y}\langle z_{X3\bullet}, z_{Y3\bullet} \rangle &\tau_{X}\tau_{Y}\langle z_{X3\bullet}, z_{Y4\bullet} \rangle &\ldots &\tau_{X}\tau_{Y}\langle z_{X3\bullet}, z_{Yd\bullet} \rangle\\
\tau_{X}\langle z_{X4\bullet}, d_{1} \rangle &\tau_{X}\langle z_{X4\bullet}, d_{2} \rangle &\tau_{X}\tau_{Y}\langle z_{X4\bullet}, z_{Y3\bullet} \rangle &\tau_{X}\tau_{Y}\langle z_{X4\bullet}, z_{Y4\bullet} \rangle &\ldots &\tau_{X}\tau_{Y}\langle z_{X4\bullet}, z_{Yd\bullet} \rangle\\
\vdots &\vdots &\vdots &\vdots &\ddots &\vdots\\
\tau_{X}\langle z_{Xd\bullet}, d_{1} \rangle &\tau_{X}\langle z_{Xd\bullet}, d_{2} \rangle &\tau_{X}\tau_{Y}\langle z_{Xd\bullet}, z_{Y3\bullet} \rangle &\tau_{X}\tau_{Y}\langle z_{Xd\bullet}, z_{Y4\bullet} \rangle &\ldots &\tau_{X}\tau_{Y}\langle z_{Xd\bullet}, z_{Yd\bullet} \rangle
\end{bmatrix}.
\end{align*}

\subsection{Notation}

Here, we introduce notations to be used in the proof for HDLSS asymptotic behaviors of a random variable. 

\begin{itemize}
\item For a $d$-dimensional random variable $X^{(d)}$, we say that $X^{(d)}=o_{P}(d^{\alpha})$, for $\alpha \in R$, if, $\forall \epsilon > 0$,
\begin{align*}
\lim_{d\to\infty} P\left( \left| \frac{X^{(d)}}{d^{\alpha}} > \epsilon \right| \right) = 0.
\end{align*}
\item For a $d$-dimensional random variable $X^{(d)}$, we say that $X^{(d)}=O_{P}(d^{\alpha})$, for $\alpha \in R$, if, $\forall \epsilon > 0$, there exists a finite $C > 0$ such that,
\begin{align*}
P\left( \left| \frac{X^{(d)}}{d^{\alpha}} > M \right| \right) < \epsilon, \,\, \forall d.
\end{align*}
\item For a $d$-dimensional random variable $X^{(d)}$, we say that $X^{(d)} \asymp d^{\alpha}$, for $\alpha \in R$, if,
\begin{align*}
X^{(d)}=O_{P}(d^{\alpha}) \,\, \text{and} \,\, X^{(d)} \not= o_{P}(d^{\alpha}).
\end{align*}
\end{itemize}

\subsection{Case of \texorpdfstring{$\alpha > 1$}{}}

First, we prove Theorem~\ref{thm1} under the condition of $\alpha > 1$ with the spiked model of the population covariance structure of $X^{(d)}$ and $Y^{(d)}$ decribed in~(\ref{sec3:equ8}).

\subsubsection{Behavior of the sample cross-covariance matrix}

We now investigate the HDLSS asymptotic behavior of the sample cross-covariance matrix $\hat{\mathbf{\Sigma}}^{(d)}_{XY}$ given in~(\ref{sec3:equ19}), in specific, its sample singular values and singular vectors. The singular value decomposition (SVD) of $\hat{\mathbf{\Sigma}}^{(d)}_{XY}$ gives,
\begin{align}
\label{sec3:equ20}
\hat{\mathbf{\Sigma}}^{(d)}_{XY} =  \sum_{i=1}^{n} \hat{\lambda}_{XYi}^{(d)} \hat{\eta}_{Xi}^{(d)} \left( \hat{\eta}_{Yi}^{(d)} \right)^{T},
\end{align}
where $\hat{\lambda}_{XY1}^{(d)} \ge \hat{\lambda}_{XY2}^{(d)} \ge \dots \ge \hat{\lambda}_{XYn}^{(d)} \ge 0$, $\| \hat{\eta}_{Xi}^{(d)} \|_{2} = \| \hat{\eta}_{Yi}^{(d)} \|_{2} = 1$ for $i=1,2,\dots,n$, and $\langle \hat{\eta}_{Xi}^{(d)}, \hat{\eta}_{Xj}^{(d)} \rangle = \langle \hat{\eta}_{Yi}^{(d)}, \hat{\eta}_{Yj}^{(d)} \rangle = 0$, for $i \not= j$.
\begin{lem}[CCA HDLSS Asymptotic lemma 1.]
\label{lem2}
Let $C_{XY}$ and $\mathbf{M}_{XY}$ be,
\begin{align*}
C_{XY}=\lim_{d\to\infty}\frac{\langle c_{1}, d_{1} \rangle}{d^{\alpha}}, \,\, \mathbf{M}_{XY}=
\begin{bmatrix}
C_{XY} &\underset{1 \times (d-1)}{\mathbf{0}}\\
\underset{(d-1) \times 1}{\mathbf{0}} &\underset{(d-1) \times (d-1)}{\mathbf{0}}\\
\end{bmatrix},
\end{align*}
where $c_{1}$ and $d_{1}$ are defined in~(\ref{sec3:equ18}) and so $C_{XY}$ is a non-degenerate random variable. Then, for $\alpha > 1$,
\begin{align*}
\left\| \frac{n\hat{\mathbf{\Sigma}}^{(d)}_{XY}}{d^{\alpha}}-\mathbf{M}_{XY} \right\|_{F}^{2} \mathop{\longrightarrow}_{d \to \infty}^{p} 0,
\end{align*}
where $\| \bullet \|_{F}$ is the Frobenious norm of a matrix.
\end{lem}
\begin{proof}
Let $\mathbf{M}_{XY}^{(d)}$ be,
\begin{align*} 
\mathbf{M}_{XY}^{(d)}=
\begin{bmatrix}
\frac{\langle c_{1}, d_{1} \rangle}{d^{\alpha}} &\underset{1 \times (d-1)}{\mathbf{0}}\\
\underset{(d-1) \times 1}{\mathbf{0}} &\underset{(d-1) \times (d-1)}{\mathbf{0}}\\
\end{bmatrix}.
\end{align*}
It is obvious to see that,
\begin{align*}
\left\| \mathbf{M}_{XY}^{(d)}-\mathbf{M}_{XY} \right\|_{F}^{2} \mathop{\longrightarrow}_{d \to \infty}^{p} 0.
\end{align*}
Using the Cauchy-Schwarz inequality,
\begin{align*}
\left\| \frac{n\hat{\mathbf{\Sigma}}^{(d)}_{XY}}{d^{\alpha}}-\mathbf{M}_{XY}^{(d)} \right\|_{F}^{2} 
=&\frac{\langle c_{1}, d_{2} \rangle^{2}}{d^{2\alpha}}
+\frac{\langle c_{2}, d_{1} \rangle^{2}}{d^{2\alpha}}
+\frac{\langle c_{2}, d_{2} \rangle^{2}}{d^{2\alpha}}
+\tau_{Y}^{2}\sum_{i=3}^{d} \frac{\langle c_{1}, z_{Yi\bullet} \rangle^{2}}{d^{2\alpha}}\\
&+\tau_{Y}^{2}\sum_{i=3}^{d} \frac{\langle c_{2}, z_{Yi\bullet} \rangle^{2}}{d^{2\alpha}}
+\tau_{X}^{2}\sum_{i=3}^{d} \frac{\langle z_{Xi\bullet}, d_{1} \rangle^{2}}{d^{2\alpha}}\\
&+\tau_{X}^{2}\sum_{i=3}^{d} \frac{\langle z_{Xi\bullet}, d_{2} \rangle^{2}}{d^{2\alpha}}
+\tau_{X}^{2}\tau_{Y}^{2}\sum_{i=3}^{d}\sum_{j=3}^{d} \frac{\langle z_{Xi\bullet}, z_{Yj\bullet} \rangle^{2}}{d^{2\alpha}}\\
\le &\frac{\langle c_{1}, d_{2} \rangle^{2}}{d^{2\alpha}}
+\frac{\langle c_{2}, d_{1} \rangle^{2}}{d^{2\alpha}}
+\frac{\langle c_{2}, d_{2} \rangle^{2}}{d^{2\alpha}}
+\tau_{Y}^{2}\sum_{i=3}^{d} \frac{\| c_{1} \|_{2}^{2} \| z_{Yi\bullet} \|_{2}^{2}}{d^{2\alpha}}\\
&+\tau_{Y}^{2}\sum_{i=3}^{d} \frac{\| c_{2} \|_{2}^{2} \| z_{Yi\bullet} \|_{2}^{2}}{d^{2\alpha}}
+\tau_{X}^{2}\sum_{i=3}^{d} \frac{\| z_{Xi\bullet} \|_{2}^{2} \| d_{1} \|_{2}^{2}}{d^{2\alpha}}\\
&+\tau_{X}^{2}\sum_{i=3}^{d} \frac{\| z_{Xi\bullet}\|_{2}^{2} \|d_{2} \|_{2}^{2}}{d^{2\alpha}}
+\tau_{X}^{2}\tau_{Y}^{2}\sum_{i=3}^{d}\sum_{j=3}^{d} \frac{\| z_{Xi\bullet}\|_{2}^{2} \| z_{Yj\bullet} \|_{2}^{2}}{d^{2\alpha}}.
\end{align*}
Since $\langle c_{1}, d_{2} \rangle^{2}$, $\langle c_{2}, d_{1} \rangle^{2}$ are $O_{P}(\sqrt{d^{\alpha}})$, and $\langle c_{2}, d_{2} \rangle^{2}$ is $O_{P}(1)$,
\begin{align*}
\frac{\langle c_{1}, d_{2} \rangle^{2}}{d^{2\alpha}} \mathop{\longrightarrow}_{d \to \infty}^{p} 0, \,\, \frac{\langle c_{2}, d_{1} \rangle^{2}}{d^{2\alpha}} \mathop{\longrightarrow}_{d \to \infty}^{p} 0, \,\, \frac{\langle c_{2}, d_{2} \rangle^{2}}{d^{2\alpha}} \mathop{\longrightarrow}_{d \to \infty}^{p} 0.
\end{align*}
It is not hard to see that $\| z_{Xi\bullet} \|_{2}^{2} \sim \chi_{n}^{2}$ and $\| z_{Yi\bullet} \|_{2}^{2} \sim \chi_{n}^{2}$, since the elements of $z_{Xi\bullet}$ (respectively $z_{Yi\bullet}$) are the $i$th components of the i.i.d. multivariate standard normal samples of size $n$. Note that the magnitudes of $\| c_{1} \|_{2}^{2}, \| c_{2} \|_{2}^{2}, \| d_{1} \|_{2}^{2}$ and $\| d_{2} \|_{2}^{2}$ are of $O_{P}(d^{\alpha})$. Applying the law of large numbers, we have,
\begin{align*}
\tau_{Y}^{2}\sum_{i=3}^{d} \frac{\| c_{1} \|_{2}^{2} \| z_{Yi\bullet} \|_{2}^{2}}{d^{2\alpha}}
&=\frac{\tau_{Y}^{2}}{d^{\alpha-1}} \left \| \frac{c_{1}}{\sqrt{d^{\alpha}}} \right \|_{2}^{2} \sum_{i=3}^{d} \frac{\| z_{Yi\bullet} \|_{2}^{2}}{d}\\
&\mathop{\longrightarrow}_{d \to \infty}^{p} 0 \times O_{P}(1) \times E(\chi_{n}^{2}) = 0,\\
\tau_{Y}^{2}\sum_{i=3}^{d} \frac{\| c_{2} \|_{2}^{2} \| z_{Yi\bullet} \|_{2}^{2}}{d^{2\alpha}}
&=\frac{\tau_{Y}^{2}}{d^{\alpha-1}} \left \| \frac{c_{2}}{\sqrt{d^{\alpha}}} \right \|_{2}^{2} \sum_{i=3}^{d} \frac{\| z_{Yi\bullet} \|_{2}^{2}}{d}\\
&\mathop{\longrightarrow}_{d \to \infty}^{p} 0 \times 0 \times E(\chi_{n}^{2}) = 0,\\
\tau_{X}^{2}\sum_{i=3}^{d} \frac{\| d_{1} \|_{2}^{2} \| z_{Xi\bullet} \|_{2}^{2}}{d^{2\alpha}}
&=\frac{\tau_{X}^{2}}{d^{\alpha-1}} \left \| \frac{d_{1}}{\sqrt{d^{\alpha}}} \right \|_{2}^{2} \sum_{i=3}^{d} \frac{\| z_{Xi\bullet} \|_{2}^{2}}{d}\\
&\mathop{\longrightarrow}_{d \to \infty}^{p} 0 \times O_{P}(1) \times E(\chi_{n}^{2}) = 0,\\
\tau_{X}^{2}\sum_{i=3}^{d} \frac{\| d_{2} \|_{2}^{2} \| z_{Xi\bullet} \|_{2}^{2}}{d^{2\alpha}}
&=\frac{\tau_{X}^{2}}{d^{\alpha-1}} \left \| \frac{d_{2}}{\sqrt{d^{\alpha}}} \right \|_{2}^{2} \sum_{i=3}^{d} \frac{\| z_{Xi\bullet} \|_{2}^{2}}{d}\\
&\mathop{\longrightarrow}_{d \to \infty}^{p} 0 \times 0 \times E(\chi_{n}^{2}) = 0.
\end{align*}
Using $\alpha > 1$ and the law of large numbers,
\begin{align*}
\tau_{X}^{2}\tau_{Y}^{2}\sum_{i=3}^{d}\sum_{j=3}^{d} \frac{\| z_{Xi\bullet}\|_{2}^{2} \| z_{Yj\bullet} \|_{2}^{2}}{d^{2\alpha}}&=\frac{\tau_{X}^{2}\tau_{Y}^{2}}{d^{2\alpha-2}} \sum_{i=3}^{d} \frac{\| z_{Xi\bullet} \|_{2}^{2}}{d} \sum_{j=3}^{d} \frac{\| z_{Yj\bullet} \|_{2}^{2}}{d}\\
&\mathop{\longrightarrow}_{d \to \infty}^{p} 0 \times E(\chi_{n}^{2}) \times E(\chi_{n}^{2}) = 0.
\end{align*}
Therefore,
\begin{align*}
\left\| \frac{n\hat{\mathbf{\Sigma}}^{(d)}_{XY}}{d^{\alpha}}-\mathbf{M}_{XY}^{(d)} \right\|_{F}^{2} \mathop{\longrightarrow}_{d \to \infty}^{p} 0.
\end{align*}
\end{proof}
\begin{lem}[CCA HDLSS Asymptotic lemma 2.]
\label{lem3}
Let $\mathbf{M}_{XY}$ be the matrix defined in Lemma~\ref{lem2}. Let $\hat{\lambda}_{XY1}^{(d)}$, $\hat{\eta}_{X1}^{(d)}$ and $\hat{\eta}_{Y1}^{(d)}$ be the first sample singular value and singular vectors from~(\ref{sec3:equ20}). Then, for $\alpha > 1$, 
\begin{align*}
\left\| \frac{n\hat{\lambda}_{XY1}^{(d)}}{d^{\alpha}} \hat{\eta}_{X1}^{(d)} \left( \hat{\eta}_{Y1}^{(d)} \right)^{T}-\mathbf{M}_{XY} \right\|_{F}^{2} \mathop{\longrightarrow}_{d \to \infty}^{p} 0.
\end{align*}
\end{lem}
\begin{proof}
Using $\mathbf{M}_{XY}^{(d)}$ defined in Lemma~\ref{lem2} and the triangle inequality,
\begin{align*}
&\left \| \frac{n\hat{\lambda}_{XY1}^{(d)}}{d^{\alpha}} \hat{\eta}_{X1}^{(d)} \left( \hat{\eta}_{Y1}^{(d)} \right)^{T} - \mathbf{M}_{XY}^{(d)} \right \|_{F}\\
&\quad \quad \quad =\left \| \left( \frac{n\hat{\mathbf{\Sigma}}^{(d)}_{XY}}{d^{\alpha}} - \mathbf{M}_{XY}^{(d)} \right) - \left( \frac{n\hat{\mathbf{\Sigma}}^{(d)}_{XY}}{d^{\alpha}} - \frac{\hat{\lambda}_{XY1}^{(d)}}{d^{\alpha}} \hat{\eta}_{X1}^{(d)} \left( \hat{\eta}_{Y1}^{(d)} \right)^{T} \right) \right \|_{F}\\
&\quad \quad \quad \le \left \| \frac{n\hat{\mathbf{\Sigma}}^{(d)}_{XY}}{d^{\alpha}} - \mathbf{M}_{XY}^{(d)} \right \|_{F} + \left \| \frac{n\hat{\mathbf{\Sigma}}^{(d)}_{XY}}{d^{\alpha}} - \frac{\hat{\lambda}_{XY1}^{(d)}}{d^{\alpha}} \hat{\eta}_{X1}^{(d)} \left( \hat{\eta}_{Y1}^{(d)} \right)^{T} \right \|_{F}.
\end{align*}
By Lemma~\ref{lem2},
\begin{align*}
\left \| \frac{n\hat{\mathbf{\Sigma}}^{(d)}_{XY}}{d^{\alpha}} - \mathbf{M}_{XY}^{(d)} \right \|_{F}^{2} \mathop{\longrightarrow}_{d \to \infty}^{p} 0.
\end{align*}
Write $\mathbf{M}_{XY}^{(d)}$ as,
\begin{align*}
\mathbf{M}_{XY}^{(d)} = \frac{\langle c_{1}, d_{1} \rangle}{d^{\alpha}} e_{1}^{(d)} \left( e_{1}^{(d)} \right)^{T}.
\end{align*}
Since the first sample singular value $\hat{\lambda}_{XY1}^{(d)}$ and singular vectors $\hat{\eta}_{X1}^{(d)}$ and $\hat{\eta}_{Y1}^{(d)}$ provide the best rank-1 approximation to $\hat{\lambda}_{XY1}^{(d)}/d^{\alpha}$,
\begin{align*}
\left\| \frac{\hat{\mathbf{\Sigma}}^{(d)}_{XY}}{d^{\alpha}} - \frac{\hat{\lambda}_{XY1}^{(d)}}{d^{\alpha}} \hat{\eta}_{X1}^{(d)} \left( \hat{\eta}_{Y1}^{(d)} \right)^{T} \right\|_{F}^{2}
&\le \left \| \frac{\hat{\mathbf{\Sigma}}^{(d)}_{XY}}{d^{\alpha}} - \frac{\langle c_{1}, d_{1} \rangle}{nd^{\alpha}} e_{1}^{(d)} \left( e_{1}^{(d)} \right)^{T} \right \|_{F}^{2}\\
&= \left \| \frac{\hat{\mathbf{\Sigma}}^{(d)}_{XY}}{d^{\alpha}} - \frac{\mathbf{M}_{XY}^{(d)}}{n} \right \|_{F}^{2}\\ 
&\mathop{\longrightarrow}_{d \to \infty}^{p} 0.
\end{align*}
\end{proof}
\begin{lem}[CCA HDLSS Asymptotic lemma 3.]
\label{lem4}
For $\alpha > 1$, the first singular value $\hat{\lambda}_{XY1}^{(d)}$ and singular vectors $\hat{\eta}_{X1}^{(d)}$ and $\hat{\eta}_{Y1}^{(d)}$ converge in probability to the following quantities as $d \to \infty$,
\begin{align*}
\frac{n\hat{\lambda}_{XY1}^{(d)}}{d^{\alpha}} \mathop{\longrightarrow}_{d \to \infty}^{p} C_{XY}, \,\, \left\| \hat{\eta}_{X1}^{(d)}-e_{1}^{(d)} \right\|_{2}^{2} \mathop{\longrightarrow}_{d \to \infty}^{p} 0, \,\, \left\| \hat{\eta}_{Y1}^{(d)}-e_{1}^{(d)} \right\|_{2}^{2} \mathop{\longrightarrow}_{d \to \infty}^{p} 0,
\end{align*}
where $C_{XY}$ is defined in Lemma~\ref{lem2}.
\end{lem}
\begin{proof}
Let $\tilde{\eta}_{X1}^{(d)}$ and $\tilde{\eta}_{Y1}^{(d)}$ be $\hat{\eta}_{X1}^{(d)}$ and $\hat{\eta}_{Y1}^{(d)}$ where their first entries are set to 0. Let $\hat{\eta}_{X1}^{(d)}(i)$ and $\hat{\eta}_{Y1}^{(d)}(j)$ be the $i$th and $j$th entries of $\hat{\eta}_{X1}^{(d)}$ and $\hat{\eta}_{Y1}^{(d)}$. By Lemma~\ref{lem3},
\begin{align*}
\left( \frac{\hat{\lambda}_{XY1}^{(d)}}{d^{\alpha}} \right)^{2} \left\| \tilde{\eta}_{X1}^{(d)} \right\|_{2}^{2} \left\| \tilde{\eta}_{Y1}^{(d)} \right\|_{2}^{2}
&=\left( \frac{\hat{\lambda}_{XY1}^{(d)}}{d^{\alpha}} \right)^{2} \sum_{i=2}^{d} \left( \hat{\eta}_{X1}^{(d)}(i) \right)^{2} \sum_{j=2}^{d} \left( \hat{\eta}_{Y1}^{(d)}(j) \right)^{2}\\
&=\left( \frac{\hat{\lambda}_{XY1}^{(d)}}{d^{\alpha}} \right)^{2} \sum_{i=2}^{d} \sum_{j=2}^{d}  \left( \hat{\eta}_{X1}^{(d)}(i) \right)^{2} \left( \hat{\eta}_{Y1}^{(d)}(j) \right)^{2}\\
&\mathop{\longrightarrow}_{d \to \infty}^{p} 0.
\end{align*}
First, we show that $(\hat{\lambda}_{XY1}^{(d)}/d^{\alpha})^{2} > 0$ in probability. Suppose that $(\hat{\lambda}_{XY1}^{(d)}/d^{\alpha})^{2}$ converges in probability to 0. Then, noting that $\| \hat{\eta}_{X1}^{(d)} \|_{2}^{2} = 1$ and $\| \hat{\eta}_{Y1}^{(d)} \|_{2}^{2} = 1$,
\begin{align*}
\left\| \frac{\hat{\lambda}_{XY1}^{(d)}}{d^{\alpha}} \hat{\eta}_{X1}^{(d)} \left( \hat{\eta}_{Y1}^{(d)} \right)^{T} \right\|_{F}^{2} &= \left( \frac{\hat{\lambda}_{XY1}^{(d)}}{d^{\alpha}} \right)^{2} \sum_{i=1}^{d} \sum_{j=1}^{d}  \left( \hat{\eta}_{X1}^{(d)}(i) \right)^{2} \left( \hat{\eta}_{Y1}^{(d)}(j) \right)^{2}\\
&=\left( \frac{\hat{\lambda}_{XY1}^{(d)}}{d^{\alpha}} \right)^{2} \sum_{i=1}^{d} \left( \hat{\eta}_{X1}^{(d)}(i) \right)^{2} \sum_{j=1}^{d} \left( \hat{\eta}_{Y1}^{(d)}(j) \right)^{2}\\
&=\left( \frac{\hat{\lambda}_{XY1}^{(d)}}{d^{\alpha}} \right)^{2} \left\| \hat{\eta}_{X1}^{(d)} \right\|_{2}^{2} \left\| \hat{\eta}_{Y1}^{(d)} \right\|_{2}^{2}\\
&=\left( \frac{\hat{\lambda}_{XY1}^{(d)}}{d^{\alpha}} \right)^{2}\\
&\mathop{\longrightarrow}_{d \to \infty}^{p} 0,
\end{align*}
which contradicts to Lemma~\ref{lem3} (note that the limiting matrix $M$ is not a degenerate matrix). Therefore, 
\begin{align*}
\left\| \tilde{\eta}_{X1}^{(d)} \right\|_{2}^{2} \mathop{\longrightarrow}_{d \to \infty}^{p} 0, \,\, \text{or} \,\, \left\| \hat{\eta}_{Y1}^{(d)} \right\|_{2}^{2} \mathop{\longrightarrow}_{d \to \infty}^{p} 0.
\end{align*}
We want to show that both $\| \tilde{\eta}_{X1}^{(d)} \|_{2}^{2}$ and $\| \tilde{\eta}_{Y1}^{(d)} \|_{2}^{2}$ converge to 0 in probability. Suppose that $\| \tilde{\eta}_{X1}^{(d)} \|_{2}^{2}$ converges to 0 and $\| \tilde{\eta}_{Y1}^{(d)} \|_{2}^{2} > 0$ in probability. Since the norm of $\hat{\eta}_{X1}^{(d)}$ is 1, its first entry $\hat{\eta}_{X1}^{(d)}(1)$ converges in probability to, 
\begin{align*}
\left( \hat{\eta}_{X1}^{(d)}(1) \right)^{2} = 1 - \left\| \tilde{\eta}_{X1}^{(d)} \right\|_{2}^{2} \mathop{\longrightarrow}_{d \to \infty}^{p} 1.
\end{align*}
Then, the squared sum of the entries in the first row of $(\hat{\lambda}_{XY1}^{(d)}/d^{\alpha}) \hat{\eta}_{X1}^{(d)} \left( \hat{\eta}_{Y1}^{(d)} \right)^{T}$ with the first entry excluded becomes,
\begin{align*}
\left\| \frac{\hat{\lambda}_{XY1}^{(d)}}{d^{\alpha}} \hat{\eta}_{X1}^{(d)}(1) \tilde{\eta}_{Y1}^{(d)} \right\|_{2}^{2} &= \left( \frac{\hat{\lambda}_{XY1}^{(d)}}{d^{\alpha}} \right)^{2} \left( \hat{\eta}_{X1}^{(d)}(1) \right)^{2} \sum_{j=2}^{d} \left( \hat{\eta}_{Y1}^{(d)}(j) \right)^{2}
\mathop{\longrightarrow}_{d \to \infty}^{p} > 0,
\end{align*}
which contradicts to Lemma~\ref{lem3}, according to which the above quantity should converge to 0. Argument is similar for the case where $\| \tilde{\eta}_{Y1}^{(d)} \|_{2}^{2}$ converges to 0 and $\| \tilde{\eta}_{X1}^{(d)} \|_{2}^{2} > 0$ in probability. Hence we have
\begin{align*}
\left\| \tilde{\eta}_{X1}^{(d)} \right\|_{2}^{2} \mathop{\longrightarrow}_{d \to \infty}^{p}  0, \,\, \text{and} \,\, \left\| \hat{\eta}_{Y1}^{(d)} \right\|_{2}^{2} \mathop{\longrightarrow}_{d \to \infty}^{p} 0.
\end{align*}
Note that, since the norm of $\hat{\eta}_{Y1}^{(d)}$ is 1,
\begin{align*}
\left( \hat{\eta}_{Y1}^{(d)}(1) \right)^{2} = 1 - \left\| \tilde{\eta}_{Y1}^{(d)} \right\|_{2}^{2} \mathop{\longrightarrow}_{d \to \infty}^{p} 1.
\end{align*}
Therefore
\begin{align*}
\left\| \hat{\eta}_{X1}^{(d)} - e_{1}^{(d)} \right\|_{2}^{2} &= \left( \hat{\eta}_{X1}^{(d)}(1) - 1 \right)^{2} + \left\| \tilde{\eta}_{X1}^{(d)}(i) \right\|_{2}^{2}
\mathop{\longrightarrow}_{d \to \infty}^{p} 0,\\
\left\| \hat{\eta}_{Y1}^{(d)} - e_{1}^{(d)} \right\|_{2}^{2} &= \left( \hat{\eta}_{Y1}^{(d)}(1) - 1 \right)^{2} + \left\| \tilde{\eta}_{Y1}^{(d)}(i) \right\|_{2}^{2}
\mathop{\longrightarrow}_{d \to \infty}^{p} 0.
\end{align*}
To calculate the limiting value of $\hat{\lambda}_{XY1}^{(d)}/d^{\alpha}$ as $d \to \infty$, using the unitary invariance property of the Frobenius norm and the previous result about the limiting vectors of $\hat{\eta}_{X1}^{(d)}$ and $\hat{\eta}_{Y1}^{(d)}$ as $d \to \infty$,
\begin{align*}
&\left\| \frac{n\hat{\lambda}_{XY1}^{(d)}}{d^{\alpha}} \hat{\eta}_{X1}^{(d)} \left( \hat{\eta}_{Y1}^{(d)} \right)^{T} - \mathbf{M}_{XY}^{(d)} \right\|_{F}^{2}\\
&= \left\| \frac{n\hat{\lambda}_{XY1}^{(d)}}{d^{\alpha}} \hat{\eta}_{X1}^{(d)} \left( \hat{\eta}_{Y1}^{(d)} \right)^{T} - \frac{\langle c_{1}, d_{1} \rangle}{d^{\alpha}} e_{1}^{(d)} \left( e_{1}^{(d)} \right)^{T} \right\|_{F}^{2}\\
&=\left\| \frac{n\hat{\lambda}_{XY1}^{(d)}}{d^{\alpha}} \left( \hat{\eta}_{Y1}^{(d)} \right)^{T} \hat{\eta}_{X1}^{(d)} \left( \hat{\eta}_{Y1}^{(d)} \right)^{T} \hat{\eta}_{X1}^{(d)} - \frac{\langle c_{1}, d_{1} \rangle}{d^{\alpha}} \left( \hat{\eta}_{Y1}^{(d)} \right)^{T} e_{1}^{(d)} \left( e_{1}^{(d)} \right)^{T} \hat{\eta}_{X1}^{(d)} \right\|_{F}^{2}\\
&=\left\| \frac{n\hat{\lambda}_{XY1}^{(d)}}{d^{\alpha}} - \frac{\langle c_{1}, d_{1} \rangle}{d^{\alpha}} \left( \hat{\eta}_{Y1}^{(d)} \right)^{T} e_{1}^{(d)} \left( e_{1}^{(d)} \right)^{T} \hat{\eta}_{X1}^{(d)} \right\|_{F}^{2}\\
&=\left( \frac{\langle c_{1}, d_{1} \rangle}{d^{\alpha}} \right)^{2} \left\| \frac{n\hat{\lambda}_{XY1}^{(d)}/d^{\alpha}}{\langle c_{1}, d_{1} \rangle/d^{\alpha}} - \left( \hat{\eta}_{Y1}^{(d)} \right)^{T} e_{1}^{(d)} \left( e_{1}^{(d)} \right)^{T} \hat{\eta}_{X1}^{(d)} \right\|_{F}^{2}\\
&\mathop{\longrightarrow}_{d \to \infty}^{p} \text{A}^{2} \left( \frac{\text{B}}{\text{C}} - 1 \right)^{2}.
\end{align*}
where,
\begin{align*}
\text{A}=\lim_{d \to \infty} \frac{\langle c_{1}, d_{1} \rangle}{d^{\alpha}}, \,\, \text{B}=\lim_{d \to \infty} \frac{n\hat{\lambda}_{XY1}^{(d)}}{d^{\alpha}},\,\, \text{C}=\lim_{d \to \infty} \frac{\langle c_{1}, d_{1} \rangle}{d^{\alpha}}.
\end{align*}
By Lemma~\ref{thm2},
\begin{align*}
\left \| \frac{n\hat{\lambda}_{XY1}^{(d)}}{d^{\alpha}} \hat{\eta}_{X1}^{(d)} \hat{\eta}_{Y1}^{(d)} - \mathbf{M}_{XY}^{(d)} \right \|_{F}^{2} \mathop{\longrightarrow}_{d \to \infty}^{p} 0.
\end{align*}
Since $\text{A} \asymp O_{p}(1)$,
\begin{align*}
\left( \frac{B}{C} - 1 \right)^{2} \mathop{\longrightarrow}_{d \to \infty}^{p} 0.
\end{align*}
Therefore,
\begin{align*}
\frac{B}{C} \mathop{\longrightarrow}_{d \to \infty}^{p} 1.
\end{align*}
\end{proof}
\begin{lem}[CCA HDLSS Asymptotic lemma 4.]
\label{lem5}
The magnitudes of the entries $\hat{\eta}_{X1}^{(d)}(i)$ and $\hat{\eta}_{Y1}^{(d)}(i)$, $i=2,3,..,d$, of the first sample singular vectors $\hat{\eta}_{X1}^{(d)}$ and $\hat{\eta}_{Y1}^{(d)}$ are of $O_{p}(1/\sqrt{d^{\alpha}})$ as $d \to \infty$.
\end{lem}
\begin{proof}
Let $\hat{\mathbf{\Sigma}}^{(d)}_{XY}(i,j)$ be the entry of the $i$th row and $j$th column of $\hat{\mathbf{\Sigma}}^{(d)}_{XY}$. Consider $\hat{\mathbf{\Sigma}}^{(d)}_{XY}(1,3)$, which is $O_{p}(\sqrt{d^{\alpha}})$. The contribution of $\hat{\lambda}_{XY1}^{(d)}$, $\hat{\eta}_{X1}^{(d)}$ and $\hat{\eta}_{Y1}^{(d)}$ to $\hat{\mathbf{\Sigma}}^{(d)}_{XY}(1,3)$ is,
\begin{align*}
\hat{\lambda}_{XY1}^{(d)} \hat{\eta}_{X1}^{(d)}(1) \hat{\eta}_{Y1}^{(d)}(3).
\end{align*}
From Lemma~\ref{lem4},
\begin{align*}
\hat{\lambda}_{XY1}^{(d)} \asymp O_{p}(d^{\alpha}), \,\, \hat{\eta}_{X1}^{(d)}(1) \mathop{\longrightarrow}_{d \to \infty}^{p} 1.
\end{align*}
Therefore,
\begin{align*}
\hat{\eta}_{Y1}^{(d)}(3) = O_{p}\left( \frac{1}{\sqrt{d^{\alpha}}} \right).
\end{align*}
Argument is similar for the rest of entries.
\end{proof}
\begin{lem}[CCA HDLSS Asymptotic lemma 5.]
\label{lem6}
The magnitudes of the sample singular values $\hat{\lambda}_{XYi}^{(d)}$ for $i=2,3,..,n$ are of $O_{p}(d)$.
\end{lem}
\begin{proof}
Let $\tilde{\mathbf{\Sigma}}^{(d)}_{XY}$ be the $(d-1) \times (d-1)$ sample cross-covariance matrix $\hat{\mathbf{\Sigma}}^{(d)}_{XY}$ from which the first row and column are set to 0. Let $\tilde{\eta}_{Xi}^{(d)}$ and $\tilde{\eta}_{Yi}^{(d)}$ be $\hat{\eta}_{Xi}^{(d)}$ and $\hat{\eta}_{Yi}^{(d)}$ where their first entries are set to 0. Then,
\begin{align*}
&\frac{n\hat{\lambda}_{XY1}^{(d)}}{d} \tilde{\eta}_{X1}^{(d)} \left( \tilde{\eta}_{Y1}^{(d)} \right)^{T}+ \sum_{i=2}^{n} \frac{n\hat{\lambda}_{XYi}^{(d)}}{d} \tilde{\eta}_{Xi}^{(d)} \left( \tilde{\eta}_{Yi}^{(d)} \right)^{T}= \frac{n\tilde{\mathbf{\Sigma}}^{(d)}_{XY}}{d},\\
&\sum_{i=2}^{n} \frac{n\hat{\lambda}_{XYi}^{(d)}}{d} \tilde{\eta}_{Xi}^{(d)} \left( \tilde{\eta}_{Yi}^{(d)}\right)^{T} = \frac{n\tilde{\mathbf{\Sigma}}^{(d)}_{XY}}{d} - \frac{\hat{n\lambda}_{XY1}^{(d)}}{d} \tilde{\eta}_{X1}^{(d)} \left( \tilde{\eta}_{Y1}^{(d)}\right)^{T},\\
&\left\| \sum_{i=2}^{n} \frac{n\hat{\lambda}_{XYi}^{(d)}}{d} \tilde{\eta}_{Xi}^{(d)} \left( \tilde{\eta}_{Yi}^{(d)} \right)^{T}\right\|_{F} \le \left\| \frac{n\tilde{\mathbf{\Sigma}}^{(d)}_{XY}}{d} \right\|_{F} + \left\| \frac{n\hat{\lambda}_{XY1}^{(d)}}{d} \tilde{\eta}_{X1}^{(d)} \left( \tilde{\eta}_{Y1}^{(d)} \right)^{T}\right\|_{F}.
\end{align*}
Note that $\| c_{2} \|_{2}^{2}$ and $\| d_{2} \|_{2}^{2}$ are $O_{P}(1)$ and that $\| z_{Yi\bullet} \|_{2}^{2} \sim \chi_{n}^{2}$ and $\| z_{Xi\bullet} \|_{2}^{2} \sim \chi_{n}^{2}$. Using the Cauchy-Schwarz inequality,
\begin{align*}
\left\| \frac{n\tilde{\mathbf{\Sigma}}^{(d)}_{XY}}{d} \right\|_{F}^{2} 
&=\frac{\langle c_{2}, d_{2} \rangle^{2}}{d^{2}}
+\tau_{Y}^{2}\sum_{i=3}^{d} \frac{\langle c_{2}, z_{Yi\bullet} \rangle^{2}}{d^{2}}\\
&\quad +\tau_{X}^{2}\sum_{i=3}^{d} \frac{\langle z_{Xi\bullet}, d_{2} \rangle^{2}}{d^{2}}
+\tau_{X}^{2}\tau_{Y}^{2}\sum_{i=3}^{d}\sum_{j=3}^{d} \frac{\langle z_{Xi\bullet}, z_{Yj\bullet} \rangle^{2}}{d^{2}}\\
& \le \frac{\langle c_{2}, d_{2} \rangle^{2}}{d^{2}}
+\tau_{Y}^{2}\sum_{i=3}^{d} \frac{\| c_{2} \|_{2}^{2} \| z_{Yi\bullet} \|_{2}^{2}}{d^{2}}\\
&\quad +\tau_{X}^{2}\sum_{i=3}^{d} \frac{\| z_{Xi\bullet}\|_{2}^{2} \|d_{2} \|_{2}^{2}}{d^{2}}
+\tau_{X}^{2}\tau_{Y}^{2}\sum_{i=3}^{d}\sum_{j=3}^{d} \frac{\| z_{Xi\bullet}\|_{2}^{2} \| z_{Yj\bullet} \|_{2}^{2}}{d^{2}}\\
&=\frac{\langle c_{2}, d_{2} \rangle^{2}}{d^{2}}
+\tau_{Y}^{2} \left \| \frac{c_{2}}{\sqrt{d}} \right \|_{2}^{2} \sum_{i=3}^{d} \frac{\| z_{Yi\bullet} \|_{2}^{2}}{d}\\
&\quad +\tau_{Y}^{2} \left \| \frac{d_{2}}{\sqrt{d}} \right \|_{2}^{2} \sum_{i=3}^{d} \frac{\| z_{Yi\bullet} \|_{2}^{2}}{d}
+\tau_{X}^{2}\tau_{Y}^{2} \sum_{i=3}^{d} \frac{\| z_{Xi\bullet} \|_{2}^{2}}{d} \sum_{j=3}^{d} \frac{\| z_{Yj\bullet} \|_{2}^{2}}{d}\\
&\mathop{\longrightarrow}_{d \to \infty}^{p} 0+0+0+\tau_{X}^{2}\tau_{Y}^{2}E^{2}(\chi_{n}^{2})\\
&=\tau_{X}^{2}\tau_{Y}^{2}n^{2}.
\end{align*}
Using Lemma~\ref{lem4} and~\ref{lem5} on the magnitudes of $\hat{\lambda}_{XY1}^{(d)}$ and the entries of $\hat{\eta}_{X1}^{(d)}(i)$ and $\hat{\eta}_{Y1}^{(d)}(i)$,
\begin{align*}
\hat{\lambda}_{XY1}^{(d)} \asymp O_{p}(d^{\alpha}), \,\, \hat{\eta}_{X1}^{(d)}(i) = O_{p}(\frac{1}{\sqrt{d^{\alpha}}}), i=2,3,..,d.
\end{align*}
Then we have,
\begin{align*}
\left\| \frac{\hat{\lambda}_{XY1}^{(d)}}{d} \tilde{\eta}_{X1}^{(d)} \left(\tilde{\eta}_{Y1}^{(d)} \right)^{T}\right\|_{F}^{2} &= \left( \frac{\hat{\lambda}_{XY1}^{(d)}}{d} \right)^{2} \sum_{i=2}^{d} \sum_{j=2}^{d} \left( \hat{\eta}_{X1}^{(d)}(i) \right)^{2} \left( \hat{\eta}_{X1}^{(d)}(j) \right)^{2}\\
&= \left( \hat{\lambda}_{XY1}^{(d)} \right)^{2} \sum_{i=2}^{d} \frac{\left( \hat{\eta}_{X1}^{(d)}(i) \right)^{2}}{d} \sum_{j=2}^{d} \frac{\left( \hat{\eta}_{X1}^{(d)}(j) \right)^{2}}{d}\\
&= \left( \frac{\hat{\lambda}_{XY1}^{(d)}}{d^{\alpha}} \right)^{2} \sum_{i=2}^{d} \frac{\left( \sqrt{d^{\alpha}}\hat{\eta}_{X1}^{(d)}(i) \right)^{2}}{d} \sum_{j=2}^{d} \frac{\left( \sqrt{d^{\alpha}}\hat{\eta}_{X1}^{(d)}(j) \right)^{2}}{d}\\
&= O_{p}(1).
\end{align*}
Hence
\begin{align}
\label{sec3:equ37}
\left\| \sum_{i=2}^{n} \frac{\hat{\lambda}_{XYi}^{(d)}}{d} \tilde{\eta}_{Xi}^{(d)} \left(\tilde{\eta}_{Y1}^{(d)} \right)^{T} \right\|_{F}^{2} = O_{p}(1).
\end{align}
Expanding the squared Frobenius norm of the above matrix,
\begin{align*}
\left\| \sum_{i=2}^{n} \frac{\hat{\lambda}_{XYi}^{(d)}}{d} \tilde{\eta}_{Xi}^{(d)} \left(\tilde{\eta}_{Y1}^{(d)} \right)^{T} \right\|_{F}^{2} &= \sum_{j=2}^{d} \sum_{k=2}^{d} \left( \sum_{i=2}^{n} \frac{\hat{\lambda}_{XYi}^{(d)}}{d} \hat{\eta}_{Xi}^{(d)}(j) \hat{\eta}_{Yi}^{(d)}(k) \right)^{2}\\
&= \sum_{i=2}^{n} \left( \frac{\hat{\lambda}_{XYi}^{(d)}}{d} \right)^{2} \sum_{j=2}^{d} \sum_{k=2}^{d} \left( \hat{\eta}_{Xi}^{(d)}(j) \right)^{2} \left( \hat{\eta}_{Yi}^{(d)}(k) \right)^{2}\\
&\quad + 2\sum_{i \not= i'} \frac{\hat{\lambda}_{XYi}^{(d)}}{d} \frac{\hat{\lambda}_{XYi'}^{(d)}}{d} \sum_{j=2}^{d} \sum_{k=2}^{d} \left( \hat{\eta}_{Xi}^{(d)}(j) \hat{\eta}_{Xi'}^{(d)}(j)\right) \left( \hat{\eta}_{Yi}^{(d)}(k) \hat{\eta}_{Yi'}^{(d)}(k) \right)\\
&= \sum_{i=2}^{n} \left( \frac{\hat{\lambda}_{XYi}^{(d)}}{d} \right)^{2} \sum_{j=2}^{d} \left( \hat{\eta}_{Xi}^{(d)}(j) \right)^{2} \sum_{k=2}^{d} \left( \hat{\eta}_{Yi}^{(d)}(k) \right)^{2}\\
&\quad + 2\sum_{i \not= i'} \frac{\hat{\lambda}_{XYi}^{(d)}}{d} \frac{\hat{\lambda}_{XYi'}^{(d)}}{d} \sum_{j=2}^{d} \left( \hat{\eta}_{Xi}^{(d)}(j) \hat{\eta}_{Xi'}^{(d)}(j) \right) \sum_{k=2}^{d} \left( \hat{\eta}_{Yi}^{(d)}(k) \hat{\eta}_{Yi'}^{(d)}(k) \right)\\
&= \sum_{i=2}^{n} \left( \frac{\hat{\lambda}_{XYi}^{(d)}}{d} \right)^{2} \left\| \tilde{\eta}_{Xi}^{(d)} \right\|_{2}^{2} \left\| \tilde{\eta}_{Yi}^{(d)} \right\|_{2}^{2}\\
&\quad + 2\sum_{2 \le i,i' \le n, \, i \not= i'} \frac{\hat{\lambda}_{XYi}^{(d)}}{d} \frac{\hat{\lambda}_{XYi'}^{(d)}}{d} \left\langle \tilde{\eta}_{Xi}^{(d)}, \tilde{\eta}_{Xi'}^{(d)} \right\rangle \left\langle \tilde{\eta}_{Yi}^{(d)}, \hat{\eta}_{Yi'}^{(d)} \right\rangle.
\end{align*}
We, to simplify the last part in the above equalities, show that $\| \tilde{\eta}_{Xi}^{(d)} \|_{2}^{2}$ and $\| \tilde{\eta}_{Yi}^{(d)} \|_{2}^{2}$, for $i=2,3,\dots,n$, approach to 1 and that $\langle \tilde{\eta}_{Xi}^{(d)}, \tilde{\eta}_{Xi'}^{(d)} \rangle$ and $\langle \tilde{\eta}_{Yi}^{(d)}, \tilde{\eta}_{Yi'}^{(d)} \rangle$, for $i,i'=2,3,\dots,n$ and $i\not=i'$, approach to 0 as $d \to \infty$. Consider $\| \tilde{\eta}_{X2}^{(d)} \|_{2}^{2}$. Since $\|\hat{\eta}_{X1}^{(d)}-e_{1}^{(d)}\|_{2}$ converges in probability to 0 by Lemma~\ref{lem4} and $\| \hat{\eta}_{X1}^{(d)} \|_{2}^{2} = 1$,
\begin{align*}
\left\| \tilde{\eta}_{X1}^{(d)} \right\|_{2}^{2} = 1 - \left( \hat{\eta}_{X1}^{(d)}(1) \right)^{2} \mathop{\longrightarrow}_{d \to \infty}^{p} 0,
\end{align*}
which leads to,
\begin{align*}
\left\langle \tilde{\eta}_{X1}^{(d)}, \tilde{\eta}_{X2}^{(d)} \right\rangle^{2} \le \left\| \tilde{\eta}_{X1}^{(d)} \right\|_{2}^{2} \left\| \tilde{\eta}_{X2}^{(d)} \right\|_{2}^{2} \le \left\| \tilde{\eta}_{X1}^{(d)} \right\|_{2}^{2} \times 1 \mathop{\longrightarrow}_{d \to \infty}^{p} 0.
\end{align*}
Using the orthogonality of $\hat{\eta}_{X1}^{(d)}$ and $\hat{\eta}_{X2}^{(d)}$,
\begin{align*}
\left\langle \hat{\eta}_{X1}^{(d)}, \hat{\eta}_{X2}^{(d)} \right\rangle = \hat{\eta}_{X1}^{(d)}(1) \hat{\eta}_{X2}^{(d)}(1) + \left\langle \tilde{\eta}_{X1}^{(d)}, \tilde{\eta}_{X2}^{(d)} \right\rangle = 0.
\end{align*}
Since $\hat{\eta}_{X1}^{(d)}(1)$ converges in probability to 1 by Lemma~\ref{lem4}, we have,
\begin{align*}
\hat{\eta}_{X2}^{(d)}(1) \mathop{\longrightarrow}_{d \to \infty}^{p} 0.
\end{align*}
Since $\| \hat{\eta}_{X2}^{(d)} \|_{2}^{2} = 1$,
\begin{align*}
\left\| \tilde{\eta}_{X2}^{(d)} \right\|_{2}^{2} \mathop{\longrightarrow}_{d \to \infty}^{p} 1.
\end{align*}
Simiarly,
\begin{align*}
\left\| \tilde{\eta}_{Xi}^{(d)} \right\|_{2}^{2} \mathop{\longrightarrow}_{d \to \infty}^{p} 1, \,\, \left\| \tilde{\eta}_{Yi}^{(d)} \right\|_{2}^{2}\mathop{\longrightarrow}_{d \to \infty}^{p} 1, \,\, i=2,3,\dots,n.
\end{align*}
Now consider $\langle \tilde{\eta}_{X2}^{(d)}, \tilde{\eta}_{X3}^{(d)} \rangle$. Since $\hat{\eta}_{X2}^{(d)}$ and $\hat{\eta}_{X3}^{(d)}$ are orthogonal,
\begin{equation*}
\left\langle \hat{\eta}_{X2}^{(d)}, \hat{\eta}_{X3}^{(d)} \right\rangle = \hat{\eta}_{X2}^{(d)}(1) \hat{\eta}_{X3}^{(d)}(1) + \left\langle \tilde{\eta}_{X2}^{(d)}, \tilde{\eta}_{X3}^{(d)} \right\rangle = 0.
\end{equation*}
We showed that $\| \tilde{\eta}_{X2}^{(d)} \|_{2}^{2}$ and $\| \tilde{\eta}_{X3}^{(d)} \|_{2}^{2}$ converge in probability to 0, which implies,
\begin{equation*}
\left\langle \tilde{\eta}_{X2}^{(d)}, \tilde{\eta}_{X3}^{(d)} \right\rangle \mathop{\longrightarrow}_{d \to \infty}^{p} 0.
\end{equation*}
Simiarly,
\begin{equation*}
\left\langle \tilde{\eta}_{Xi}^{(d)}, \tilde{\eta}_{Xi'}^{(d)} \right\rangle \mathop{\longrightarrow}_{d \to \infty}^{p} 0, \left\langle \tilde{\eta}_{Yi}^{(d)}, \,\, \tilde{\eta}_{Yi'}^{(d)} \right\rangle \mathop{\longrightarrow}_{d \to \infty}^{p} 0, \,\, i,i'=2,3,\dots, n, \,\, i\not=i'.
\end{equation*}
Then, by~(\ref{sec3:equ37}),
\begin{align*}
&\left\| \sum_{i=2}^{n} \frac{\hat{\lambda}_{XYi}^{(d)}}{d} \tilde{\eta}_{Xi}^{(d)} \left(\tilde{\eta}_{Y1}^{(d)} \right)^{T} \right\|_{F}^{2}\\
&\quad \quad =\sum_{i=2}^{n} \left( \frac{\hat{\lambda}_{XYi}^{(d)}}{d} \right)^{2} \left\| \tilde{\eta}_{Xi}^{(d)} \right\|_{2}^{2} \left\| \tilde{\eta}_{Yi}^{(d)} \right\|_{2}^{2}
+ 2\sum_{2 \le i,i' \le n, \, i \not= i'} \frac{\hat{\lambda}_{XYi}^{(d)}}{d} \frac{\hat{\lambda}_{XYi'}^{(d)}}{d} \left\langle \tilde{\eta}_{Xi}^{(d)}, \tilde{\eta}_{Xi'}^{(d)} \right\rangle \left\langle \tilde{\eta}_{Yi}^{(d)}, \hat{\eta}_{Yi'}^{(d)} \right\rangle\\
&\quad \quad \mathop{\longrightarrow}_{d \to \infty}^{p} \sum_{i=2}^{n} \left( \frac{\hat{\lambda}_{XYi}^{(d)}}{d} \right)^{2} = O_{p}(1).
\end{align*}
Hence, the magnitudes of $\hat{\lambda}_{XYi}^{(d)}$, $i=2,3,..n$ are of $O_{p}(d)$ as $d \to \infty$.
\end{proof}
\begin{lem}[CCA HDLSS Asymptotic lemma 6.]
\label{lem15}
This lemma improves Lemma~\ref{lem6} by providing a precise limiting value. The sample singular values $\hat{\lambda}_{XYi}^{(d)}$ scaled by $1/d$, for $i=2,3,..,n$, converge in probability to the following quantity as $d \to \infty$,
\begin{align*}
\frac{\hat{\lambda}_{XYi}^{(d)}}{d} \mathop{\longrightarrow}_{d \to \infty}^{p} 0, \,\, i=2,3,\dots,n.
\end{align*}
\end{lem}
\begin{proof}
Let $\tilde{\mathbf{\Sigma}}^{(d)}_{XY}$ be the $(d-1) \times (d-1)$ sample cross-covariance matrix $\hat{\mathbf{\Sigma}}^{(d)}_{XY}$ from which the first row and column are set to 0. We showed that,
\begin{align}
\label{sec3:equ36}
\left\| \frac{\tilde{\mathbf{\Sigma}}^{(d)}_{XY}}{d} \right\|_{F}^{2} \mathop{\longrightarrow}_{d \to \infty}^{p} \le \tau_{X}^{2}\tau_{Y}^{2}.
\end{align}
Let $\tilde{\eta}_{Xi}^{(d)}$ and $\tilde{\eta}_{Yi}^{(d)}$ be $\hat{\eta}_{Xi}^{(d)}$ and $\hat{\eta}_{Yi}^{(d)}$ where their first entries are set to 0. Then,
\begin{align*}
\left\| \frac{\tilde{\mathbf{\Sigma}}^{(d)}_{XY}}{d} \right\|_{F}^{2}
= &\sum_{j=2}^{d} \sum_{k=2}^{d} \left( \sum_{i=1}^{n} \frac{\hat{\lambda}_{XYi}^{(d)}}{d} \hat{\eta}_{Xi}^{(d)}(j) \hat{\eta}_{Yi}^{(d)}(k) \right)^{2}\\
= &\sum_{i=1}^{n} \left( \frac{\hat{\lambda}_{XYi}^{(d)}}{d} \right)^{2} \sum_{j=2}^{d} \sum_{k=2}^{d} \left( \hat{\eta}_{Xi}^{(d)}(j) \right)^{2} \left( \hat{\eta}_{Yi}^{(d)}(k) \right)^{2}\\
&+ 2\sum_{1 \le i,i' \le n, \, i \not= i'} \frac{\hat{\lambda}_{XYi}^{(d)}}{d} \frac{\hat{\lambda}_{XYi'}^{(d)}}{d} \sum_{j=2}^{d} \sum_{k=2}^{d} \left( \hat{\eta}_{Xi}^{(d)}(j) \hat{\eta}_{Xi'}^{(d)}(j)\right) \left( \hat{\eta}_{Yi}^{(d)}(k) \hat{\eta}_{Yi'}^{(d)}(k) \right)\\
= &\sum_{i=1}^{n} \left( \frac{\hat{\lambda}_{XYi}^{(d)}}{d} \right)^{2} \sum_{j=2}^{d} \left( \hat{\eta}_{Xi}^{(d)}(j) \right)^{2} \sum_{k=2}^{d} \left( \hat{\eta}_{Yi}^{(d)}(k) \right)^{2}\\
&+ 2\sum_{1 \le i,i' \le n, \, i \not= i'} \frac{\hat{\lambda}_{XYi}^{(d)}}{d} \frac{\hat{\lambda}_{XYi'}^{(d)}}{d} \sum_{j=2}^{d} \left( \hat{\eta}_{Xi}^{(d)}(j) \hat{\eta}_{Xi'}^{(d)}(j) \right) \sum_{k=2}^{d} \left( \hat{\eta}_{Yi}^{(d)}(k) \hat{\eta}_{Yi'}^{(d)}(k) \right)\\
= &\underset{(1)}{\underline{\left( \frac{\hat{\lambda}_{XY1}^{(d)}}{d} \right)^{2} \sum_{j=2}^{d} \left( \tilde{\eta}_{X1}^{(d)}(j) \right)^{2} \sum_{k=2}^{d} \left( \tilde{\eta}_{Y1}^{(d)}(k) \right)^{2}}}\\
&+ \underset{(2)}{\underline {2\sum_{i=2}^{n} \frac{\hat{\lambda}_{XY1}^{(d)}}{d} \frac{\hat{\lambda}_{XYi}^{(d)}}{d} \left\langle \tilde{\eta}_{X1}^{(d)}, \tilde{\eta}_{Xi}^{(d)} \right\rangle \left\langle \tilde{\eta}_{Y1}^{(d)}, \tilde{\eta}_{Yi}^{(d)} \right\rangle}}
+\underset{(3)}{\underline{\sum_{i=2}^{n} \left( \frac{\hat{\lambda}_{XYi}^{(d)}}{d} \right)^{2} \left\| \tilde{\eta}_{Xi}^{(d)} \right\|_{2}^{2} \left\| \tilde{\eta}_{Yi}^{(d)} \right\|_{2}^{2}}}\\
&+ \underset{(4)}{\underline{2\sum_{2 \le i,i' \le n, \, i \not= i'} \frac{\hat{\lambda}_{XYi}^{(d)}}{d} \frac{\hat{\lambda}_{XYi'}^{(d)}}{d} \left\langle \tilde{\eta}_{Xi}^{(d)}, \tilde{\eta}_{Xi'}^{(d)} \right\rangle \left\langle \tilde{\eta}_{Yi}^{(d)}, \tilde{\eta}_{Yi'}^{(d)} \right\rangle}}.
\end{align*}
In the proof of Lemma~\ref{lem6}, we showed that $\| \tilde{\eta}_{Xi}^{(d)} \|_{2}^{2}$ and $\| \tilde{\eta}_{Yi}^{(d)}\|_{2}^{2}$, for $i=2,3,\dots,n$, converge in probability to 1 and that $\langle \tilde{\eta}_{Xi}^{(d)}, \tilde{\eta}_{Xi'}^{(d)} \rangle$ and $\langle \tilde{\eta}_{Yi}^{(d)}, \tilde{\eta}_{Xi'}^{(d)} \rangle$, for $i,i'=2,3,\dots,n$ and $i \not= i'$, converge in probability to 1 as $d \to \infty$. Since $\hat{\lambda}_{XYi}$, for $i=2,3,\dots,n$, is of magnitude of $O_{p}(d)$ by Lemma~\ref{lem6}, we have for (3) and (4),
\begin{align*}
\left| \sum_{i=2}^{n} \left( \frac{\hat{\lambda}_{XYi}^{(d)}}{d} \right)^{2} \left\| \tilde{\eta}_{Xi}^{(d)} \right\|_{2}^{2} \left\| \tilde{\eta}_{Yi}^{(d)} \right\|_{2}^{2}-\sum_{i=2}^{n} \left( \frac{\hat{\lambda}_{XYi}^{(d)}}{d} \right)^{2} \right| &\mathop{\longrightarrow}_{d \to \infty}^{p} 0,\\[5pt]
2\sum_{2 \le i,i' \le n, \, i \not= i'} \frac{\hat{\lambda}_{XYi}^{(d)}}{d} \frac{\hat{\lambda}_{XYi'}^{(d)}}{d} \left\langle \tilde{\eta}_{Xi}^{(d)}, \tilde{\eta}_{Xi'}^{(d)} \right\rangle \left\langle \tilde{\eta}_{Yi}^{(d)}, \tilde{\eta}_{Yi'}^{(d)} \right\rangle &\mathop{\longrightarrow}_{d \to \infty}^{p} 0.
\end{align*}
Since $\| \tilde{\eta}_{Xi}^{(d)} \|_{2}^{2}$ and $\| \tilde{\eta}_{Yi}^{(d)}\|_{2}^{2}$, for $i=2,3,\dots,n$, converge in probability to 1, the probability that an infinite number of entries of $\tilde{\eta}_{Xi}^{(d)}$ and $\tilde{\eta}_{Yi}^{(d)}$ are of ($\asymp 1/d^{a}$), for $0 \le a < 1/2$, is 0. Suppose that an infinite number of entries of $\tilde{\eta}_{Xi}^{(d)}$ are of ($\asymp 1/d^{a}$), for $0 \le a < 1/2$. Then, the squared sum of its entries blows up,
\begin{align*}
\sum_{j=2}^{d} \left(\tilde{\eta}_{Xi}^{(d)}(j)\right)^{2}=\frac{1}{d^{2a-1}} \sum_{j=2}^{d} \frac{d^{2a}\tilde{\eta}_{Xi}^{(d)}(j)}{d} \mathop{\longrightarrow}_{d \to \infty}^{p} \infty,\,\, 0 \le a < \frac{1}{2}.
\end{align*}
Hence, only a finite number of entries of $\tilde{\eta}_{Xi}^{(d)}$ and $\tilde{\eta}_{Yi}^{(d)}$ are of magnitude of ($\asymp 1/d^{a}$), for $0 \le a < 1/2$ and the rest of entries are of $O_{P}(1/\sqrt{d})$ as $d \to \infty$. Using this fact and Lemma~\ref{lem5},
\begin{align*}
\left\langle \tilde{\eta}_{X1}^{(d)}, \tilde{\eta}_{Xi}^{(d)} \right\rangle&=\sum_{j=2}^{d} \tilde{\eta}_{X1}^{(d)}(j) \tilde{\eta}_{Xi}^{(d)}(j)\\
&=\sum_{j \in I} \tilde{\eta}_{X1}^{(d)}(j) \tilde{\eta}_{Xi}^{(d)}(j) + \sum_{j \not\in I} \tilde{\eta}_{X1}^{(d)}(j) \tilde{\eta}_{Xi}^{(d)}(j)\\
&=\sum_{j \in I} \tilde{\eta}_{X1}^{(d)}(j) \tilde{\eta}_{Xi}^{(d)}(j) + \frac{1}{\sqrt{d^{\alpha-1}}}\sum_{j \not\in I} \frac{\left(\sqrt{d^{\alpha}}\tilde{\eta}_{X1}^{(d)}(j)\right) \left(\sqrt{d}\tilde{\eta}_{Xi}^{(d)}(j)\right)}{d}\\
&=O_{P}(1/\sqrt{d^{\alpha-1}}), \,\, i=2,3,\dots,n,
\end{align*}
where $I$ is an index set denoting the entries of $\tilde{\eta}_{Xi}^{(d)}$ of magnitude of ($\asymp 1/d^{a}$), for $0 \le a < 1/2$. Similarly, the magnitide of $\left\langle \tilde{\eta}_{Y1}^{(d)}, \tilde{\eta}_{Yi}^{(d)} \right\rangle$ is $O_{P}(1/\sqrt{d^{\alpha-1}})$, for $i=2,3,\dots,n$.
Hence, by Lemma~\ref{lem4} and~\ref{lem6},
\begin{align*}
(2)&=2\sum_{i=2}^{n} \frac{\hat{\lambda}_{XY1}^{(d)}}{d} \frac{\hat{\lambda}_{XYi}^{(d)}}{d} \left\langle \tilde{\eta}_{X1}^{(d)}, \tilde{\eta}_{Xi}^{(d)} \right\rangle \left\langle \tilde{\eta}_{Y1}^{(d)}, \tilde{\eta}_{Yi}^{(d)} \right\rangle\\
&=2\sum_{i=2}^{n} \frac{\hat{\lambda}_{XY1}^{(d)}}{d^{\alpha}} \frac{\hat{\lambda}_{XYi}^{(d)}}{d} \left(\sqrt{d^{\alpha-1}}\left\langle \tilde{\eta}_{X1}^{(d)}, \tilde{\eta}_{Xi}^{(d)} \right\rangle \right) \left(\sqrt{d^{\alpha-1}} \left\langle \tilde{\eta}_{Y1}^{(d)}, \tilde{\eta}_{Yi}^{(d)}\right\rangle \right)\\
&= O_{P}(1).
\end{align*}
We prove that the term (2) above indeed converges to 0 as $d \to \infty$. With Lemma~\ref{lem5} and the fact that $\langle \hat{\eta}_{X1}^{(d)}, \hat{\eta}_{Xi}^{(d)} \rangle=0$ and $\| \tilde{\eta}_{Xi}^{(d)}\|$ converges in probability to 0, for $i=2,3,\dots,n$, the Cauchy-Schwarz inequality implies,
\begin{align*}
\left( \hat{\eta}_{X1}^{(d)}(1)\hat{\eta}_{Xi}^{(d)}(1) \right)^{2}
&=\left( -\sum_{j=2}^{d}\hat{\eta}_{X1}^{(d)}(j)\hat{\eta}_{Xi}^{(d)}(j) \right)^{2}\\
&=\left\langle \tilde{\eta}_{X1}^{(d)},\tilde{\eta}_{Xi}^{(d)} \right\rangle^{2}\\
&\le \left\|\tilde{\eta}_{X1}^{(d)}\right\|_{2}^{2} \left\|\tilde{\eta}_{Xi}^{(d)}\right\|_{2}^{2}\\
&=\sum_{j=2}^{d}\left(\hat{\eta}_{X1}^{(d)}(j)\right)^{2} \left\|\tilde{\eta}_{Xi}^{(d)}\right\|_{2}^{2}\\
&=\frac{1}{d^{\alpha-1}} \sum_{j=2}^{d} \frac{\left(\sqrt{d^{\alpha}}\hat{\eta}_{X1}^{(d)}(j)\right)^{2}}{d} \left\|\tilde{\eta}_{Xi}^{(d)}\right\|_{2}^{2}\\
&= O_{P}(\frac{1}{d^{\alpha-1}}).
\end{align*}
Since $\hat{\eta}_{X1}^{(d)}(1)$ converges in probability to 1 by Lemma~\ref{lem4}, we have,
\begin{align}
\label{sec3:equ35}
\hat{\eta}_{Xi}^{(d)}(1) = O_{P}(\frac{1}{\sqrt{d^{\alpha-1}}}), \,\, i=2,3,\dots,n.
\end{align}
Similarly,
\begin{equation*}
\hat{\eta}_{Yi}^{(d)}(1) = O_{P}(\frac{1}{\sqrt{d^{\alpha-1}}}), \,\, i=2,3,\dots,n.
\end{equation*}
Consider $\hat{\mathbf{\Sigma}}^{(d)}_{XY}(3,1)$, an entry in the third row and first column of the sample cross-covariance matrix $\hat{\mathbf{\Sigma}}^{(d)}_{XY}$ given in~(\ref{sec3:equ19}). Then, we have,
\begin{align*}
\frac{\hat{\mathbf{\Sigma}}^{(d)}_{XY}(3,1)}{\sqrt{d^{\alpha}}}
&=\frac{\tau_{X}\langle z_{X3\bullet}, d_{1} \rangle}{n\sqrt{d^{\alpha}}}\\
&=\frac{1}{\sqrt{d^{\alpha}}}\sum_{i=1}^{n} \hat{\lambda}_{XYi}^{(d)} \hat{\eta}_{Xi}^{(d)}(3) \hat{\eta}_{Yi}^{(d)}(1)\\
&=\frac{\hat{\lambda}_{XY1}^{(d)}}{d^{\alpha}} \left( \sqrt{d^{\alpha}} \hat{\eta}_{X1}^{(d)}(3) \right) \hat{\eta}_{Y1}^{(d)}(1)+\sum_{i=2}^{n} \frac{\hat{\lambda}_{XYi}^{(d)}}{\sqrt{d^{2\alpha-1}}} \hat{\eta}_{Xi}^{(d)}(3) \left( \sqrt{d^{\alpha-1}}\hat{\eta}_{Yi}^{(d)}(1) \right),
\end{align*}
which implies that,
\begin{align*}
\sqrt{d^{\alpha}}\hat{\eta}_{X1}^{(d)}(3)=\frac{d^{\alpha}} {\hat{\lambda}_{XY1}^{(d)} \hat{\eta}_{Y1}^{(d)}(1)} \left( \frac{\tau_{X}\langle z_{X3\bullet}, d_{1} \rangle}{n\sqrt{d^{\alpha}}}-\sum_{i=2}^{n} \frac{\hat{\lambda}_{XYi}^{(d)}}{\sqrt{d^{2\alpha-1}}} \hat{\eta}_{Xi}^{(d)}(3) \left( \sqrt{d^{\alpha-1}}\hat{\eta}_{Yi}^{(d)}(1) \right) \right).
\end{align*}
In a similar argument as above, it can be shown that,
\begin{align*}
\sqrt{d^{\alpha}}\hat{\eta}_{X1}^{(d)}(j)&=\frac{d^{\alpha}} {\hat{\lambda}_{XY1}^{(d)} \hat{\eta}_{Y1}^{(d)}(1)} \left( \frac{\tau_{X}\langle z_{Xj\bullet}, d_{1} \rangle}{n\sqrt{d^{\alpha}}}-\sum_{i=2}^{n} \frac{\hat{\lambda}_{XYi}^{(d)}}{\sqrt{d^{2\alpha-1}}} \hat{\eta}_{Xi}^{(d)}(j) \left( \sqrt{d^{\alpha-1}}\hat{\eta}_{Yi}^{(d)}(1) \right) \right),\,\, j=4,5,\dots,d,\\[6pt]
\sqrt{d^{\alpha}}\hat{\eta}_{Y1}^{(d)}(k)&=\frac{d^{\alpha}} {\hat{\lambda}_{XY1}^{(d)} \hat{\eta}_{X1}^{(d)}(1)} \left( \frac{\tau_{Y}\langle c_{1}, z_{Yk\bullet} \rangle}{n\sqrt{d^{\alpha}}}-\sum_{i=2}^{n} \frac{\hat{\lambda}_{XYi}^{(d)}}{\sqrt{d^{2\alpha-1}}} \left( \sqrt{d^{\alpha-1}}\hat{\eta}_{Xi}^{(d)}(1) \right) \hat{\eta}_{Yi}^{(d)}(k) \right),\,\, k=3,4,\dots,d.
\end{align*}
Then, the term (2) goes as follows,
\begin{align*}
(2)&=2\sum_{i=2}^{n} \frac{\hat{\lambda}_{XY1}^{(d)}}{d} \frac{\hat{\lambda}_{XYi}^{(d)}}{d} \left\langle \tilde{\eta}_{X1}^{(d)}, \tilde{\eta}_{Xi}^{(d)} \right\rangle \left\langle \tilde{\eta}_{Y1}^{(d)}, \tilde{\eta}_{Yi}^{(d)} \right\rangle\\
&=2 \frac{\hat{\lambda}_{XY1}^{(d)}}{d} \sum_{i=2}^{n} \frac{\hat{\lambda}_{XYi}^{(d)}}{d} \left( \sum_{j=2}^{d} \hat{\eta}_{X1}^{(d)}(j) \hat{\eta}_{Xi}^{(d)}(j) \right) \left( \sum_{k=2}^{d} \hat{\eta}_{Y1}^{(d)}(k) \hat{\eta}_{Yi}^{(d)}(k) \right)\\
&=2 \frac{\hat{\lambda}_{XY1}^{(d)}}{d} \hat{\eta}_{X1}^{(d)}(2) \hat{\eta}_{Y1}^{(d)}(2) \sum_{i=2}^{n} \frac{\hat{\lambda}_{XYi}^{(d)}}{d} \hat{\eta}_{Xi}^{(d)}(2) \hat{\eta}_{Yi}^{(d)}(2)\\
&\quad +2 \frac{\hat{\lambda}_{XY1}^{(d)}}{d} \hat{\eta}_{X1}^{(d)}(2) \sum_{i=2}^{n} \frac{\hat{\lambda}_{XYi}^{(d)}}{d} \hat{\eta}_{Xi}^{(d)}(2) \left( \sum_{k=3}^{d} \hat{\eta}_{Y1}^{(d)}(k) \hat{\eta}_{Yi}^{(d)}(k) \right)\\
&\quad +2 \frac{\hat{\lambda}_{XY1}^{(d)}}{d} \hat{\eta}_{Y1}^{(d)}(2) \sum_{i=2}^{n} \frac{\hat{\lambda}_{XYi}^{(d)}}{d} \hat{\eta}_{Yi}^{(d)}(2) \left( \sum_{k=3}^{d} \hat{\eta}_{X1}^{(d)}(k) \hat{\eta}_{Xi}^{(d)}(k) \right)\\
&\quad +2 \frac{\hat{\lambda}_{XY1}^{(d)}}{d} \sum_{i=2}^{n} \frac{\hat{\lambda}_{XYi}^{(d)}}{d} \left( \sum_{j=3}^{d} \hat{\eta}_{X1}^{(d)}(j) \hat{\eta}_{Xi}^{(d)}(j) \right) \left( \sum_{k=3}^{d} \hat{\eta}_{Y1}^{(d)}(k) \hat{\eta}_{Yi}^{(d)}(k) \right).
\end{align*}
Using Lemma~\ref{lem4}, ~\ref{lem5} and the fact that only a finite number of entries of $\hat{\eta}_{Xi}^{(d)}$ and $\hat{\eta}_{Yi}^{(d)}$, for $i=2,3,\dots,n$, can be of magnitude ($\asymp 1/d^{a}$), for $0 \le a<1/2$, the first three terms in the last part of the above equalities becomes,
\begin{align*}
&2 \frac{\hat{\lambda}_{XY1}^{(d)}}{d} \hat{\eta}_{X1}^{(d)}(2) \hat{\eta}_{Y1}^{(d)}(2) \sum_{i=2}^{n} \frac{\hat{\lambda}_{XYi}^{(d)}}{d} \hat{\eta}_{Xi}^{(d)}(2) \hat{\eta}_{Yi}^{(d)}(2)\\
&\quad = \frac{2}{d} \frac{\hat{\lambda}_{XY1}^{(d)}}{d^{\alpha}} \left( \sqrt{d^{\alpha}} \hat{\eta}_{X1}^{(d)}(2) \right) \left( \sqrt{d^{\alpha}} \hat{\eta}_{Y1}^{(d)}(2) \right) \sum_{i=2}^{n} \frac{\hat{\lambda}_{XYi}^{(d)}}{d} \hat{\eta}_{Xi}^{(d)}(2) \hat{\eta}_{Yi}^{(d)}(2) \mathop{\longrightarrow}_{d \to \infty}^{p} 0,\\[11pt]
&2 \frac{\hat{\lambda}_{XY1}^{(d)}}{d} \hat{\eta}_{X1}^{(d)}(2) \sum_{i=2}^{n} \frac{\hat{\lambda}_{XYi}^{(d)}}{d} \hat{\eta}_{Xi}^{(d)}(2) \left( \sum_{k=3}^{d} \hat{\eta}_{Y1}^{(d)}(k) \hat{\eta}_{Yi}^{(d)}(k) \right)\\
&\quad = \frac{2}{d} \frac{\hat{\lambda}_{XY1}^{(d)}}{d^{\alpha}} \left( \sqrt{d^{\alpha}} \hat{\eta}_{X1}^{(d)}(2) \right) \sum_{i=2}^{n} \frac{\hat{\lambda}_{XYi}^{(d)}}{d} \left( \sqrt{d} \hat{\eta}_{Xi}^{(d)}(2) \right) \left( \sum_{k=3}^{d} \frac{\left( \sqrt{d^{\alpha}} \hat{\eta}_{Y1}^{(d)}(k) \right) \left( \sqrt{d} \hat{\eta}_{Yi}^{(d)}(k) \right)}{d} \right) \mathop{\longrightarrow}_{d \to \infty}^{p} 0,\\[11pt]
&2 \frac{\hat{\lambda}_{XY1}^{(d)}}{d} \hat{\eta}_{Y1}^{(d)}(2) \sum_{i=2}^{n} \frac{\hat{\lambda}_{XYi}^{(d)}}{d} \hat{\eta}_{Yi}^{(d)}(2) \left( \sum_{k=3}^{d} \hat{\eta}_{X1}^{(d)}(k) \hat{\eta}_{Xi}^{(d)}(k) \right)\\
&\quad = \frac{2}{d} \frac{\hat{\lambda}_{XY1}^{(d)}}{d^{\alpha}} \left( \sqrt{d^{\alpha}} \hat{\eta}_{Y1}^{(d)}(2) \right) \sum_{i=2}^{n} \frac{\hat{\lambda}_{XYi}^{(d)}}{d} \left( \sqrt{d} \hat{\eta}_{Yi}^{(d)}(2) \right) \left( \sum_{k=3}^{d} \frac{\left( \sqrt{d^{\alpha}} \hat{\eta}_{Y1}^{(d)}(k) \right) \left( \sqrt{d} \hat{\eta}_{Yi}^{(d)}(k) \right)}{d} \right) \mathop{\longrightarrow}_{d \to \infty}^{p} 0.
\end{align*}
The fourth term expands as,
\begin{align*}
&2 \frac{\hat{\lambda}_{XY1}^{(d)}}{d} \sum_{i=2}^{n} \frac{\hat{\lambda}_{XYi}^{(d)}}{d} \left( \sum_{j=3}^{d} \hat{\eta}_{X1}^{(d)}(j) \hat{\eta}_{Xi}^{(d)}(j) \right) \left( \sum_{k=3}^{d} \hat{\eta}_{Y1}^{(d)}(k) \hat{\eta}_{Yi}^{(d)}(k) \right)\\
&\quad =2 \frac{\hat{\lambda}_{XY1}^{(d)}}{d^{\alpha}} \frac{1}{d} \sum_{i=2}^{n} \frac{\hat{\lambda}_{XYi}^{(d)}}{d} \left( \sum_{j=3}^{d} \sqrt{d^{\alpha}} \hat{\eta}_{X1}^{(d)}(j) \hat{\eta}_{Xi}^{(d)}(j) \right) \left( \sum_{k=3}^{d} \sqrt{d^{\alpha}} \hat{\eta}_{Y1}^{(d)}(k) \hat{\eta}_{Yi}^{(d)}(k) \right)\\
&\quad =2 \frac{\hat{\lambda}_{XY1}^{(d)}}{d^{\alpha}} \frac{1}{d} \sum_{i=2}^{n} \frac{\hat{\lambda}_{XYi}^{(d)}}{d}\\
&\quad\quad \times \left( \sum_{j=3}^{d} \frac{d^{\alpha}} {\hat{\lambda}_{XY1}^{(d)} \hat{\eta}_{Y1}^{(d)}(1)} \left( \frac{\tau_{X}\langle z_{Xj\bullet}, d_{1} \rangle}{n\sqrt{d^{\alpha}}}-\sum_{l=2}^{n} \frac{\hat{\lambda}_{XYl}^{(d)}}{\sqrt{d^{2\alpha-1}}} \hat{\eta}_{Xl}^{(d)}(j) \left( \sqrt{d^{\alpha-1}}\hat{\eta}_{Yl}^{(d)}(1) \right) \right) \hat{\eta}_{Xi}^{(d)}(j) \right)\\
&\quad\quad \times \left( \sum_{k=3}^{d} \frac{d^{\alpha}} {\hat{\lambda}_{XY1}^{(d)} \hat{\eta}_{X1}^{(d)}(1)} \left( \frac{\tau_{Y}\langle c_{1}, z_{Yk\bullet} \rangle}{n\sqrt{d^{\alpha}}}-\sum_{l=2}^{n} \frac{\hat{\lambda}_{XYl}^{(d)}}{\sqrt{d^{2\alpha-1}}} \left( \sqrt{d^{\alpha-1}}\hat{\eta}_{Xl}^{(d)}(1) \right) \hat{\eta}_{Yl}^{(d)}(k) \right) \hat{\eta}_{Yi}^{(d)}(k) \right).
\end{align*}
Using the the Cauchy-Schwarz inequality on $\langle \hat{\eta}_{Xi}^{(d)}, \hat{\eta}_{Yj}^{(d)} \rangle$, for $i,j=2,3,\dots,n$, the law of large number and the fact that $\alpha > 1$ and only a finite number of entries of $\hat{\eta}_{Xi}^{(d)}$ and $\hat{\eta}_{Yi}^{(d)}$, for $i=2,3,\dots,n$, can be of magnitude ($\asymp 1/d^{a}$), for $0 \le a<1/2$,, it is not hard to see that, 
\begin{align*}
&2 \frac{\hat{\lambda}_{XY1}^{(d)}}{d^{\alpha}} \frac{1}{d} \sum_{i=2}^{n} \left[ \frac{\hat{\lambda}_{XYi}^{(d)}}{d} \left( \sum_{j=3}^{d} \frac{d^{\alpha}} {\hat{\lambda}_{XY1}^{(d)} \hat{\eta}_{Y1}^{(d)}(1)} \left( -\sum_{l=2}^{n} \frac{\hat{\lambda}_{XYl}^{(d)}}{\sqrt{d^{2\alpha-1}}} \hat{\eta}_{Xl}^{(d)}(j) \left( \sqrt{d^{\alpha-1}}\hat{\eta}_{Yl}^{(d)}(1) \right) \right) \hat{\eta}_{Xi}^{(d)}(j) \right) \right.\\
&\left. \quad \times \left( \sum_{k=3}^{d} \frac{d^{\alpha}} {\hat{\lambda}_{XY1}^{(d)} \hat{\eta}_{X1}^{(d)}(1)} \left( -\sum_{l=2}^{n} \frac{\hat{\lambda}_{XYl}^{(d)}}{\sqrt{d^{2\alpha-1}}} \left( \sqrt{d^{\alpha-1}}\hat{\eta}_{Xl}^{(d)}(1) \right) \hat{\eta}_{Yl}^{(d)}(k) \right) \hat{\eta}_{Yi}^{(d)}(k) \right) \right]\\
&\,\, =2 \frac{\hat{\lambda}_{XY1}^{(d)}}{d^{\alpha}} \frac{1}{d^{\alpha-1}} \left( \frac{d^{\alpha}}{\hat{\lambda}_{XY1}^{(d)} \hat{\eta}_{Y1}^{(d)}(1)} \right)^{2} \sum_{i=2}^{n} \left[ \frac{\hat{\lambda}_{XYi}^{(d)}}{d} \left( \sum_{j=3}^{d} \left( -\sum_{l=2}^{n} \frac{\hat{\lambda}_{XYl}^{(d)}}{d} \hat{\eta}_{Xl}^{(d)}(j) \left( \sqrt{d^{\alpha-1}}\hat{\eta}_{Yl}^{(d)}(1) \right) \right) \hat{\eta}_{Xi}^{(d)}(j) \right) \right.\\
&\left. \quad \times \left( \sum_{k=3}^{d} \left( -\sum_{l=2}^{n} \frac{\hat{\lambda}_{XYl}^{(d)}}{d} \left( \sqrt{d^{\alpha-1}}\hat{\eta}_{Xl}^{(d)}(1) \right) \hat{\eta}_{Yl}^{(d)}(k) \right) \hat{\eta}_{Yi}^{(d)}(k) \right) \right] \mathop{\longrightarrow}_{d \to \infty}^{p} 0,\\[13pt]
&2 \frac{\hat{\lambda}_{XY1}^{(d)}}{d^{\alpha}} \frac{1}{d} \sum_{i=2}^{n} \left[ \frac{\hat{\lambda}_{XYi}^{(d)}}{d} \sum_{j=3}^{d} \left( \frac{ d^{\alpha}} {\hat{\lambda}_{XY1}^{(d)} \hat{\eta}_{Y1}^{(d)}(1)} \frac{\tau_{X}\langle z_{Xj\bullet}, d_{1} \rangle}{n\sqrt{d^{\alpha}}} \hat{\eta}_{Xi}^{(d)}(j) \right) \right.\\
&\left. \quad \times \left( \sum_{k=3}^{d} \frac{d^{\alpha}} {\hat{\lambda}_{XY1}^{(d)} \hat{\eta}_{X1}^{(d)}(1)} \left( -\sum_{l=2}^{n} \frac{\hat{\lambda}_{XYl}^{(d)}}{\sqrt{d^{2\alpha-1}}} \left( \sqrt{d^{\alpha-1}}\hat{\eta}_{Xl}^{(d)}(1) \right) \hat{\eta}_{Yi}^{(d)}(k) \right) \hat{\eta}_{Yl}^{(d)}(k) \right) \right]\\
&\,\, =2 \frac{\hat{\lambda}_{XY1}^{(d)}}{d^{\alpha}} \frac{1}{\sqrt{d^{\alpha-1}}} \left( \frac{d^{\alpha}}{\hat{\lambda}_{XY1}^{(d)} \hat{\eta}_{Y1}^{(d)}(1)} \right)^{2} \sum_{i=2}^{n} \left[ \frac{\hat{\lambda}_{XYi}^{(d)}}{d} \sum_{j=3}^{d} \left( \frac{\tau_{X}\langle z_{Xj\bullet}, d_{1} \rangle}{n\sqrt{d^{\alpha}}} \frac{\hat{ \sqrt{d}\eta}_{Xi}^{(d)}(j)}{d} \right) \right.\\
&\left. \quad \times \left( \sum_{k=3}^{d} \left( -\sum_{l=2}^{n} \frac{\hat{\lambda}_{XYl}^{(d)}}{d} \left( \sqrt{d^{\alpha-1}}\hat{\eta}_{Xl}^{(d)}(1) \right) \hat{\eta}_{Yl}^{(d)}(k) \right) \hat{\eta}_{Yi}^{(d)}(k) \right) \right] \mathop{\longrightarrow}_{d \to \infty}^{p} 0,\\[13pt]
&2 \frac{\hat{\lambda}_{XY1}^{(d)}}{d^{\alpha}} \frac{1}{d} \sum_{i=2}^{n} \left[ \frac{\hat{\lambda}_{XYi}^{(d)}}{d} \sum_{j=3}^{d} \left( \frac{ d^{\alpha}} {\hat{\lambda}_{XY1}^{(d)} \hat{\eta}_{Y1}^{(d)}(1)} \frac{\tau_{Y}\langle c_{1}, z_{Yk\bullet} \rangle}{n\sqrt{d^{\alpha}}} \hat{\eta}_{Yi}^{(d)}(j) \right) \right.\\
&\left. \quad \times \left( \sum_{k=3}^{d} \frac{d^{\alpha}} {\hat{\lambda}_{XY1}^{(d)} \hat{\eta}_{X1}^{(d)}(1)} \left( -\sum_{l=2}^{n} \frac{\hat{\lambda}_{XYl}^{(d)}}{\sqrt{d^{2\alpha-1}}} \left( \sqrt{d^{\alpha-1}}\hat{\eta}_{Xl}^{(d)}(1) \right) \hat{\eta}_{Yl}^{(d)}(k) \right) \hat{\eta}_{Yi}^{(d)}(k) \right)\right] 
\mathop{\longrightarrow}_{d \to \infty}^{p} 0,\\[13pt]
&2 \frac{\hat{\lambda}_{XY1}^{(d)}}{d^{\alpha}} \frac{1}{d} \sum_{i=2}^{n} \frac{\hat{\lambda}_{XYi}^{(d)}}{d} \left( \sum_{j=3}^{d} \frac{d^{\alpha}} {\hat{\lambda}_{XY1}^{(d)} \hat{\eta}_{Y1}^{(d)}(1)} \left( \frac{\tau_{X}\langle z_{Xj\bullet}, d_{1} \rangle}{n\sqrt{d^{\alpha}}} \right) \hat{\eta}_{Xi}^{(d)}(j) \right)\\
&\quad \times \left( \sum_{k=3}^{d} \frac{d^{\alpha}} {\hat{\lambda}_{XY1}^{(d)} \hat{\eta}_{X1}^{(d)}(1)} \left( \frac{\tau_{Y}\langle c_{1}, z_{Yk\bullet} \rangle}{n\sqrt{d^{\alpha}}} \right) \hat{\eta}_{Yi}^{(d)}(k) \right)\\
&\,\, =2 \frac{\hat{\lambda}_{XY1}^{(d)}}{d^{\alpha}} \left( \frac{d^{\alpha}} {\hat{\lambda}_{XY1}^{(d)} \hat{\eta}_{Y1}^{(d)}(1)} \right)^{2} \sum_{i=2}^{n} \left[ \frac{\hat{\lambda}_{XYi}^{(d)}}{d} \left( \sum_{j=3}^{d} \left( \frac{\tau_{X}\langle z_{Xj\bullet}, d_{1} \rangle}{n\sqrt{d^{\alpha}}} \right) \frac{\sqrt{d}\hat{\eta}_{Xi}^{(d)}(j)}{d} \right) \right.\\
&\left. \quad \times \left( \sum_{k=3}^{d} \left( \frac{\tau_{Y}\langle c_{1}, z_{Yk\bullet} \rangle}{n\sqrt{d^{\alpha}}} \right) \frac{\sqrt{d} \hat{\eta}_{Yi}^{(d)}(k)}{d} \right) \right] \mathop{\longrightarrow}_{d \to \infty}^{p} 0.
\end{align*}
We now prove that the term (1) converges in probability to $\tau_{X}^{2}\tau_{Y}^{2}$ as $d \to \infty$,
\begin{align*}
(1)=\left( \frac{\hat{\lambda}_{XY1}^{(d)}}{d} \right)^{2} \sum_{j=2}^{d} \left( \hat{\eta}_{X1}^{(d)}(j) \right)^{2} \sum_{k=2}^{d} \left( \hat{\eta}_{Y1}^{(d)}(k) \right)^{2}
\mathop{\longrightarrow}_{d \to \infty}^{p} \tau_{X}^{2}\tau_{Y}^{2}.
\end{align*}
which implies that $\hat{\lambda}_{XYi}^{(d)}/d$, for $i=2,3,\dots,n$, are squeezed, by the condition of~(\ref{sec3:equ36}), to converge in probability to 0 as $d \to \infty$. The term (1) can be written as,
\begin{align*}
&\left( \frac{\hat{\lambda}_{XY1}^{(d)}}{d} \right)^{2} \sum_{j=2}^{d} \left( \hat{\eta}_{X1}^{(d)}(j) \right)^{2} \sum_{k=2}^{d} \left( \hat{\eta}_{Y1}^{(d)}(k) \right)^{2}\\
&\,\, =\left( \frac{\hat{\lambda}_{XY1}^{(d)}}{d} \right)^{2} \left( \hat{\eta}_{X1}^{(d)}(2) \right)^{2} \left( \hat{\eta}_{Y1}^{(d)}(2) \right)^{2}
+\left( \frac{\hat{\lambda}_{XY1}^{(d)}}{d} \right)^{2} \left( \hat{\eta}_{X1}^{(d)}(2) \right)^{2} \sum_{k=3}^{d} \left( \hat{\eta}_{Y1}^{(d)}(k) \right)^{2}\\
&\quad +\left( \frac{\hat{\lambda}_{XY1}^{(d)}}{d} \right)^{2} \left( \hat{\eta}_{Y1}^{(d)}(2) \right)^{2} \sum_{j=3}^{d} \left( \hat{\eta}_{X1}^{(d)}(j) \right)^{2}
+\left( \frac{\hat{\lambda}_{XY1}^{(d)}}{d} \right)^{2} \sum_{j=3}^{d} \left( \hat{\eta}_{X1}^{(d)}(j) \right)^{2} \sum_{k=3}^{d} \left( \hat{\eta}_{Y1}^{(d)}(k) \right)^{2}.
\end{align*}
By the use of Lemma~\ref{lem4} and~\ref{lem5},
\begin{align*}
\left( \frac{\hat{\lambda}_{XY1}^{(d)}}{d} \right)^{2} \left( \hat{\eta}_{X1}^{(d)}(2) \right)^{2} \left( \hat{\eta}_{Y1}^{(d)}(2) \right)^{2}
&=\left( \frac{\hat{\lambda}_{XY1}^{(d)}}{d^{\alpha}} \right)^{2} \left( \sqrt{d^{\alpha}}\hat{\eta}_{X1}^{(d)}(2) \right)^{2} \frac{1}{d^{2}}\left( \sqrt{d^{\alpha}}\hat{\eta}_{Y1}^{(d)}(2) \right)^{2} \mathop{\longrightarrow}_{d \to \infty}^{p} 0,\\[7pt]
\left( \frac{\hat{\lambda}_{XY1}^{(d)}}{d} \right)^{2} \left( \hat{\eta}_{X1}^{(d)}(2) \right)^{2} \sum_{k=3}^{d} \left( \hat{\eta}_{Y1}^{(d)}(k) \right)^{2}
&=\left( \frac{\hat{\lambda}_{XY1}^{(d)}}{d^{\alpha}} \right)^{2} \frac{\left( \sqrt{d^{\alpha}}\hat{\eta}_{X1}^{(d)}(2) \right)^{2}}{d} \sum_{k=3}^{n}\frac{\left( \sqrt{d^{\alpha}}\hat{\eta}_{Y1}^{(d)}(k) \right)^{2}}{d}\\
&=\left( \frac{\hat{\lambda}_{XY1}^{(d)}}{d^{\alpha}} \right)^{2} \frac{1}{d} \frac{\sum_{k=3}^{n}\left( \sqrt{d^{\alpha}}\hat{\eta}_{X1}^{(d)}(2) \right)^{2} \left( \sqrt{d^{\alpha}}\hat{\eta}_{Y1}^{(d)}(k) \right)^{2}}{d}
\mathop{\longrightarrow}_{d \to \infty}^{p} 0,\\[7pt]
\left( \frac{\hat{\lambda}_{XY1}^{(d)}}{d} \right)^{2} \left( \hat{\eta}_{Y1}^{(d)}(2) \right)^{2} \sum_{k=3}^{d} \left( \hat{\eta}_{X1}^{(d)}(k) \right)^{2}
&=\left( \frac{\hat{\lambda}_{XY1}^{(d)}}{d^{\alpha}} \right)^{2} \frac{\left( \sqrt{d^{\alpha}}\hat{\eta}_{Y1}^{(d)}(2) \right)^{2}}{d} \sum_{k=3}^{n}\frac{\left( \sqrt{d^{\alpha}}\hat{\eta}_{X1}^{(d)}(k) \right)^{2}}{d}\\
&=\left( \frac{\hat{\lambda}_{XY1}^{(d)}}{d^{\alpha}} \right)^{2} \frac{1}{d} \frac{\sum_{k=3}^{n}\left( \sqrt{d^{\alpha}}\hat{\eta}_{Y1}^{(d)}(2) \right)^{2} \left( \sqrt{d^{\alpha}}\hat{\eta}_{X1}^{(d)}(k) \right)^{2}}{d}
\mathop{\longrightarrow}_{d \to \infty}^{p} 0.
\end{align*}
Note that the magnitudes of $\tau_{X}\langle z_{X3\bullet}, d_{1} \rangle/n\sqrt{d^{\alpha}}$ and $\tau_{Y}\langle c_{1}, z_{Y3\bullet} \rangle/n\sqrt{d^{\alpha}}$ are of ($\asymp 1$) from~(\ref{sec3:equ18}). Using Lemma~\ref{lem4},~\ref{lem5} and~(\ref{sec3:equ35}),
\begin{align*}
&\left( \frac{\hat{\lambda}_{XY1}^{(d)}}{d} \right)^{2} \sum_{j=3}^{d} \left( \hat{\eta}_{X1}^{(d)}(j) \right)^{2} \sum_{k=3}^{d} \left( \hat{\eta}_{Y1}^{(d)}(k) \right)^{2}\\
&\,\, =\left( \frac{\hat{\lambda}_{XY1}^{(d)}}{d^{\alpha}} \right)^{2} \frac{1}{d^{2}} \sum_{j=3}^{d} \left( \sqrt{d^{\alpha}}\hat{\eta}_{X1}^{(d)}(j) \right)^{2} \sum_{k=3}^{d} \left( \sqrt{d^{\alpha}}\hat{\eta}_{Y1}^{(d)}(k) \right)^{2}\\
&\,\, =\left( \frac{d^{\alpha}}{\hat{\lambda}_{XY1}^{(d)}} \right)^{2} \frac{1}{d^{2}} \sum_{j=3}^{d} \left( \frac{\tau_{X}\langle z_{Xj\bullet}, d_{1} \rangle}{n\sqrt{d^{\alpha}}}-\sum_{i=2}^{n} \frac{\hat{\lambda}_{XYi}^{(d)}}{\sqrt{d^{2\alpha-1}}} \hat{\eta}_{Xi}^{(d)}(j) \left( \sqrt{d^{\alpha-1}}\hat{\eta}_{Yi}^{(d)}(1) \right) \right)^{2}\\
&\quad \quad \times \sum_{k=3}^{d} \left( \frac{\tau_{Y}\langle c_{1}, z_{Yk\bullet} \rangle}{n\sqrt{d^{\alpha}}}-\sum_{i=2}^{n} \frac{\hat{\lambda}_{XYi}^{(d)}}{\sqrt{d^{2\alpha-1}}} \left( \sqrt{d^{\alpha-1}}\hat{\eta}_{Xi}^{(d)}(1) \right) \hat{\eta}_{Yi}^{(d)}(k) \right)^{2}\\
&\,\, =\left( \frac{d^{\alpha}}{\hat{\lambda}_{XY1}^{(d)}} \right)^{2} \frac{1}{d^{2}} \sum_{j=3}^{d} \left( \frac{\tau_{X}\langle z_{Xj\bullet}, d_{1} \rangle}{n\sqrt{d^{\alpha}}} \right)^{2} \sum_{k=3}^{d} \left( \frac{\tau_{Y}\langle c_{1}, z_{Yk\bullet} \rangle}{n\sqrt{d^{\alpha}}} \right)^{2}\\
&\quad -4\left( \frac{d^{\alpha}}{\hat{\lambda}_{XY1}^{(d)}} \right)^{2} \frac{1}{d^{2\alpha-1}} \left( \sum_{j=3}^{d} \frac{\tau_{X}\langle z_{Xj\bullet}, d_{1} \rangle}{n\sqrt{d^{\alpha}}}\sum_{i=2}^{n} \frac{\hat{\lambda}_{XYi}^{(d)}}{d} \hat{\eta}_{Xi}^{(d)}(j) \left( \sqrt{d^{\alpha-1}}\hat{\eta}_{Yi}^{(d)}(1) \right) \right)\\
&\quad \quad \times \left( \sum_{k=3}^{d} \frac{\tau_{Y}\langle c_{1}, z_{Yk\bullet} \rangle}{n\sqrt{d^{\alpha}}}\sum_{i=2}^{n} \frac{\hat{\lambda}_{XYi}^{(d)}}{d} \left( \sqrt{d^{\alpha-1}}\hat{\eta}_{Xi}^{(d)}(1) \right) \hat{\eta}_{Yi}^{(d)}(k) \right)\\
&\quad +\left( \frac{d^{\alpha}}{\hat{\lambda}_{XY1}^{(d)}} \right)^{2} \frac{1}{d^{4(\alpha-1)}} \sum_{j=3}^{d} \left( \sum_{i=2}^{n} \frac{\hat{\lambda}_{XYi}^{(d)}}{d} \hat{\eta}_{Xi}^{(d)}(j) \left( \sqrt{d^{\alpha-1}}\hat{\eta}_{Yi}^{(d)}(1) \right) \right)^{2}\\
&\quad \quad \times \sum_{k=3}^{d} \left( \sum_{i=2}^{n} \frac{\hat{\lambda}_{XYi}^{(d)}}{d} \left( \sqrt{d^{\alpha-1}}\hat{\eta}_{Xi}^{(d)}(1) \right) \hat{\eta}_{Yi}^{(d)}(k) \right)^{2}.
\end{align*}
The second term in the last equality above converges in probability to 0. Using the fact that $\alpha > 1$ and only a finite number of entries of $\hat{\eta}_{Xi}^{(d)}$ and $\hat{\eta}_{Yi}^{(d)}$, for $i=2,3,\dots,n$, can be of magnitude ($\asymp 1/d^{a}$), for $0 \le a<1/2$,
\begin{align*}
&\frac{1}{d^{2\alpha-2}} \left( \sum_{i=2}^{n} \frac{\hat{\lambda}_{XYi}^{(d)}}{d} \left( \sqrt{d^{\alpha-1}}\hat{\eta}_{Yi}^{(d)}(1) \right) \sum_{j=3}^{d} \frac{\tau_{X}\langle z_{Xj\bullet}, d_{1} \rangle}{n\sqrt{d^{\alpha}}} \frac{\sqrt{d}\hat{\eta}_{Xi}^{(d)}(j)}{d} \right)\\
&\,\, \times \left( \sum_{i=2}^{n} \frac{\hat{\lambda}_{XYi}^{(d)}}{d} \left( \sqrt{d^{\alpha-1}}\hat{\eta}_{Xi}^{(d)}(1) \right) \sum_{k=3}^{d} \frac{\tau_{Y}\langle c_{1}, z_{Yk\bullet} \rangle}{n\sqrt{d^{\alpha}}} \frac{\sqrt{d}\hat{\eta}_{Yi}^{(d)}(k)}{d} \right) \mathop{\longrightarrow}_{d \to \infty}^{p} 0.
\end{align*}
The third term also converges in probability to 0. It easy to see the result with the following equivalent form, noting that each inner product is bounded by 1 by the Cauchy-Schwarz inequality and that $\alpha > 1$,
\begin{align*}
&\frac{1}{d^{4(\alpha-1)}} \left( \sum_{i,j=2}^{n} \left( \frac{\hat{\lambda}_{XYi}^{(d)}}{d} \right)^{2} \left( \frac{\hat{\lambda}_{XYj}^{(d)}}{d} \right)^{2} \left( \sqrt{d^{\alpha-1}}\hat{\eta}_{Yi}^{(d)}(1) \right)^{2} \left( \sqrt{d^{\alpha-1}}\hat{\eta}_{Xj}^{(d)}(1) \right)^{2} \sum_{k=3}^{d} \left( \hat{\eta}_{Xi}^{(d)}(k) \right)^{2} \left( \hat{\eta}_{Yj}^{(d)}(k) \right)^{2} \right)\\
&\,\, + \frac{4}{d^{4(\alpha-1)}} \left( \sum_{2 \le a,a' \le n, a \not= a'} \sum_{2 \le b,b' \le n, b \not= b'} \frac{\hat{\lambda}_{XYa}^{(d)}}{d} \frac{\hat{\lambda}_{XYa'}^{(d)}}{d} \frac{\hat{\lambda}_{XYb}^{(d)}}{d} \frac{\hat{\lambda}_{XYb'}^{(d)}}{d} \sqrt{d^{\alpha-1}}\hat{\eta}_{Ya}^{(d)}(1) \sqrt{d^{\alpha-1}}\hat{\eta}_{Ya'}^{(d)}(1) \right.\\ 
&\left. \quad \quad \quad \quad \quad \quad \times \sqrt{d^{\alpha-1}}\hat{\eta}_{Xb}^{(d)}(1) \sqrt{d^{\alpha-1}}\hat{\eta}_{Xb'}^{(d)}(1) \sum_{k=3}^{d} \left\langle \hat{\eta}_{Xa}^{(d)}(k), \hat{\eta}_{Xa'}^{(d)}(k) \right\rangle \left\langle \hat{\eta}_{Yb}^{(d)}(k), \hat{\eta}_{Yb'}^{(d)}(k) \right\rangle \right) \mathop{\longrightarrow}_{d \to \infty}^{p} 0.
\end{align*}
Now, look at the first term. Denote by $d_{1}(i)$ and $c_{1}(i)$ the $i$th entries of the vectors $d_{1}$ and $c_{1}$ given in~(\ref{sec3:equ18}). Then, recalling that $z_{Xij}^{2} \sim \chi_{1}^{2}$ (similarly $z_{Yij}^{2}$) and that $z_{Xki}$ and $z_{Xkj}$ are independent (similarly for $z_{Yki}$ and $z_{Ykj}$) and using the limiting quantity of the first sample singular value in Lemma~\ref{lem4},
\begin{align*}
\tau_{X}^{2} \tau_{Y}^{2} &\left( \frac{d^{\alpha}}{\hat{\lambda}_{XY1}^{(d)}} \right)^{2} \frac{1}{d^{2}} \sum_{j=3}^{d} \left( \frac{\langle z_{Xj\bullet}, d_{1} \rangle}{n\sqrt{d^{\alpha}}} \right)^{2} \sum_{k=3}^{d} \left( \frac{\langle c_{1}, z_{Yk\bullet} \rangle}{n\sqrt{d^{\alpha}}} \right)^{2}\\
&=\tau_{X}^{2} \tau_{Y}^{2} \frac{1}{n^{2}}\left( \frac{d^{\alpha}}{\hat{\lambda}_{XY1}^{(d)}} \right)^{2} \left(\sum_{i=1}^{n} \frac{d_{1}^{2}(i)}{d^{\alpha}} \sum_{j=3}^{d} \frac{z_{Xji}^{2}}{d}+2\sum_{1\le i,j \le n, i\not= j} \sum_{k=1}^{d} \frac{d_{1}(i)d_{1}(j)}{d^{\alpha}} \frac{z_{Xki}z_{Xkj}}{d} \right)\\
&\quad \times \left(\sum_{i=1}^{n} \frac{c_{1}^{2}(i)}{d^{\alpha}} \sum_{j=3}^{d} \frac{z_{Yji}^{2}}{d}+2\sum_{1\le i,j \le n, i\not= j} \sum_{k=1}^{d} \frac{c_{1}(i)c_{1}(j)}{d^{\alpha}} \frac{z_{Yki}z_{Ykj}}{d} \right)\\
& \mathop{\longrightarrow}_{d \to \infty}^{p} \tau_{X}^{2} \tau_{Y}^{2} \frac{1}{n^{2}} \frac{n^{2} \|c_{1}\|_{2}^{2} \|d_{1}\|_{2}^{2}}{\langle c_{1}, d_{1} \rangle^{2}}\\
&\ge \tau_{X}^{2} \tau_{Y}^{2},
\end{align*}
where the last inequality results from the Cauchy-Schwarz inequality of $\langle c_{1}, d_{1} \rangle^{2} \le \|c_{1}\|_{2}^{2} \|d_{1}\|_{2}^{2}$. Then, the condition of~(\ref{sec3:equ36}) completes the proof. 
\end{proof}

\subsubsection{Behavior of the sample covariance matrices}

We investigate the HDLSS asymptotic behavior of the sample covariance matrices $\hat{\mathbf{\Sigma}}^{(d)}_{X}$ and $\hat{\mathbf{\Sigma}}^{(d)}_{Y}$ given in~(\ref{sec3:equ19}), in specific, its sample eigenvalues and eigenvectors. Here, we only include the result of the analysis of $\hat{\mathbf{\Sigma}}^{(d)}_{X}$ as that of $\hat{\mathbf{\Sigma}}^{(d)}_{Y}$ is similar. The eigendecomposition of $\hat{\mathbf{\Sigma}}^{(d)}_{X}$ gives,
\begin{align}
\label{sec3:equ21}
\hat{\mathbf{\Sigma}}^{(d)}_{X} =  \sum_{i=1}^{n} \hat{\lambda}_{Xi}^{(d)} \hat{\xi}_{Xi}^{(d)} \left( \hat{\xi}_{Xi}^{(d)} \right)^{T},
\end{align}
where $\hat{\lambda}_{X1}^{(d)} \ge \hat{\lambda}_{X2}^{(d)} \ge \dots \ge \hat{\lambda}_{Xn}^{(d)} \ge 0$, $\| \hat{\xi}_{Xi}^{(d)} \|_{2} = 1$ and $\langle \hat{\xi}_{Xi}^{(d)}, \hat{\xi}_{Xj}^{(d)} \rangle = 0$ for $i \not= j$.
\begin{lem}[CCA HDLSS Asymptotic lemma 7.]
\label{lem7}
Let $C_{X}$ and $\mathbf{M}_{X}$ be,
\begin{align*}
C_{X}=\lim_{d\to\infty}\frac{\| c_{1} \|_{2}^{2}}{d^{\alpha}}, \,\, \mathbf{M}_{X}=
\begin{bmatrix}
C_{X} &\underset{1 \times (d-1)}{\mathbf{0}}\\
\underset{(d-1) \times 1}{\mathbf{0}} &\underset{(d-1) \times (d-1)}{\mathbf{0}}\\
\end{bmatrix}.
\end{align*}
where $c_{1}$ is defined in~(\ref{sec3:equ18}) and so $C_{X}$ is a non-degenerate random variable. Then, for $\alpha > 1$,
\begin{align*}
\left\| \frac{n\hat{\mathbf{\Sigma}}^{(d)}_{X}}{d^{\alpha}}-\mathbf{M}_{X} \right\|_{F}^{2} \mathop{\longrightarrow}_{d \to \infty}^{p} 0,
\end{align*}
\end{lem}
\begin{proof}
Let $\mathbf{M}_{X}^{(d)}$ be,
\begin{align*} 
\mathbf{M}_{X}^{(d)}=
\begin{bmatrix}
\frac{\| c_{1} \|_{2}^{2}}{d^{\alpha}} &\underset{1 \times (d-1)}{\mathbf{0}}\\
\underset{(d-1) \times 1}{\mathbf{0}} &\underset{(d-1) \times (d-1)}{\mathbf{0}}\\
\end{bmatrix}.
\end{align*}
It is obvious to see that,
\begin{align*}
\left\| \mathbf{M}_{X}^{(d)}-\mathbf{M}_{X} \right\|_{F}^{2} \mathop{\longrightarrow}_{d \to \infty}^{p} 0.
\end{align*}
Using the Cauchy-Schwarz inequality,
\begin{align*}
\left\| \frac{n\hat{\mathbf{\Sigma}}^{(d)}_{X}}{d^{\alpha}}-\mathbf{M}_{X}^{(d)} \right\|_{F}^{2} 
=&\frac{\langle c_{1}, c_{2} \rangle^{2}}{d^{2\alpha}}
+\frac{\langle c_{2}, c_{1} \rangle^{2}}{d^{2\alpha}}
+\frac{\langle c_{2}, c_{2} \rangle^{2}}{d^{2\alpha}}
+2\tau_{X}^{2}\sum_{i=3}^{d} \frac{\langle c_{1}, z_{Xi\bullet} \rangle^{2}}{d^{2\alpha}}\\
&+2\tau_{X}^{2}\sum_{i=3}^{d} \frac{\langle c_{2}, z_{Xi\bullet} \rangle^{2}}{d^{2\alpha}}
+\tau_{X}^{4}\sum_{i=3}^{d}\sum_{j=3}^{d} \frac{\langle z_{Xi\bullet}, z_{Xj\bullet} \rangle^{2}}{d^{2\alpha}}\\
\le &\frac{\langle c_{1}, c_{2} \rangle^{2}}{d^{2\alpha}}
+\frac{\langle c_{2}, c_{1} \rangle^{2}}{d^{2\alpha}}
+\frac{\langle c_{2}, c_{2} \rangle^{2}}{d^{2\alpha}}
+2\tau_{X}^{2}\sum_{i=3}^{d} \frac{\| c_{1} \|_{2}^{2} \| z_{Xi\bullet} \|_{2}^{2}}{d^{2\alpha}}\\
&+2\tau_{X}^{2}\sum_{i=3}^{d} \frac{\| c_{2} \|_{2}^{2} \| z_{Xi\bullet} \|_{2}^{2}}{d^{2\alpha}}
+\tau_{X}^{4}\sum_{i=3}^{d}\sum_{j=3}^{d} \frac{\| z_{Xi\bullet}\|_{2}^{2} \| z_{Xj\bullet} \|_{2}^{2}}{d^{2\alpha}}.
\end{align*}
Since $\langle c_{1}, c_{2} \rangle^{2}$, $\langle c_{2}, c_{1} \rangle^{2}$ are $O_{P}(\sqrt{d^{\alpha}})$, and $\langle c_{2}, c_{2} \rangle^{2}$ is $O_{P}(1)$,
\begin{align*}
\frac{\langle c_{1}, c_{2} \rangle^{2}}{d^{2\alpha}} \mathop{\longrightarrow}_{d \to \infty}^{p} 0, \,\, \frac{\langle c_{2}, c_{1} \rangle^{2}}{d^{2\alpha}} \mathop{\longrightarrow}_{d \to \infty}^{p} 0, \,\, \frac{\langle c_{2}, c_{2} \rangle^{2}}{d^{2\alpha}} \mathop{\longrightarrow}_{d \to \infty}^{p} 0.
\end{align*}
Note that $\| z_{Xi\bullet} \|_{2}^{2} \sim \chi_{n}^{2}$ and that $\| c_{1} \|_{2}^{2}$ and $\| c_{2} \|_{2}^{2}$ are of $O_{P}(d^{\alpha})$. By the law of large numbers,
\begin{align*}
2\tau_{X}^{2}\sum_{i=3}^{d} \frac{\| c_{1} \|_{2}^{2} \| z_{Xi\bullet} \|_{2}^{2}}{d^{2\alpha}}
&=\frac{\tau_{X}^{2}}{d^{\alpha-1}} \left \| \frac{d_{1}}{\sqrt{d^{\alpha}}} \right \|_{2}^{2} \sum_{i=3}^{d} \frac{\| z_{Xi\bullet} \|_{2}^{2}}{d}\\
&\mathop{\longrightarrow}_{d \to \infty}^{p} 0 \times O_{P}(1) \times E(\chi_{n}^{2}) = 0,\\
2\tau_{X}^{2}\sum_{i=3}^{d} \frac{\| c_{2} \|_{2}^{2} \| z_{Xi\bullet} \|_{2}^{2}}{d^{2\alpha}}
&=\frac{\tau_{X}^{2}}{d^{\alpha-1}} \left \| \frac{d_{2}}{\sqrt{d^{\alpha}}} \right \|_{2}^{2} \sum_{i=3}^{d} \frac{\| z_{Xi\bullet} \|_{2}^{2}}{d}\\
&\mathop{\longrightarrow}_{d \to \infty}^{p} 0 \times 0 \times E(\chi_{n}^{2}) = 0.
\end{align*}
Using the fact that $\alpha > 1$, $\| z_{Xi\bullet} \|_{2}^{2} \sim \chi_{n}^{2}$ and $\| z_{Yi\bullet} \|_{2}^{2} \sim \chi_{n}^{2}$  and applying the law of large numbers,
\begin{align*}
\tau_{X}^{4}\sum_{i=3}^{d}\sum_{j=3}^{d} \frac{\| z_{Xi\bullet}\|_{2}^{2} \| z_{Xj\bullet} \|_{2}^{2}}{d^{2\alpha}}&=\frac{\tau_{X}^{4}}{d^{2\alpha-2}} \sum_{i=3}^{d} \frac{\| z_{Xi\bullet} \|_{2}^{2}}{d} \sum_{j=3}^{d} \frac{\| z_{Xj\bullet} \|_{2}^{2}}{d}\\
&\mathop{\longrightarrow}_{d \to \infty}^{p} 0 \times n \times n = 0.
\end{align*}
Therefore,
\begin{align*}
\left\| \frac{n\hat{\mathbf{\Sigma}}^{(d)}_{X}}{d^{\alpha}}-\mathbf{M}_{X}^{(d)} \right\|_{F}^{2} \mathop{\longrightarrow}_{d \to \infty}^{p} 0.
\end{align*}
\end{proof}
\begin{lem}[CCA HDLSS Asymptotic lemma 8.]
\label{lem8}
Let $\mathbf{M}_{X}$ be the matrix defined in Lemma~\ref{lem7}. Let $\hat{\lambda}_{X1}^{(d)}$ and $\hat{\xi}_{X1}^{(d)}$ be the first sample eigenvalue and eigenvectors from~(\ref{sec3:equ21}). Then, for $\alpha > 1$, 
\begin{align*}
\left\| \frac{n\hat{\lambda}_{X1}^{(d)}}{d^{\alpha}} \hat{\xi}_{X1}^{(d)} \left( \hat{\xi}_{X1}^{(d)} \right)^{T}-\mathbf{M}_{X} \right\|_{F}^{2} \mathop{\longrightarrow}_{d \to \infty}^{p} 0.
\end{align*}
\end{lem}
\begin{proof}
Using $\mathbf{M}_{X}^{(d)}$ defined in Lemma~\ref{lem7} and the triangle inequality,
\begin{align*}
&\left \| \frac{n\hat{\lambda}_{X1}^{(d)}}{d^{\alpha}} \hat{\xi}_{X1}^{(d)} \left( \hat{\xi}_{X1}^{(d)} \right)^{T} - \mathbf{M}_{X}^{(d)} \right \|_{F}\\
&\quad \quad =\left \| \left( \frac{n\hat{\mathbf{\Sigma}}^{(d)}_{X}}{d^{\alpha}} - \mathbf{M}_{X}^{(d)} \right) - \left( \frac{n\hat{\mathbf{\Sigma}}^{(d)}_{X}}{d^{\alpha}} - \frac{n\hat{\lambda}_{X1}^{(d)}}{d^{\alpha}} \hat{\xi}_{X1}^{(d)} \left( \hat{\xi}_{X1}^{(d)} \right)^{T} \right) \right \|_{F}\\
&\quad \quad \le \left \| \frac{n\hat{\mathbf{\Sigma}}^{(d)}_{X}}{d^{\alpha}} - \mathbf{M}_{X}^{(d)} \right \|_{F} + \left \| \frac{n\hat{\mathbf{\Sigma}}^{(d)}_{X}}{d^{\alpha}} - \frac{\hat{n\lambda}_{X1}^{(d)}}{d^{\alpha}} \hat{\xi}_{X1}^{(d)} \left( \hat{\xi}_{X1}^{(d)} \right)^{T} \right \|_{F}.
\end{align*}
By Lemma~\ref{lem7},
\begin{align*}
\left \| \frac{n\hat{\mathbf{\Sigma}}^{(d)}_{X}}{d^{\alpha}} - \mathbf{M}_{X}^{(d)} \right \|_{F}^{2} \mathop{\longrightarrow}_{d \to \infty}^{p} 0.
\end{align*}
Write $\mathbf{M}_{X}^{(d)}$ as,
\begin{align*}
\mathbf{M}_{X}^{(d)} = \frac{\| c_{1} \|^{2}_{2}}{d^{\alpha}} e_{1}^{(d)} \left( e_{1}^{(d)} \right)^{T}.
\end{align*}
Since the first eigenvalue $\hat{\lambda}_{X1}^{(d)}$ and eigenvector $\hat{\xi}_{X1}^{(d)}$ provide the best rank-1 approximation to the matrix $\hat{\lambda}_{X1}^{(d)}/d^{\alpha}$,
\begin{align*}
\left\| \frac{\hat{\lambda}_{X1}^{(d)}}{d^{\alpha}} - \frac{\hat{\lambda}_{X1}^{(d)}}{d^{\alpha}} \hat{\xi}_{X1}^{(d)} \left( \hat{\xi}_{X1}^{(d)} \right)^{T} \right\|_{F}^{2} 
&\le \left \| \frac{\hat{\mathbf{\Sigma}}^{(d)}_{X}}{d^{\alpha}} - \frac{\| c_{1} \|^{2}_{2}}{nd^{\alpha}} e_{1}^{(d)} \left( e_{1}^{(d)} \right)^{T} \right \|_{F}^{2}\\
&\le \left \| \frac{\hat{\mathbf{\Sigma}}^{(d)}_{X}}{d^{\alpha}} - \frac{\mathbf{M}_{X}^{(d)}}{n} \right \|_{F}^{2}\\ 
&\mathop{\longrightarrow}_{d \to \infty}^{p} 0.
\end{align*}
\end{proof}
\begin{lem}[CCA HDLSS Asymptotic lemma 9.]
\label{lem9}
For $\alpha > 1$, the first eigenvalue $\hat{\lambda}_{X1}^{(d)}$ and eigenvectors $\hat{\xi}_{X1}^{(d)}$ converge in probability to the following quantities as $d \to \infty$,
\begin{align*}
\frac{n\hat{\lambda}_{X1}^{(d)}}{d^{\alpha}} \mathop{\longrightarrow}_{d \to \infty}^{p} C_{X}, \,\, \left\| \hat{\xi}_{X1}^{(d)}-e_{1}^{(d)} \right\|_{2} \mathop{\longrightarrow}_{d \to \infty}^{p} 0,
\end{align*}
where $C_{X}$ is defined in Lemma~\ref{lem7}.
\end{lem}
\begin{proof}
Let $\tilde{\xi}_{X1}^{(d)}$ be $\hat{\xi}_{X1}^{(d)}$ where their first entries are set to 0. Let $\hat{\xi}_{X1}^{(d)}(i)$ be the $i$th and $j$th entries of $\hat{\xi}_{X1}^{(d)}$. By Lemma~\ref{lem8},
\begin{align*}
\left( \frac{\hat{\lambda}_{X1}^{(d)}}{d^{\alpha}} \right)^{2} \left\| \tilde{\xi}_{X1}^{(d)} \right\|_{2}^{2} \left\| \tilde{\xi}_{X1}^{(d)} \right\|_{2}^{2}
&=\left( \frac{\hat{\lambda}_{X1}^{(d)}}{d^{\alpha}} \right)^{2} \sum_{i=2}^{d} \left( \hat{\xi}_{X1}^{(d)}(i) \right)^{2} \sum_{j=2}^{d} \left( \hat{\xi}_{X1}^{(d)}(j) \right)^{2}\\
&=\left( \frac{\hat{\lambda}_{X1}^{(d)}}{d^{\alpha}} \right)^{2} \sum_{i=2}^{d} \sum_{j=2}^{d}  \left( \hat{\xi}_{X1}^{(d)}(i) \right)^{2} \left( \hat{\xi}_{X1}^{(d)}(j) \right)^{2}\\
&=\left\| \frac{\hat{\lambda}_{X1}^{(d)}}{d^{\alpha}} \tilde{\xi}_{X1}^{(d)} \left( \tilde{\xi}_{X1}^{(d)} \right)^{T} \right\|_{F}^{2}\\ 
&\mathop{\longrightarrow}_{d \to \infty}^{p} 0.
\end{align*}
First, we show that $(\hat{\lambda}_{X1}^{(d)}/d^{\alpha})^{2} > 0$ in probability. Suppose that $(\hat{\lambda}_{X1}^{(d)}/d^{\alpha})^{2}$ converges in probability to 0. Then, noting that $\| \hat{\xi}_{X1}^{(d)} \|_{2}^{2} = 1$ and $\| \hat{\xi}_{X1}^{(d)} \|_{2}^{2} = 1$,
\begin{align*}
\left\| \frac{\hat{\lambda}_{X1}^{(d)}}{d^{\alpha}} \hat{\xi}_{X1}^{(d)} \left( \hat{\xi}_{X1}^{(d)} \right)^{T} \right\|_{F}^{2} &= \left( \frac{\hat{\lambda}_{X1}^{(d)}}{d^{\alpha}} \right)^{2} \sum_{i=1}^{d} \sum_{j=1}^{d}  \left( \hat{\xi}_{X1}^{(d)}(i) \right)^{2} \left( \hat{\xi}_{X1}^{(d)}(j) \right)^{2}\\
&=\left( \frac{\hat{\lambda}_{X1}^{(d)}}{d^{\alpha}} \right)^{2} \sum_{i=1}^{d} \left( \hat{\xi}_{X1}^{(d)} \right)^{2} \sum_{j=1}^{d} \left( \hat{\xi}_{X1}^{(d)}(j) \right)^{2}\\
&=\left( \frac{\hat{\lambda}_{X1}^{(d)}}{d^{\alpha}} \right)^{2} \left\| \hat{\xi}_{X1}^{(d)} \right\|_{2}^{2} \left\| \hat{\xi}_{X1}^{(d)} \right\|_{2}^{2}\\
&=\left( \frac{\hat{\lambda}_{X1}^{(d)}}{d^{\alpha}} \right)^{2}\\ 
&\mathop{\longrightarrow}_{d \to \infty}^{p} 0,
\end{align*}
which contradicts to Lemma~\ref{lem7} (note that the limiting matrix $M$ is not a degenerate matrix). Therefore, 
\begin{align*}
\left\| \tilde{\xi}_{X1}^{(d)} \right\|_{2}^{2} \mathop{\longrightarrow}_{d \to \infty}^{p} 0.
\end{align*}
Note that, since the norm of $\hat{\xi}_{X1}^{(d)}$ is 1,
\begin{align*}
\left( \hat{\xi}_{X1}^{(d)}(1) \right)^{2} = 1 - \left\| \tilde{\xi}_{X1}^{(d)} \right\|_{2}^{2} \mathop{\longrightarrow}_{d \to \infty}^{p} 1.
\end{align*}
Therefore
\begin{align*}
\left\| \hat{\xi}_{X1}^{(d)} - e_{1}^{(d)} \right\|_{2}^{2} = \left( \hat{\xi}_{X1}^{(d)}(1) - 1 \right)^{2} + \left\| \tilde{\xi}_{X1}^{(d)} \right\|_{2}^{2} \mathop{\longrightarrow}_{d \to \infty}^{p} 0.
\end{align*}
To find the limiting value of $\hat{\lambda}_{X1}^{(d)}/d^{\alpha}$ as $d \to \infty$, using the unitary invariance property of the Frobenius norm and the previous result about the limiting vectors of $\hat{\xi}_{X1}^{(d)}$ as $d \to \infty$,
\begin{align*}
&\left\| \frac{n\hat{\lambda}_{X1}^{(d)}}{d^{\alpha}} \hat{\xi}_{X1}^{(d)} \left( \hat{\xi}_{X1}^{(d)} \right)^{T} - \mathbf{M}_{X}^{(d)} \right\|_{F}^{2}\\ 
&= \left\| \frac{n\hat{\lambda}_{X1}^{(d)}}{d^{\alpha}} \hat{\xi}_{X1}^{(d)} \left( \hat{\xi}_{X1}^{(d)} \right)^{T} - \frac{\| c_{1} \|_{2}^{2}}{d^{\alpha}} e_{1}^{(d)} \left( e_{1}^{(d)} \right)^{T} \right\|_{F}^{2}\\
&=\left\| \frac{n\hat{\lambda}_{X1}^{(d)}}{d^{\alpha}} \left( \hat{\xi}_{X1}^{(d)} \right)^{T} \hat{\xi}_{X1}^{(d)} \left( \hat{\xi}_{X1}^{(d)} \right)^{T} \hat{\xi}_{X1}^{(d)} - \frac{\| c_{1} \|_{2}^{2}}{d^{\alpha}} \left( \hat{\xi}_{X1}^{(d)} \right)^{T} e_{1}^{(d)} \left( e_{1}^{(d)} \right)^{T} \hat{\xi}_{X1}^{(d)} \right\|_{F}^{2}\\
&=\left\| \frac{n\hat{\lambda}_{X1}^{(d)}}{d^{\alpha}} - \frac{\| c_{1} \|_{2}^{2}}{d^{\alpha}} \left( \hat{\xi}_{X1}^{(d)} \right)^{T} e_{1}^{(d)} \left( e_{1}^{(d)} \right)^{T} \hat{\xi}_{X1}^{(d)} \right\|_{F}^{2}\\
&=\left( \frac{\| c_{1} \|_{2}^{2}}{d^{\alpha}} \right)^{2} \left\| \frac{n\hat{\lambda}_{X1}^{(d)}/d^{\alpha}}{\| c_{1} \|_{2}^{2}/d^{\alpha}} - \left( \hat{\xi}_{X1}^{(d)} \right)^{T} e_{1}^{(d)} \left( e_{1}^{(d)} \right)^{T} \hat{\xi}_{X1}^{(d)} \right\|_{F}^{2}\\
&\mathop{\longrightarrow}_{d \to \infty}^{p} \text{A}^{2} \left( \frac{\text{B}}{\text{C}} - 1 \right)^{2},
\end{align*}
where,
\begin{align*}
\text{A}=\lim_{d \to \infty} \frac{\| c_{1} \|_{2}^{2}}{d^{\alpha}}, \,\, \text{B}=\lim_{d \to \infty} \frac{n\hat{\lambda}_{X1}^{(d)}}{d^{\alpha}},\,\, \text{C}=\lim_{d \to \infty} \frac{\| c_{1} \|_{2}^{2}}{d^{\alpha}}.
\end{align*}
By Lemma~\ref{lem7},
\begin{align*}
\left \| \frac{n\hat{\lambda}_{X1}^{(d)}}{d^{\alpha}} \hat{\xi}_{X1}^{(d)} \left(\hat{\xi}_{X1}^{(d)}\right)^{T} - \mathbf{M}_{X}^{(d)} \right \|_{F}^{2} \mathop{\longrightarrow}_{d \to \infty}^{p} 0,
\end{align*}
and the fact that $\text{A} \asymp O_{p}(1)$,
\begin{align*}
\left( \frac{\text{B}}{\text{C}} - 1 \right)^{2} \mathop{\longrightarrow}_{d \to \infty}^{p} 0.
\end{align*}
Therefore,
\begin{align*}
\frac{\text{B}}{\text{C}} \mathop{\longrightarrow}_{d \to \infty}^{p} 1.
\end{align*}
\end{proof}
\begin{lem}[CCA HDLSS Asymptotic lemma 10.]
\label{lem10}
The sample eigenvalue $\hat{\lambda}_{Xi}^{(d)}$ and eigenvector $\hat{\xi}_{Xi}^{(d)}$, $i=2,3,..,d$, converge in probability to the following quantities as $d \to \infty$,
\begin{align*}
\frac{n\hat{\lambda}_{Xi}^{(d)}}{d} \mathop{\longrightarrow}_{d \to \infty}^{P} \tau_{X}^{2}, \,\, \langle \hat{\xi}_{Xi}^{(d)}, \xi^{(d)} \rangle \mathop{\longrightarrow}_{d \to \infty}^{P} 0, \,\, i=2,3,\dots,n,
\end{align*}
where $\xi^{(d)}$ is an any given vector in $R^{d}$. 
\end{lem}
\begin{proof}
Recall that the underlying random variable $X^{(d)}$ of the sample covariance matrix $\hat{\mathbf{\Sigma}}_{X}^{(d)}$ has a simple spiked covariance structure of~(\ref{sec3:equ8}). Then, the asymptotic behavior of a sample eigenvalue is a direct consequence of Theorem~\ref{thm1} for $\alpha > 1$. The result on the behavior of $\langle \hat{\xi}_{Xi}^{(d)}, \xi^{(d)} \rangle$ is indicated in the proof of Theorem~\ref{thm1} given in~\cite{38}, even though $\xi^{(d)}$ is fixed at the population counterpart of $\hat{\xi}_{Xi}^{(d)}$ in Theorem~\ref{thm1} for comparison purpose. Note that an inner product of two unit vectors measures a consine of the angle between the two. Lemma~\ref{lem10} implies that the sample eigenvector $\hat{\xi}_{Xi}^{(d)}$ becomes completely random such that the probability of it being consistent with any given vector is 0 as $d \to \infty$.
\end{proof}

\subsubsection{Behavior of the matrix \texorpdfstring{$\hat{\mathbf{R}}^{(d)}$}{}}

The definition of the matrix $\hat{\mathbf{R}}^{(d)}$ is given in~(\ref{sec2:equ3}) with the inverse matrix we are going to use explained in~(\ref{sec3:equ5}). The sample canonical correlation coefficients are found as the singular values of $\hat{\mathbf{R}}^{(d)}$ and sample canonical weight vectors are obtained via unscaling and normalizing the singular vectors of $\hat{\mathbf{R}}^{(d)}$ shown in~(\ref{sec2:equ4}).
\begin{align}
\label{sec3:equ22}
\begin{aligned}
\hat{\mathbf{R}}^{(d)} &= \left( \hat{\mathbf{\Sigma}}_{X}^{(d)} \right)^{-\frac{1}{2}} \hat{\mathbf{\Sigma}}_{XY}^{(d)} \left( \hat{\mathbf{\Sigma}}_{Y}^{(d)} \right)^{-\frac{1}{2}}\\
&= \sum_{i=1}^{n} \sqrt{\frac{1}{\hat{\lambda}_{Xi}^{(d)}}} \hat{\xi}_{Xi}^{(d)} \left( \hat{\xi}_{Xi}^{(d)} \right)^{T} \sum_{i=1}^{n} \hat{\lambda}_{XYi}^{(d)} \hat{\eta}_{Xi}^{(d)} \left( \hat{\eta}_{Yi}^{(d)} \right)^{T} \sum_{i=1}^{n} \sqrt{\frac{1}{\hat{\lambda}_{Yi}^{(d)}}} \hat{\xi}_{Yi}^{(d)} \left( \hat{\xi}_{Yi}^{(d)} \right)^{T}
\end{aligned}
\end{align}
Consider the two parts of the matrix $\hat{\mathbf{R}}^{(d)}$ in the following,
\begin{align*}
\begin{bmatrix}
\frac{1}{\sqrt{\hat{\lambda}_{X1}^{(d)}}} &0 &\ldots &0\\
0 &\frac{1}{\sqrt{\hat{\lambda}_{X2}^{(d)}}} &\ldots &0\\
\vdots &\vdots &\ddots &\vdots\\
0 &0 &\ldots &\frac{1}{\sqrt{\hat{\lambda}_{Xn}^{(d)}}}
\end{bmatrix}
\begin{bmatrix}
\left( \hat{\xi}_{X1}^{(d)} \right)^{T}\\
\left( \hat{\xi}_{X2}^{(d)} \right)^{T}\\
\vdots\\
\left( \hat{\xi}_{Xn}^{(d)} \right)^{T}
\end{bmatrix}
\begin{bmatrix}
\left( \hat{\eta}_{X1}^{(d)} \right)^{T}\\
\left( \hat{\eta}_{X2}^{(d)} \right)^{T}\\
\vdots\\
\left( \hat{\eta}_{Xn}^{(d)} \right)^{T}
\end{bmatrix}^{T}
\begin{bmatrix}
\sqrt{\hat{\lambda}_{XY1}^{(d)}} &0 &\ldots &0\\
0 &\sqrt{\hat{\lambda}_{XY2}^{(d)}} &\ldots &0\\
\vdots &\vdots &\ddots &\vdots\\
0 &0 &\ldots &\sqrt{\hat{\lambda}_{XYn}^{(d)}}
\end{bmatrix},\\[5pt]
\begin{bmatrix}
\sqrt{\hat{\lambda}_{XY1}^{(d)}} &0 &\ldots &0\\
0 &\sqrt{\hat{\lambda}_{XY2}^{(d)}} &\ldots &0\\
\vdots &\vdots &\ddots &\vdots\\
0 &0 &\ldots &\sqrt{\hat{\lambda}_{XYn}^{(d)}}
\end{bmatrix}
\begin{bmatrix}
\left( \hat{\eta}_{X1}^{(d)} \right)^{T}\\
\left( \hat{\eta}_{X2}^{(d)} \right)^{T}\\
\vdots\\
\left( \hat{\eta}_{Xn}^{(d)} \right)^{T}
\end{bmatrix}^{T}
\begin{bmatrix}
\left( \hat{\xi}_{Y1}^{(d)} \right)^{T}\\
\left( \hat{\xi}_{Y2}^{(d)} \right)^{T}\\
\vdots\\
\left( \hat{\xi}_{Yn}^{(d)} \right)^{T}
\end{bmatrix}
\begin{bmatrix}
\frac{1}{\sqrt{\hat{\lambda}_{Y1}^{(d)}}} &0 &\ldots &0\\
0 &\frac{1}{\sqrt{\hat{\lambda}_{Y2}^{(d)}}} &\ldots &0\\
\vdots &\vdots &\ddots &\vdots\\
0 &0 &\ldots &\frac{1}{\sqrt{\hat{\lambda}_{Yn}^{(d)}}}
\end{bmatrix}.
\end{align*}
The first part reduces to the following matrix (call it $\hat{\mathbf{R}}^{(d)}_{1}$),
\begin{align}
\label{sec3:equ23}
\hat{\mathbf{R}}^{(d)}_{1}=
\begin{bmatrix}
\frac{\sqrt{\hat{\lambda}_{XY1}^{(d)}}}{\sqrt{\hat{\lambda}_{X1}^{(d)}}} \langle \hat{\xi}_{X1}^{(d)}, \hat{\eta}_{X1}^{(d)} \rangle &\frac{\sqrt{\hat{\lambda}_{XY2}^{(d)}}}{\sqrt{\hat{\lambda}_{X1}^{(d)}}} \langle \hat{\xi}_{X1}^{(d)}, \hat{\eta}_{X2}^{(d)} \rangle &\ldots &\frac{\sqrt{\hat{\lambda}_{XYn}^{(d)}}}{\sqrt{\hat{\lambda}_{X1}^{(d)}}} \langle \hat{\xi}_{X1}^{(d)}, \hat{\eta}_{Xn}^{(d)} \rangle\\
\frac{\sqrt{\hat{\lambda}_{XY1}^{(d)}}}{\sqrt{\hat{\lambda}_{X2}^{(d)}}} \langle \hat{\xi}_{X2}^{(d)}, \hat{\eta}_{X1}^{(d)} \rangle &\frac{\sqrt{\hat{\lambda}_{XY2}^{(d)}}}{\sqrt{\hat{\lambda}_{X2}^{(d)}}} \langle \hat{\xi}_{X2}^{(d)}, \hat{\eta}_{X2}^{(d)} \rangle &\ldots &\frac{\sqrt{\hat{\lambda}_{XYn}^{(d)}}}{\sqrt{\hat{\lambda}_{X2}^{(d)}}} \langle \hat{\xi}_{X1}^{(d)}, \hat{\eta}_{Xn}^{(d)} \rangle\\
\vdots &\vdots &\ddots &\vdots\\
\frac{\sqrt{\hat{\lambda}_{XY1}^{(d)}}}{\sqrt{\hat{\lambda}_{Xn}^{(d)}}} \langle \hat{\xi}_{Xn}^{(d)}, \hat{\eta}_{X1}^{(d)} \rangle &\frac{\sqrt{\hat{\lambda}_{XY2}^{(d)}}}{\sqrt{\hat{\lambda}_{Xn}^{(d)}}} \langle \hat{\xi}_{Xn}^{(d)}, \hat{\eta}_{X2}^{(d)} \rangle &\ldots &\frac{\sqrt{\hat{\lambda}_{XYn}^{(d)}}}{\sqrt{\hat{\lambda}_{Xn}^{(d)}}} \langle \hat{\xi}_{X1}^{(d)}, \hat{\eta}_{Xn}^{(d)} \rangle
\end{bmatrix}.
\end{align}
We investigate which value each entry of $\hat{\mathbf{R}}^{(d)}_{1}(i,j)$ converges in probability to as $d \to \infty$. Referring to Lemma~\ref{lem4},~\ref{lem15},~\ref{lem9} and~\ref{lem10} on the magnitudes of the sample eigenvalues and singualer values, it can be easily noticed that,
\begin{align*}
&\frac{\sqrt{\hat{\lambda}_{XY1}^{(d)}}}{\sqrt{\hat{\lambda}_{X1}^{(d)}}}=\frac{\sqrt{\langle c_{1}, d_{1} \rangle/d^{\alpha}}}{\sqrt{\langle c_{1}, c_{1} \rangle/d^{\alpha}}} \asymp O_{P}(1),
 \,\, \frac{\sqrt{\hat{\lambda}_{XYi}^{(d)}}}{\sqrt{\hat{\lambda}_{Xj}^{(d)}}} = o_{P}(1), \,\, i,j=2,3,\dots,n, \\
&\frac{\sqrt{\hat{\lambda}_{XYi}^{(d)}}}{\sqrt{\hat{\lambda}_{X1}^{(d)}}} = O_{P}(1/\sqrt{d^{\alpha-1}}), 
\,\, \frac{\sqrt{\hat{\lambda}_{XY1}^{(d)}}}{\sqrt{\hat{\lambda}_{Xi}^{(d)}}} = O_{P}(\sqrt{d^{\alpha-1}}), \,\, i=2,3,\dots,n.
\end{align*}
Then, $\hat{\mathbf{R}}^{(d)}_{1}$ can be written as,
\begin{align*}
&\left\| \hat{\mathbf{R}}^{(d)}_{1} -
\begin{bmatrix}
\frac{\sqrt{\langle c_{1}, d_{1} \rangle}}{\| c_{1}\|_{2}} \langle \hat{\xi}_{X1}^{(d)}, \hat{\eta}_{X1}^{(d)} \rangle &\underset{1 \times (d-1)}{\mathbf{0}}\\
\frac{\sqrt{\hat{\lambda}_{XY1}^{(d)}}}{\sqrt{\hat{\lambda}_{X2}^{(d)}}} \langle \hat{\xi}_{X2}^{(d)}, \hat{\eta}_{X1}^{(d)} \rangle &\underset{1 \times (d-1)}{\mathbf{0}}\\
\vdots &\vdots\\
\frac{\sqrt{\hat{\lambda}_{XY1}^{(d)}}}{\sqrt{\hat{\lambda}_{Xn}^{(d)}}} \langle \hat{\xi}_{Xn}^{(d)}, \hat{\eta}_{X1}^{(d)} \rangle &\underset{1 \times (d-1)}{\mathbf{0}}
\end{bmatrix}\right\|_{2} \mathop{\longrightarrow}_{d \to \infty}^{P} 0.
\end{align*}
Since $(\sqrt{\hat{\lambda}_{XY1}^{(d)}}/\sqrt{\hat{\lambda}_{Xj}^{(d)}}) \langle \hat{\xi}_{Xi}^{(d)}, \hat{\eta}_{X1}^{(d)} \rangle$, for $i=2,3,\dots,n$, appears to blow up, we further investigate their magnitudes.
\begin{lem}[CCA HDLSS Asymptotic lemma 11.]
\label{lem11}
The inner product of the first sample singular vector $\hat{\eta}_{X1}^{(d)}$ and the $i$th eigenvector $\hat{\xi}_{Xi}^{(d)}$, $i=2,3,..n$, converges to 0 with the speed of $o_{P}(1/\sqrt{d^{\alpha-1}})$ as $d \to \infty$.
\end{lem}
\begin{proof}
Consider $\mathbf{X}^{(d)}$ given in~(\ref{sec3:equ14}), which contains the observations of the random variable $X^{(d)}$. The singular value decomposition of $\mathbf{X}^{(d)}$ gives,
\begin{align*}
\underset{d \times n}{\mathbf{X}^{(d)}}=
\begin{bmatrix}
\hat{\xi}_{X1}^{(d)} &\hat{\xi}_{X2}^{(d)} &\ldots &\hat{\xi}_{Xn}^{(d)}
\end{bmatrix}.
\begin{bmatrix}
\sqrt{\hat{\lambda}_{X1}^{(d)}} &0 &\ldots &0\\
0 &\sqrt{\hat{\lambda}_{X2}^{(d)}} &\ldots &0\\
\vdots &\vdots &\ddots &\vdots\\
0 &0 &\ldots &\sqrt{\hat{\lambda}_{Xn}^{(d)}}
\end{bmatrix}
\begin{bmatrix}
\left( Z_{\hat{\xi}_{X1}}^{(d)} \right)^{T} \\
\left(Z_{\hat{\xi}_{X2}}^{(d)} \right)^{T}\\
\vdots\\
\left(Z_{\hat{\xi}_{Xn}}^{(d)} \right)^{T}
\end{bmatrix},
\end{align*}
where $\hat{\xi}_{Xi}^{(d)}$ and $\hat{\lambda}_{Xi}^{(d)})$ are eigenvectors and eigenvalues of the sample covariance matrix $\hat{\mathbf{\Sigma}}_{X}^{(d)}$, $\|Z_{\hat{\xi}_{Xi}}^{(d)}\|=1$ and $\langle Z_{\hat{\xi}_{Xi}}^{(d)}, Z_{\hat{\xi}_{Xj}}^{(d)} \rangle=0$ for $i \not= j$.  Note that $Z_{\hat{\xi}_{Xi}}^{(d)}$ is the standerdized scores of the projections of the $n$ observations in $\hat{\xi}_{Xi}^{(d)}$ onto the eigenvectors $\hat{\xi}_{Xi}^{(d)}$,
\begin{align*}
Z_{\hat{\xi}_{Xi}}^{(d)}=\frac{\left( \mathbf{X}^{(d)} \right)^{T} \hat{\xi}_{Xi}^{(d)}}{\sqrt{\hat{\lambda}_{Xi}^{(d)}}}, \,\, i=1,2,..,n.
\end{align*}
Define $Z_{\hat{\eta}_{X1}}^{(d)}$ to be the standerdized scores of the projections of the $n$ observations in $\hat{\xi}_{Xi}^{(d)}$ onto the first singular vector $\hat{\eta}_{X1}^{(d)}$ obtained from the sample cross-covariance matrix $\hat{\mathbf{\Sigma}}_{XY}^{(d)}$,
\begin{align*}
Z_{\hat{\eta}_{X1}}^{(d)}=\frac{\left( \mathbf{X}^{(d)} \right)^{T} \hat{\eta}_{X1}^{(d)}}{\sqrt{\hat{\lambda}_{X1}^{(d)}}}.
\end{align*}
First we show that $Z_{\hat{\xi}_{X1}}^{(d)}$ converges in probability to $Z_{\hat{\eta}_{X1}}^{(d)}$ as $d \to \infty$. By the triangle inequality,
\begin{align*}
\left\| Z_{\hat{\xi}_{X1}}^{(d)}-Z_{\hat{\eta}_{X1}}^{(d)} \right\|_{2} &= \left\| \frac{\left( \mathbf{X}^{(d)} \right)^{T} \hat{\xi}_{X1}^{(d)}}{\sqrt{\hat{\lambda}_{X1}^{(d)}}}-\frac{\left( \mathbf{X}^{(d)} \right)^{T} \hat{\eta}_{X1}^{(d)}}{\sqrt{\hat{\lambda}_{X1}^{(d)}}} \right\|_{2}\\
&\le \left\| \frac{\left( \mathbf{X}^{(d)} \right)^{T}}{\sqrt{\hat{\lambda}_{X1}^{(d)}}}(\hat{\xi}_{X1}^{(d)}-e_{1}^{(d)}) \right\|_{2}+\left\| \frac{\left( \mathbf{X}^{(d)} \right)^{T}}{\sqrt{\hat{\lambda}_{X1}^{(d)}}}(\hat{\eta}_{X1}^{(d)}-e_{1}^{(d)}) \right\|_{2}.
\end{align*}
The first term can be written as,
\begin{align*}
\left\| \frac{\left( \mathbf{X}^{(d)} \right)^{T}}{\sqrt{\hat{\lambda}_{X1}^{(d)}}}(\hat{\xi}_{X1}^{(d)}-e_{1}^{(d)}) \right\|_{2}^{2} = \sum_{i=1}^{n} \left\langle \frac{\mathbf{X}^{(d)}(\bullet,i)}{\sqrt{\hat{\lambda}_{X1}^{(d)}}}, \hat{\xi}_{X1}^{(d)}-e_{1}^{(d)} \right\rangle^{2},
\end{align*}
where $\mathbf{X}^{(d)}(\bullet,i)$ stands for the $i$th column of $\mathbf{X}^{(d)}$. Note that the first component in each column of $\mathbf{X}^{(d)}$ is of $O_{P}(\sqrt{d^{\alpha}})$, which is that of $\sqrt{\hat{\lambda}_{X1}^{(d)}}$ by Lemma~\ref{lem9}, and the rest of the elements are of $O_{P}(1)$ by Lemma~\ref{lem10}. Since $\| \hat{\xi}_{X1}^{(d)} - e_{1}^{(D)} \|_{2}$ converge in probability to 0 by Lemma~\ref{lem9}, using the Cauchy-Schwarz inequality and the fact that $\alpha > 1$,
\begin{align*}
\left\langle \frac{\mathbf{X}^{(d)}(\bullet,i)}{\sqrt{\hat{\lambda}_{X1}^{(d)}}}, \hat{\xi}_{X1}^{(d)}-e_{1}^{(d)} \right\rangle^{2} 
&\le \left\| \frac{\mathbf{X}^{(d)}(\bullet,i)}{\sqrt{\hat{\lambda}_{X1}^{(d)}}} \right\|_{2}^{2} \left\| \hat{\xi}_{X1}^{(d)}-e_{1}^{(d)} \right\|_{2}^{2}\\
&=\left( \left( \frac{\mathbf{X}^{(d)}(1,i)}{\sqrt{\hat{\lambda}_{X1}^{(d)}}} \right)^{2} + \sum_{j=2}^{d} \left( \frac{\mathbf{X}^{(d)}(j,i)}{\sqrt{\hat{\lambda}_{X1}^{(d)}}} \right)^{2} \right) \left\| \hat{\xi}_{X1}^{(d)}-e_{1}^{(d)} \right\|_{2}^{2}\\
&=\left( \left( \frac{\mathbf{X}^{(d)}(1,i)}{\sqrt{\hat{\lambda}_{X1}^{(d)}}} \right)^{2} + \frac{d^{\alpha}}{\hat{\lambda}_{X1}^{(d)}} \frac{1}{d^{\alpha-1}} \sum_{j=2}^{d} \frac{\left(\mathbf{X}^{(d)}(j,i)\right)^{2}}{d} \right) \left\| \hat{\xi}_{X1}^{(d)}-e_{1}^{(d)} \right\|_{2}^{2}\\
&\mathop{\longrightarrow}_{d \to \infty}^{p} 0, \,\, i=1,2,..,n.
\end{align*}
Therefore,
\begin{align*}
\left\| \frac{\left( \mathbf{X}^{(d)} \right)^{T}}{\sqrt{\hat{\lambda}_{X1}^{(d)}}}(\hat{\xi}_{X1}^{(d)}-e_{1}^{(d)}) \right\|_{2}^{2} \mathop{\longrightarrow}_{d \to \infty}^{p} 0.
\end{align*}
Similarly, using the results given in Lemma~\ref{lem4},~\ref{lem9} and~\ref{lem10}, we have,
\begin{align*}
\left\| \frac{\left( \mathbf{X}^{(d)} \right)^{T}}{\sqrt{\hat{\lambda}_{X1}^{(d)}}}(\hat{\eta}_{X1}^{(d)}-e_{1}^{(d)}) \right\|_{2}^{2} \mathop{\longrightarrow}_{d \to \infty}^{p} 0,
\end{align*}
which leads to,
\begin{align*}
\left\| Z_{\hat{\xi}_{X1}}^{(d)}-Z_{\hat{\eta}_{X1}}^{(d)} \right\|_{2} \mathop{\longrightarrow}_{d \to \infty}^{p} 0.
\end{align*}
Now we show that the inner product $\langle Z_{\hat{\xi}_{Xi}}^{(d)}, Z_{\hat{\eta}_{X1}}^{(d)} \rangle$ for $i=2,3,..,n$, converges to 0 as $d \to \infty$. Using $\langle Z_{\hat{\xi}_{Xi}}^{(d)}, Z_{\hat{\xi}_{X1}}^{(d)} \rangle = 0$ for $i=2,3,..,n$,
\begin{align*}
\left\langle Z_{\hat{\xi}_{Xi}}^{(d)}, Z_{\hat{\eta}_{X1}}^{(d)} \right\rangle &= \left\langle Z_{\hat{\xi}_{Xi}}^{(d)}, Z_{\hat{\eta}_{X1}}^{(d)} - Z_{\hat{\xi}_{X1}}^{(d)} + Z_{\hat{\xi}_{X1}}^{(d)} \right\rangle\\
&= \left\langle Z_{\hat{\xi}_{Xi}}^{(d)}, Z_{\hat{\eta}_{X1}}^{(d)} - Z_{\hat{\xi}_{X1}}^{(d)} \right\rangle + \left\langle Z_{\hat{\xi}_{Xi}}^{(d)}, Z_{\hat{\xi}_{X1}}^{(d)} \right\rangle\\
&= \left\langle Z_{\hat{\xi}_{Xi}}^{(d)}, Z_{\hat{\eta}_{X1}}^{(d)} - Z_{\hat{\xi}_{X1}}^{(d)} \right\rangle.
\end{align*}
Using the Cauchy-Schwarz inequality and the fact that $Z_{\hat{\xi}_{Xi}}^{(d)}$ for $i=1,2,..,n$, consists of standerdized scores (of unit variance),
\begin{align*}
\left\langle Z_{\hat{\xi}_{Xi}}^{(d)}, Z_{\hat{\eta}_{X1}}^{(d)} - Z_{\hat{\xi}_{X1}}^{(d)} \right\rangle^{2} &\le \left\| Z_{\hat{\xi}_{Xi}}^{(d)} \right\|_{2}^{2} \left\| Z_{\hat{\eta}_{X1}}^{(d)} - Z_{\hat{\xi}_{X1}}^{(d)} \right\|_{2}^{2}\\
&= 1 \times \left\| Z_{\hat{\eta}_{X1}}^{(d)} - Z_{\hat{\xi}_{X1}}^{(d)} \right\|_{2}^{2}\\
&\mathop{\longrightarrow}_{d \to \infty}^{p} 0, \,\, i=1,2,..,n.
\end{align*}
We now show that $\mathbf{X}^{(d)} Z_{\hat{\eta}_{X1}}^{(d)}/\sqrt{\hat{\lambda}_{X1}^{(d)}}$ converges in probability to $\hat{\eta}_{X1}^{(d)}$ as $d \to \infty$. By the triangle inequality,
\begin{align*}
\left\| \hat{\eta}_{X1}^{(d)} - \frac{\mathbf{X}^{(d)} Z_{\hat{\eta}_{X1}}^{(d)}}{\sqrt{\hat{\lambda}_{X1}^{(d)}}} \right\|_{2} &= \left\| \left(\hat{\eta}_{X1}^{(d)} - \hat{\xi}_{X1}^{(d)} \right) - \left( \hat{\xi}_{X1}^{(d)} - \frac{\mathbf{X}^{(d)} Z_{\hat{\eta}_{X1}}^{(d)}}{\sqrt{\hat{\lambda}_{X1}^{(d)}}} \right) \right\|_{2}\\
&\le \left\| \hat{\eta}_{X1}^{(d)} - \hat{\xi}_{X1}^{(d)} \right\|_{2} + \left\| \hat{\xi}_{X1}^{(d)} - \frac{\mathbf{X}^{(d)} Z_{\hat{\eta}_{X1}}^{(d)}}{\sqrt{\hat{\lambda}_{X1}^{(d)}}} \right\|_{2}.
\end{align*}
As $\| \hat{\xi}_{X1}^{(d)} - e_{1}^{(d)} \|_{2}$ and $\| \hat{\eta}_{X1}^{(d)} - e_{1}^{(d)} \|_{2}$ converge in probability to 0 by Lemma~\ref{lem4} and~\ref{lem10}, so does $\| \hat{\eta}_{X1}^{(d)} - \hat{\xi}_{X1}^{(d)} \|_{2}$. By the triangle inequality, the second term becomes,
\begin{align*}
\left\| \hat{\xi}_{X1}^{(d)} - \frac{\mathbf{X}^{(d)} Z_{\hat{\eta}_{X1}}^{(d)}}{\sqrt{\hat{\lambda}_{X1}^{(d)}}} \right\|_{2} &= \left\| \left( \hat{\xi}_{X1}^{(d)} - \frac{\mathbf{X}^{(d)} Z_{\hat{\xi}_{X1}}^{(d)}}{\sqrt{\hat{\lambda}_{X1}^{(d)}}}\right) - \left(\frac{\mathbf{X}^{(d)} Z_{\hat{\xi}_{X1}}^{(d)}}{\sqrt{\hat{\lambda}_{X1}^{(d)}}} -\frac{\mathbf{X}^{(d)} Z_{\hat{\eta}_{X1}}^{(d)}}{\sqrt{\hat{\lambda}_{X1}^{(d)}}} \right) \right\|_{2}\\
&\le \left\| \hat{\xi}_{X1}^{(d)} - \frac{\mathbf{X}^{(d)} Z_{\hat{\xi}_{X1}}^{(d)}}{\sqrt{\hat{\lambda}_{X1}^{(d)}}} \right\|_{2} + \left\| \frac{\mathbf{X}^{(d)} Z_{\hat{\xi}_{X1}}^{(d)}}{\sqrt{\hat{\lambda}_{X1}^{(d)}}} -\frac{\mathbf{X}^{(d)} Z_{\hat{\eta}_{X1}}^{(d)}}{\sqrt{\hat{\lambda}_{X1}^{(d)}}} \right\|_{2}.
\end{align*}
The first term is 0 since $\mathbf{X}^{(d)} Z_{\hat{\xi}_{X1}}^{(d)}/\sqrt{\hat{\lambda}_{X1}^{(d)}} = \hat{\xi}_{X1}^{(d)}$ in the singular decompostion of $\mathbf{X}^{(d)}$. The second term goes as follows,
\begin{align*}
\left\| \frac{\mathbf{X}^{(d)} Z_{\hat{\xi}_{X1}}^{(d)}}{\sqrt{\hat{\lambda}_{X1}^{(d)}}}-\frac{\mathbf{X}^{(d)} Z_{\hat{\eta}_{X1}}^{(d)}}{\sqrt{\hat{\lambda}_{X1}^{(d)}}} \right\|_{2}^{2} &= \left\| \frac{\mathbf{X}^{(d)}}{\sqrt{\hat{\lambda}_{X1}^{(d)}}}(Z_{\hat{\xi}_{X1}}^{(d)}-Z_{\hat{\eta}_{X1}}^{(d)}) \right\|_{2}^{2}\\
&= \sum_{i=1}^{d} \left\langle \left( \frac{\mathbf{X}^{(d)}(i,\bullet)}{\sqrt{\hat{\lambda}_{X1}^{(d)}}} \right)^{T}, Z_{\hat{\xi}_{X1}}^{(d)}-Z_{\hat{\eta}_{X1}}^{(d)} \right\rangle^{2},
\end{align*}
where $\mathbf{X}^{(d)}(i,\bullet)$ stands for the $i$th row of the $d \times n$ data matrix $\mathbf{X}^{(d)}$. Note that only the first row of the data matrix $\mathbf{X}^{(d)}$ has its components of magnitutde of $O_{P}(\sqrt{d^{\alpha}})$, and the components of the rest of the rows are all of $O_{P}(1)$. Since we showed that $\|Z_{\hat{\xi}_{X1}}^{(d)}-Z_{\hat{\eta}_{X1}}^{(d)}\|_{2}$ converges in probability to 0 as $d \to \infty$,
\begin{align*}
\left\langle \left( \frac{\mathbf{X}^{(d)}(1,\bullet)}{\sqrt{\hat{\lambda}_{X1}^{(d)}}} \right)^{T}, Z_{\hat{\xi}_{X1}}^{(d)}-Z_{\hat{\eta}_{X1}}^{(d)} \right\rangle^{2} &\le \left\| \left( \frac{\mathbf{X}^{(d)}(1,\bullet)}{\sqrt{\hat{\lambda}_{X1}^{(d)}}} \right)^{T} \right\|_{2}^{2} \left\| Z_{\hat{\xi}_{X1}}^{(d)}-Z_{\hat{\eta}_{X1}}^{(d)} \right\|_{2}^{2}\\
&\mathop{\longrightarrow}_{d \to \infty}^{p} 0.
\end{align*}
Similarly,
\begin{align*}
\left\langle \left( \frac{\mathbf{X}^{(d)}(i,\bullet)}{\sqrt{\hat{\lambda}_{X1}^{(d)}}} \right)^{T}, Z_{\hat{\xi}_{X1}}^{(d)}-Z_{\hat{\eta}_{X1}}^{(d)} \right\rangle^{2} &\le \left\| \left( \frac{\mathbf{X}^{(d)}(i,\bullet)}{\sqrt{\hat{\lambda}_{X1}^{(d)}}} \right)^{T} \right\|_{2}^{2} \left\| Z_{\hat{\xi}_{X1}}^{(d)}-Z_{\hat{\eta}_{X1}}^{(d)} \right\|_{2}^{2}\\
&\mathop{\longrightarrow}_{d \to \infty}^{p} 0, \,\, i=2,3,..,d.
\end{align*}
Therefore, we proved that,
\begin{align*}
\left\| \hat{\eta}_{X1}^{(d)} - \frac{\mathbf{X}^{(d)} Z_{\hat{\eta}_{X1}}^{(d)}}{\sqrt{\hat{\lambda}_{X1}^{(d)}}} \right\|_{2} \mathop{\longrightarrow}_{d \to \infty}^{p} 0.
\end{align*}
Finally to determine the speed of $\langle \hat{\eta}_{X1}^{(d)}, \hat{\xi}_{Xi}^{(d)} \rangle$ converging to 0 as $d \to \infty$, 
\begin{align*}
\left( Z_{\hat{\eta}_{X1}}^{(d)} \right)^{T} Z_{\hat{\xi}_{Xi}}^{(d)} &= \frac{\left( Z_{\hat{\eta}_{X1}}^{(d)} \right)^{T} \left( \mathbf{X}^{(d)} \right)^{T} \hat{\xi}_{Xi}^{(d)}}{\sqrt{\hat{\lambda}_{Xi}^{(d)}}}\\
&= \frac{\sqrt{\hat{\lambda}_{X1}^{(d)}} \frac{\left( \mathbf{X}^{(d)} Z_{\hat{\eta}_{X1}}^{(d)} \right)^{T}}{\sqrt{\hat{\lambda}_{X1}^{(d)}}} \hat{\xi}_{Xi}^{(d)}}{\sqrt{\hat{\lambda}_{Xi}^{(d)}}}\\
&= \frac{\sqrt{d^{\alpha}}}{\sqrt{d}}\frac{\frac{\sqrt{\hat{\lambda}_{X1}^{(d)}}}{\sqrt{d^{\alpha}}} \frac{\left\langle \mathbf{X}^{(d)} Z_{\hat{\eta}_{X1}}^{(d)}, \hat{\xi}_{Xi}^{(d)} \right\rangle}{\sqrt{\hat{\lambda}_{X1}^{(d)}}}}{\frac{\sqrt{\hat{\lambda}_{Xi}^{(d)}}}{\sqrt{d}}}, \,\, i=2,3,..n.
\end{align*}
By Lemma~\ref{lem9} and~\ref{lem10}, the magnitudes of $\sqrt{\hat{\lambda}_{X1}^{(d)}}/\sqrt{d^{\alpha}}$ and $\sqrt{\hat{\lambda}_{Xi}^{(d)}}/\sqrt{d}$ are of $O_{P}(1)$. Recall that $\langle Z_{\hat{\eta}_{X1}}^{(d)}, Z_{\hat{\xi}_{Xi}}^{(d)} \rangle$ converges in probability to 0 as $d \to \infty$. Hence,
\begin{align*}
\left\langle \frac{\mathbf{X}^{(d)} Z_{\hat{\eta}_{X1}}^{(d)}}{\sqrt{\hat{\lambda}_{X1}^{(d)}}}, \hat{\xi}_{Xi}^{(d)} \right\rangle \mathop{\longrightarrow}_{d \to \infty}^{p} o_{P} \left(\frac{1}{\sqrt{d^{\alpha-1}}} \right), \,\, i=2,3,..,n.
\end{align*}
However, we know that,
\begin{align*}
\left\|\frac{\mathbf{X}^{(d)} Z_{\hat{\eta}_{X1}}^{(d)}}{\sqrt{\hat{\lambda}_{X1}^{(d)}}} - \hat{\eta}_{X1}^{(d)} \right\|_{2} = \left\| \frac{\mathbf{X}^{(d)} Z_{\hat{\eta}_{X1}}^{(d)}}{\sqrt{\hat{\lambda}_{X1}^{(d)}}} - \frac{\mathbf{X}^{(d)} Z_{\hat{\eta}_{X1}}^{(d)}}{\sqrt{\hat{\lambda}_{XY1}^{(d)}}} \right\|_{2} \mathop{\longrightarrow}_{d \to \infty}^{p} 0,
\end{align*}
where the magnitude of the sample singular value $\hat{\lambda}_{XY1}^{(d)}$ is the same as that of $\hat{\lambda}_{X1}^{(d)}$ as shown in Lemma~\ref{lem4}. Therefore, 
\begin{align*}
\langle \hat{\eta}_{X1}^{(d)}, \hat{\xi}_{Xi}^{(d)} \rangle = o_{P} \left(\frac{1}{\sqrt{d^{\alpha-1}}} \right), \,\, i=2,3,..,n.
\end{align*}
\end{proof}
\noindent By Lemma~\ref{lem11}, the matrix $\hat{\mathbf{R}}^{(d)}_{1}$ in~(\ref{sec3:equ23}) has the following form as $d \to \infty$,
\begin{align}
\label{sec3:equ24}
\left\| \hat{\mathbf{R}}^{(d)}_{1}-
\begin{bmatrix}
\frac{\sqrt{\langle c_{1}, d_{1} \rangle}}{\| c_{1} \|_{2}} &\underset{1 \times (d-1)}{\mathbf{0}}\\
\underset{(d-1) \times 1}{\mathbf{0}} &\underset{(d-1) \times (d-1)}{\mathbf{0}}
\end{bmatrix} \right\|_{2} \mathop{\longrightarrow}_{d \to \infty}^{P} 0.
\end{align}
Similarly, the corresponding part of $\hat{\mathbf{R}}^{(d)}_{1}$ on the right side of the matrix $\hat{\mathbf{R}}^{(d)}$ (call it $\hat{\mathbf{R}}^{(d)}_{2}$) reduces to a similar form of~(\ref{sec3:equ24}),
\begin{align*}
\left\| \hat{\mathbf{R}}^{(d)}_{2}-
\begin{bmatrix}
\frac{\sqrt{\langle c_{1}, d_{1} \rangle}}{\| c_{2} \|_{2}} &\underset{1 \times (d-1)}{\mathbf{0}}\\
\underset{(d-1) \times 1}{\mathbf{0}} &\underset{(d-1) \times (d-1)}{\mathbf{0}}
\end{bmatrix} \right\|_{2} \mathop{\longrightarrow}_{d \to \infty}^{P} 0.
\end{align*}
Then, the matrix $\hat{\mathbf{R}}^{(d)}$ is written as,
\begin{align*}
\hat{\mathbf{R}}^{(d)}=
\begin{bmatrix}
\hat{\xi}_{X1}^{(d)} &\hat{\xi}_{X2}^{(d)} &\ldots &\hat{\xi}_{Xn}^{(d)}
\end{bmatrix}
\hat{\mathbf{R}}^{(d)}_{1} \hat{\mathbf{R}}^{(d)}_{2}
\begin{bmatrix}
\left( \hat{\xi}_{Y1}^{(d)} \right)^{T}\\ 
\left( \hat{\xi}_{Y2}^{(d)} \right)^{T}\\
\vdots\\
\left( \hat{\xi}_{Yn}^{(d)} \right)^{T}
\end{bmatrix},
\end{align*}
and its limiting matrix is,
\begin{align}
\label{sec3:equ25}
\left\| \hat{\mathbf{R}}^{(d)}-
\begin{bmatrix}
\hat{\xi}_{X1}^{(d)} &\hat{\xi}_{X2}^{(d)} &\ldots &\hat{\xi}_{Xn}^{(d)}
\end{bmatrix}
\begin{bmatrix}
\frac{\langle c_{1}, d_{1} \rangle}{\| c_{1} \|_{2}\| c_{2} \|_{2}} &\underset{1 \times (d-1)}{\mathbf{0}}\\
\underset{(d-1) \times 1}{\mathbf{0}} &\underset{(d-1) \times (d-1)}{\mathbf{0}}
\end{bmatrix}
\begin{bmatrix}
\left( \hat{\xi}_{Y1}^{(d)} \right)^{T}\\
\left( \hat{\xi}_{Y2}^{(d)} \right)^{T}\\
\vdots\\
\left( \hat{\xi}_{Yn}^{(d)} \right)^{T}
\end{bmatrix}\right\|_{2}
\mathop{\longrightarrow}_{d \to \infty}^{P} 0,
\end{align}
where $\hat{\xi}_{Yi}^{(d)}$ is a sample eigenvector obtained from the eigendecomposition of the sample covariance matrix shown in~(\ref{sec3:equ19}) from the data matrix $\mathbf{Y}^{(d)}$ of~(\ref{sec3:equ14}). 

Perform the SVD of $\hat{\mathbf{R}}^{(d)}$ and denote its sample left and right singular vectors by $\hat{\eta}_{RXi}^{(d)}$ and $\hat{\eta}_{RYi}^{(d)}$ for $i=1,2,..,n$. From~(\ref{sec3:equ25}), the sample singular vectors of $\hat{\mathbf{R}}^{(d)}$ are linear combinations of $\hat{\xi}_{Xi}^{(d)}$'s or $\hat{\xi}_{Yi}^{(d)}$'s,
\begin{align}
\label{sec3:equ26}
\begin{aligned}
\hat{\eta}_{RXi}^{(d)}=\sum_{j=1}^{n} a_{j}^{(d)} \hat{\xi}_{Xj}^{(d)}, \,\, \hat{\eta}_{RYi}^{(d)}=\sum_{k=1}^{n} b_{k}^{(d)} \hat{\xi}_{Yk}^{(d)}, \,\, i=1,2,\dots,n,
\end{aligned}
\end{align}
where $(a_{1}^{(d)})^{2}+ (a_{2}^{(d)})^{2}+\dots+(a_{n}^{(d)})^{2}=(b_{1}^{(d)})^{2}+ (b_{2}^{(d)})^{2}+\dots+(b_{n}^{(d)})^{2}=1$ to ensure that norms of $\hat{\eta}_{RXi}^{(d)}$ and $\hat{\eta}_{RYi}^{(d)}$ are of unit length.

\subsubsection{Behavior of the first sample canonical weight vector}

The sample canonical weight vectors $\hat{\psi}_{Xi}^{(d)}$ and $\hat{\psi}_{Yi}^{(d)}$ are found by unscaling and normalizing $\hat{\eta}_{RX1}^{(d)}$ and $\hat{\eta}_{RY1}^{(d)}$ as in~(\ref{sec2:equ4}) using the pseudoinverse of~(\ref{sec3:equ5}),
\begin{align}
\label{sec3:equ27}
\hat{\psi}_{Xi}^{(d)} = \frac{\left( \hat{\mathbf{\Sigma}}^{(d)}_{X} \right)^{-\frac{1}{2}} \hat{\eta}_{RXi}^{(d)}}{\left \| \left( \hat{\mathbf{\Sigma}}^{(d)}_{X} \right)^{-\frac{1}{2}} \hat{\eta}_{RXi}^{(d)} \right\|_{2}} , \,\, \hat{\psi}_{Yi}^{(d)} = \frac{\left( \hat{\mathbf{\Sigma}}^{(d)}_{Y} \right)^{-\frac{1}{2}} \hat{\eta}_{RYi}^{(d)}}{\left \| \left( \hat{\mathbf{\Sigma}}^{(d)}_{Y} \right)^{-\frac{1}{2}} \hat{\eta}_{RYi}^{(d)} \right\|_{2}}, \,\, i=1,2,..,n.
\end{align}
To see if $\hat{\psi}_{X1}^{(d)}$ and $\hat{\psi}_{Y1}^{(d)}$ is consistent to their population counterparts $\psi_{X1}^{(d)}$ and $\psi_{Y1}^{(d)}$ given in~(\ref{sec3:equ13}), we investigate the limiting value of their inner products $\langle \hat{\psi}_{X1}^{(d)}, \psi_{X1}^{(d)} \rangle$ and $\langle \hat{\psi}_{Y1}^{(d)}, \psi_{Y1}^{(d)} \rangle$, which measure the cosine of the angle formed by the two. Expanding~(\ref{sec3:equ27}),
\begin{align*}
\langle \hat{\psi}_{X1}^{(d)}, \psi_{X1}^{(d)} \rangle &= \frac{ \left( \psi_{X1}^{(d)} \right)^{T} \left( \hat{\mathbf{\Sigma}}^{(d)}_{X} \right)^{-\frac{1}{2}} \hat{\eta}_{RX1}^{(d)}}{\left\| \psi_{X1}^{(d)} \right\|_{2} \left \| \left( \hat{\mathbf{\Sigma}}^{(d)}_{X} \right)^{-\frac{1}{2}} \hat{\eta}_{RX1}^{(d)} \right\|_{2}}\\
&=\frac{ \left( \psi_{X1}^{(d)} \right)^{T} \left( \sum\limits_{i=1}^{n} \hat{\lambda}_{Xi}^{(d)} \hat{\xi}_{Xi}^{(d)} \left( \hat{\xi}_{Xi}^{(d)} \right)^{T} \right)^{-\frac{1}{2}} \hat{\eta}_{RX1}^{(d)}}{\left\| \psi_{X1}^{(d)} \right\|_{2} \left \| \left( \sum\limits_{i=1}^{n} \hat{\lambda}_{Xi}^{(d)} \hat{\xi}_{Xi}^{(d)} \left( \hat{\xi}_{Xi}^{(d)} \right)^{T} \right)^{-\frac{1}{2}} \hat{\eta}_{RX1}^{(d)} \right\|_{2}}\\ 
&=\frac{ \left( \cos \theta_{X} e{1}^{(d)} + \sin \theta_{X} e_{2}^{(d)} \right)^{T} \left( \sum\limits_{i=1}^{n} \frac{1}{\sqrt{\hat{\lambda}_{Xi}^{(d)}}} \hat{\xi}_{Xi}^{(d)} \left( \hat{\xi}_{Xi}^{(d)} \right)^{T} \right) \hat{\eta}_{RX1}^{(d)}}{\left\| \cos \theta_{X} e_{1}^{(d)} + \sin \theta_{X} e_{2}^{(d)}\right\|_{2} \left \| \left( \sum\limits_{i=1}^{n} \frac{1}{\sqrt{\hat{\lambda}_{Xi}^{(d)}}} \hat{\xi}_{Xi}^{(d)} \left( \hat{\xi}_{Xi}^{(d)} \right)^{T} \right) \hat{\eta}_{RX1}^{(d)} \right\|_{2}}.
\end{align*}
Let's take a close look at the denominator term. It is easy to see that $\left\| \cos \theta_{X} e_{1}^{(d)} + \sin \theta_{X} e_{2}^{(d)}\right\|_{2}=1$. Using~(\ref{sec3:equ26}),
\begin{align*}
\left \| \left( \sum\limits_{i=1}^{n} \frac{1}{\sqrt{\hat{\lambda}_{Xi}^{(d)}}} \hat{\xi}_{Xi}^{(d)} \left( \hat{\xi}_{Xi}^{(d)} \right)^{T} \right) \hat{\eta}_{RX1}^{(d)} \right\|_{2}&=\left \| \sum_{i=1}^{n} \sum_{j=1}^{n} \frac{a_{j}}{\sqrt{\hat{\lambda}_{Xi}^{(d)}}} \langle \hat{\xi}_{Xi}^{(d)}, \hat{\xi}_{Xj}^{(d)} \rangle \hat{\xi}_{Xi}^{(d)} \right\|_{2}\\
&=\left \| \sum\limits_{i=1}^{n} \frac{a_{i}}{\sqrt{\hat{\lambda}_{Xi}^{(d)}}} \hat{\xi}_{Xi}^{(d)} \right\|_{2}.
\end{align*}
Now take a look at the numerator term,
\begin{align*}
&\left( \cos \theta_{X} e_{1}^{(d)}+\sin \theta_{X} e_{2}^{(d)} \right)^{T} \left( \sum_{i=1}^{n} \frac{1}{\sqrt{\hat{\lambda}_{Xi}^{(d)}}} \hat{\xi}_{Xi}^{(d)} \left( \hat{\xi}_{Xi}^{(d)} \right)^{T} \right) \hat{\eta}_{RX1}^{(d)}\\
&\,\, =\left( \cos \theta_{X} e_{1}^{(d)}+\sin \theta_{X} e_{2}^{(d)} \right)^{T} \left(\sum_{i=1}^{n}\sum_{j=1}^{n} \frac{a_{j}^{(d)}}{\sqrt{\hat{\lambda}_{Xi}^{(d)}}} \langle \hat{\xi}_{Xi}^{(d)}, \hat{\xi}_{Xj}^{(d)} \rangle \hat{\xi}_{Xi}^{(d)} \right)\\
&\,\, =\left( \cos \theta_{X} e_{1}^{(d)} \right)^{T} \left(\sum_{i=1}^{n} \sum_{j=1}^{n} \frac{a_{j}^{(d)}}{\sqrt{\hat{\lambda}_{Xi}^{(d)}}} \langle \hat{\xi}_{Xi}^{(d)}, \hat{\xi}_{Xj}^{(d)} \rangle \hat{\xi}_{Xi}^{(d)} \right)
+\left(\sin \theta_{X} e_{2}^{(d)} \right)^{T} \left(\sum_{i=1}^{n}\sum_{j=1}^{n} \frac{a_{j}^{(d)}}{\sqrt{\hat{\lambda}_{Xi}^{(d)}}} \langle \hat{\xi}_{Xi}^{(d)}, \hat{\xi}_{Xj}^{(d)} \rangle \hat{\xi}_{Xi}^{(d)}\right)\\
&\,\, =\left( \cos \theta_{X} e_{1}^{(d)} \right)^{T} \left(\sum\limits_{i=1}^{n} \frac{a_{i}^{(d)}}{\sqrt{\hat{\lambda}_{Xi}^{(d)}}} \hat{\xi}_{Xi}^{(d)} \right)+\left( \sin \theta_{X} e_{2}^{(d)} \right)^{T} \left(\sum\limits_{i=1}^{n} \frac{a_{i}^{(d)}}{\sqrt{\hat{\lambda}_{Xi}^{(d)}}} \hat{\xi}_{Xi}^{(d)} \right)\\
&\,\, =\cos \theta_{X} \left(\sum\limits_{i=1}^{n} \frac{a_{i}^{(d)}}{\sqrt{\hat{\lambda}_{Xi}^{(d)}}} \langle e_{1}^{(d)}, \hat{\xi}_{Xi}^{(d)} \rangle \right)+\sin \theta_{X} \left(\sum\limits_{i=1}^{n} \frac{a_{i}^{(d)}}{\sqrt{\hat{\lambda}_{Xi}^{(d)}}} \langle e_{2}^{(d)}, \hat{\xi}_{Xi}^{(d)} \rangle \right).
\end{align*}
Hence, we have,
\begin{align}
\label{sec3:equ28}
\langle \hat{\psi}_{X1}^{(d)}, \psi_{X1}^{(d)} \rangle = \frac{\cos \theta_{X} \left(\sum\limits_{i=1}^{n} \frac{a_{i}^{(d)}}{\sqrt{\hat{\lambda}_{Xi}^{(d)}}} \langle e_{1}^{(d)}, \hat{\xi}_{Xi}^{(d)} \rangle \right)+\sin \theta_{X} \left(\sum\limits_{i=1}^{n} \frac{a_{i}^{(d)}}{\sqrt{\hat{\lambda}_{Xi}^{(d)}}} \langle e_{2}^{(d)}, \hat{\xi}_{Xi}^{(d)} \rangle \right)}{\left \| \sum\limits_{i=1}^{n} \frac{a_{i}^{(d)}}{\sqrt{\hat{\lambda}_{Xi}^{(d)}}} \hat{\xi}_{Xi}^{(d)} \right\|_{2}}.
\end{align}
Simiarly,
\begin{align*}
\langle \hat{\psi}_{Y1}^{(d)}, \psi_{Y1}^{(d)} \rangle = \frac{\cos \theta_{Y} \left(\sum\limits_{i=1}^{n} \frac{b_{i}^{(d)}}{\sqrt{\hat{\lambda}_{Yi}^{(d)}}} \langle e_{1}^{(d)}, \hat{\xi}_{Yi}^{(d)} \rangle \right)+\sin \theta_{Y} \left(\sum\limits_{i=1}^{n} \frac{b_{i}^{(d)}}{\sqrt{\hat{\lambda}_{Yi}^{(d)}}} \langle e_{2}^{(d)}, \hat{\xi}_{Yi}^{(d)} \rangle \right)}{\left \| \sum\limits_{i=1}^{n} \frac{b_{i}^{(d)}}{\sqrt{\hat{\lambda}_{Yi}^{(d)}}} \hat{\xi}_{Yi}^{(d)} \right\|_{2}}.
\end{align*}
Observe from~(\ref{sec3:equ25}) that, for $\hat{\eta}_{RX1}^{(d)}$, $a_{1}^{(d)}$ converges to 1 and $a_{i}^{(d)}$ for $i=2,3,\dots,n$, does to 0 as $d \to \infty$, which implies,
\begin{align*}
&\left( \sum\limits_{i=1}^{n} \frac{a_{i}^{(d)}}{\sqrt{\hat{\lambda}_{Xi}^{(d)}}} \langle e_{1}^{(d)}, \hat{\xi}_{Xi}^{(d)} \rangle-\frac{1}{\sqrt{\hat{\lambda}_{X1}^{(d)}}} \langle e_{1}^{(d)}, \hat{\xi}_{X1}^{(d)} \rangle \right)^{2} \mathop{\longrightarrow}_{d \to \infty}^{P} 0,\\
&\left( \sum\limits_{i=1}^{n} \frac{a_{i}^{(d)}}{\sqrt{\hat{\lambda}_{Xi}^{(d)}}} \langle e_{2}^{(d)}, \hat{\xi}_{Xi}^{(d)} \rangle-\frac{1}{\sqrt{\hat{\lambda}_{X1}^{(d)}}} \langle e_{2}^{(d)}, \hat{\xi}_{X1}^{(d)} \rangle \right)^{2} \mathop{\longrightarrow}_{d \to \infty}^{P} 0,\\
&\left\| \sum\limits_{i=1}^{n} \frac{a_{i}^{(d)}}{\sqrt{\hat{\lambda}_{Xi}^{(d)}}} \hat{\xi}_{Xi}^{(d)}-\frac{1}{\sqrt{\hat{\lambda}_{X1}^{(d)}}} \hat{\xi}_{X1}^{(d)}\right\|_{2}^{2} \mathop{\longrightarrow}_{d \to \infty}^{P} 0.
\end{align*}
Then,~(\ref{sec3:equ28}) becomes,
\begin{align*}
\left( \langle \hat{\psi}_{X1}^{(d)}, \psi_{X1}^{(d)} \rangle - \frac{\cos \theta_{X} \left( \frac{d^{\alpha}}{\sqrt{\hat{\lambda}_{X1}^{(d)}}} \langle e_{1}^{(d)}, \hat{\xi}_{X1}^{(d)} \rangle \right)+\sin \theta_{X} \left(\frac{d^{\alpha}}{\sqrt{\hat{\lambda}_{X1}^{(d)}}} \langle e_{2}^{(d)}, \hat{\xi}_{X1}^{(d)} \rangle \right)}{\left \| \frac{d^{\alpha}}{\sqrt{\hat{\lambda}_{X1}^{(d)}}} \hat{\xi}_{X1}^{(d)} \right\|_{2}} \right)^{2} \mathop{\longrightarrow}_{d \to \infty}^{P} 0.
\end{align*}
Recall from Lemma~\ref{lem9} and~\ref{lem10} that $d^{\alpha}/\sqrt{\hat{\lambda}_{X1}^{(d)}}$ is of $(\asymp O_{P}(1))$, $\langle e_{1}^{(d)}, \hat{\xi}_{X1}^{(d)} \rangle$ and $\langle e_{2}^{(d)}, \hat{\xi}_{X1}^{(d)} \rangle$ converge in probability to 1 and 0, respectively. Therefore, we finally have, 
\begin{align*}
\left( \langle \hat{\psi}_{X1}^{(d)}, \psi_{X1}^{(d)} \rangle - \cos \theta_{X} \right)^{2} \mathop{\longrightarrow}_{d \to \infty}^{P} 0.
\end{align*}
Similarly for $\hat{\psi}_{Y1}^{(d)}$,
\begin{align*}
\left( \langle \hat{\psi}_{Y1}^{(d)}, \psi_{Y1}^{(d)} \rangle - \cos \theta_{Y} \right)^{2} \mathop{\longrightarrow}_{d \to \infty}^{P} 0.
\end{align*}

\subsubsection{Behavior of the first sample canonical correlation coefficient}

The limiting value of first sample canonical correlation coefficients is found as the first singular value of the matrix $\hat{\mathbf{R}}^{(d)}$ in~(\ref{sec3:equ25}) under the limiting operation of $d \to \infty$, which turns out its (1,1)th entry,
\begin{align*}
\left( \hat{\rho}_{1}-\frac{\langle c_{1}, d_{1} \rangle /d^{\alpha}}{\| c_{1} \|_{2}^{2} /d^{\alpha} \| c_{2}\|_{2}^{2} /d^{\alpha}} \right)^{2} \mathop{\longrightarrow}_{d \to \infty}^{P} 0.
\end{align*}
Recall from Lemma~\ref{lem2} and~\ref{lem7} that $\langle c_{1}, d_{1} \rangle /d^{\alpha}$ is the limiting value of (1,1) entry of $n\hat{\mathbf{\Sigma}}^{(d)}_{XY}/d^{\alpha}$ as $d \to \infty$ and that $\| c_{1} \|_{2}^{2} /d^{\alpha}$ and $\| c_{1} \|_{2}^{2} /d^{\alpha}$ are the limiting values of the (1,1) entries of $n\hat{\mathbf{\Sigma}}^{(d)}_{X}/d^{\alpha}$ and $n\hat{\mathbf{\Sigma}}^{(d)}_{Y}/d^{\alpha}$. That is, those values are found in the (d+1,d+1), (1,1) and (1,d+1) positions of the matrix of~(\ref{sec3:equ34}) scaled by $d^{\alpha}$ as $d$ increases,
\begin{align*}
\frac{1}{d^{\alpha}}
\begin{bmatrix}
\hat{n\mathbf{\Sigma}}_{X}^{(d)} &\hat{n\mathbf{\Sigma}}_{XY}^{(d)}\\
\left(n\hat{\mathbf{\Sigma}}_{XY}^{(d)} \right)^{T} &n\hat{\mathbf{\Sigma}}_{Y}^{(d)}
\end{bmatrix}=
\frac{1}{d^{\alpha}}
\begin{bmatrix}
\mathbf{X}^{(d)}\\
\mathbf{Y}^{(d)}
\end{bmatrix}
\begin{bmatrix}
\mathbf{X}^{(d)}\\
\mathbf{Y}^{(d)}
\end{bmatrix}^{T}.
\end{align*}
Referring to~(\ref{sec3:equ14}) and~(\ref{sec3:equ15}), we switch the (1,2) and (1,d+1) entries, the (2,1) and (d+1,1) entries and the (2,2) and (d+1,d+1) entries of the matrix $\mathbf{\Sigma}_{T}^{(2d)}$ and denote the resulting matrix by $\acute{\mathbf{\Sigma}}_{T}^{(2d)}$. We, also, switch the second and $(d+1)$th rows of the matrix $\mathbf{Z}^{(2d)}$ and denote the resulting matrix by $\acute{\mathbf{Z}}^{(2d)}$. Then, the two data matrices below,
\begin{align*}
\begin{bmatrix}
\mathbf{X}^{(d)}\\
\mathbf{Y}^{(d)}
\end{bmatrix} =
\left(\mathbf{\Sigma}_{T}^{(2d)}\right)^{\frac{1}{2}} \mathbf{Z}^{(2d)}
\,\, \text{and} \,\,
\left(\acute{\mathbf{\Sigma}}_{T}^{(2d)}\right)^{\frac{1}{2}} \acute{\mathbf{Z}}^{(2d)},
\end{align*}
are the same only except that the $(d+1)$th row of the first matrix is the second row of the second one. Following the similar steps as done in the proof of Lemma~\ref{lem2}, it can be shown that,
\begin{align*}
\left\| \left(\frac{\acute{\mathbf{\Sigma}}_{T}^{(2d)}}{d^{\alpha}}\right)^{\frac{1}{2}} \acute{\mathbf{Z}}^{(2d)}-
\begin{bmatrix}
\sigma_{X}^{2} &\rho\sigma_{X}\sigma_{Y} &\underset{1 \times (2d-2)}{\mathbf{0}}\\
\rho\sigma_{X}\sigma_{Y} &\sigma_{X}^{2} &\underset{1 \times (2d-2)}{\mathbf{0}}\\
\underset{(2d-2) \times 1}{\mathbf{0}} &\underset{(2d-2) \times 1}{\mathbf{0}} &\underset{(2d-2) \times (2d-2)}{\mathbf{0}}
\end{bmatrix}^{\frac{1}{2}}
\begin{bmatrix}
z_{X1\bullet}\\
z_{Y1\bullet}\\
z_{X2\bullet}\\
\vdots\\
z_{Xd\bullet}\\
z_{Y2\bullet}\\
\vdots\\
z_{Yd\bullet}
\end{bmatrix}
\right\|_{2} \mathop{\longrightarrow}_{d \to \infty}^{P} 0.
\end{align*}
Then, the limiting values of $\langle c_{1}, d_{1} \rangle /d^{\alpha}$, $\| c_{1} \|_{2}^{2} /d^{\alpha}$ and $\| c_{1} \|_{2}^{2} /d^{\alpha}$ are found in the first $2 \times 2$ block of the following matrix, as $d \to \infty$,
\begin{align*}
\frac{1}{d^{\alpha}}\left(\acute{\mathbf{\Sigma}}_{T}^{(2d)}\right)^{\frac{1}{2}} \acute{\mathbf{Z}}^{(2d)} \left( \acute{\mathbf{Z}}^{(2d)} \right)^{T} \left(\acute{\mathbf{\Sigma}}_{T}^{(2d)}\right)^{\frac{1}{2}}=\left(\frac{\acute{\mathbf{\Sigma}}_{T}^{(2d)}}{d^{\alpha}}\right)^{\frac{1}{2}} \acute{\mathbf{Z}}^{(2d)} \left( \acute{\mathbf{Z}}^{(2d)} \right)^{T} \left(\frac{\acute{\mathbf{\Sigma}}_{T}^{(2d)}}{d^{\alpha}}\right)^{\frac{1}{2}}.
\end{align*}
Therefore, the limiting values of $\langle c_{1}, d_{1} \rangle /d^{\alpha}$, $\| c_{1} \|_{2}^{2} /d^{\alpha}$ and $\| c_{1} \|_{2}^{2} /d^{\alpha}$ are identified as the (2,2), (1,1) and (2,1) entries of the following matrix,respectively,
\begin{align*}
\begin{bmatrix}
\sigma_{X}^{2} &\rho\sigma_{X}\sigma_{Y}\\
\rho\sigma_{X}\sigma_{Y} &\sigma_{X}^{2}
\end{bmatrix}^{\frac{1}{2}}
\begin{bmatrix}
z_{X1\bullet}\\
z_{Y1\bullet}
\end{bmatrix}
\begin{bmatrix}
z_{X1\bullet}\\
z_{Y1\bullet}
\end{bmatrix}^{T}
\begin{bmatrix}
\sigma_{X}^{2} &\rho\sigma_{X}\sigma_{Y}\\
\rho\sigma_{X}\sigma_{Y} &\sigma_{X}^{2}
\end{bmatrix}^{\frac{1}{2}}.
\end{align*}

\subsubsection{Behavior of the rest of sample canonical weight vectors}

We use the equation~(\ref{sec3:equ28}) for the second canonical weight vector $\hat{\psi}_{X2}^{(d)}$. Recall from~(\ref{sec3:equ25}) that $a_{1}^{(d)}$ of $\hat{\eta}_{RXi}^{(d)}$,for $i=2,3,\dots,n$, converges to 0, which implies,
\begin{align*}
&\left( \sum\limits_{i=1}^{n} \frac{a_{i}^{(d)}}{\sqrt{\hat{\lambda}_{Xi}^{(d)}}} \langle e_{1}^{(d)}, \hat{\xi}_{Xi}^{(d)} \rangle- \sum\limits_{j=2}^{n} \frac{a_{j}^{(d)}}{\sqrt{\hat{\lambda}_{Xj}^{(d)}}} \langle e_{1}^{(d)}, \hat{\xi}_{Xj}^{(d)} \rangle \right)^{2} \mathop{\longrightarrow}_{d \to \infty}^{P} 0,\\
&\left( \sum\limits_{i=1}^{n} \frac{a_{i}^{(d)}}{\sqrt{\hat{\lambda}_{Xi}^{(d)}}} \langle e_{2}^{(d)}, \hat{\xi}_{Xi}^{(d)} \rangle- \sum\limits_{j=2}^{n} \frac{a_{j}^{(d)}}{\sqrt{\hat{\lambda}_{Xj}^{(d)}}} \langle e_{2}^{(d)}, \hat{\xi}_{Xj}^{(d)} \rangle \right)^{2} \mathop{\longrightarrow}_{d \to \infty}^{P} 0,\\
&\left\| \sum\limits_{i=1}^{n} \frac{a_{i}^{(d)}}{\sqrt{\hat{\lambda}_{Xi}^{(d)}}} \hat{\xi}_{Xi}^{(d)}- \sum\limits_{j=2}^{n} \frac{a_{j}^{(d)}}{\sqrt{\hat{\lambda}_{Xj}^{(d)}}} \hat{\xi}_{Xj}^{(d)} \right\|_{2}^{2} \mathop{\longrightarrow}_{d \to \infty}^{P} 0.
\end{align*}
Then,~(\ref{sec3:equ28}) becomes,
\begin{align*}
\left( \langle \hat{\psi}_{X2}^{(d)}, \psi_{X1}^{(d)} \rangle - \frac{\cos \theta_{X} \left( \sum\limits_{j=2}^{n} \frac{da_{j}^{(d)}}{\sqrt{\hat{\lambda}_{Xj}^{(d)}}} \langle e_{1}^{(d)}, \hat{\xi}_{Xj}^{(d)} \rangle \right)+\sin \theta_{X} \left( \sum\limits_{j=2}^{n} \frac{da_{j}^{(d)}}{\sqrt{\hat{\lambda}_{Xj}^{(d)}}} \langle e_{1}^{(d)}, \hat{\xi}_{Xj}^{(d)} \rangle \right)}{\left \| \sum\limits_{j=2}^{n} \frac{da_{j}^{(d)}}{\sqrt{\hat{\lambda}_{Xj}^{(d)}}} \hat{\xi}_{Xj}^{(d)} \right\|_{2}} \right)^{2} \mathop{\longrightarrow}_{d \to \infty}^{P} 0.
\end{align*}
We know from\ref{lem10} that $d/\sqrt{\hat{\lambda}_{Xi}^{(d)}}$, for $i=2,3,\dots,n$, is of $(\asymp O_{P}(1))$ and that $\langle e_{1}^{(d)}, \hat{\xi}_{Xi}^{(d)} \rangle$ and $\langle e_{2}^{(d)}, \hat{\xi}_{Xi}^{(d)} \rangle$, for $i=2,3,\dots,n$, both converge in probability to 0. Therefore, we finally have, 
\begin{align*}
\left( \langle \hat{\psi}_{Xi}^{(d)}, \psi_{X1}^{(d)} \rangle - 0 \right)^{2} \mathop{\longrightarrow}_{d \to \infty}^{P} 0, \,\, i=2,3,\dots,n.
\end{align*}
Results are similar for $\hat{\psi}_{Xi}^{(d)}$ for $i=3,4,\dots,n$. By the similar argument,
\begin{align*}
\left( \langle \hat{\psi}_{Yi}^{(d)}, \psi_{Y1}^{(d)} \rangle - 0 \right)^{2} \mathop{\longrightarrow}_{d \to \infty}^{P} 0, \,\, i=2,3,\dots,n.
\end{align*}

\subsubsection{Behavior of the rest of sample canonical correlation coefficients}

The asymptotic $i$th sample canonical correlation coefficients, for $i=2,3,\dots,n$, is found as the $i$th singular value of the matrix $\hat{\mathbf{R}}^{(d)}$ in~(\ref{sec3:equ25}) under the limiting operation of $d \to \infty$,
\begin{align*}
\hat{\rho}_{i}\mathop{\longrightarrow}_{d \to \infty}^{p} 0, \,\, i=1,2,\dots,n.
\end{align*}

\subsection{Case of \texorpdfstring{$\alpha < 1$}{}}

We, now, prove Theorem~\ref{thm1} under the condition of $\alpha < 1$ with the spiked model of the population covariance structure of $X^{(d)}$ and $Y^{(d)}$ decribed in~(\ref{sec3:equ8}).

\subsubsection{Behavior of the cross-covariance matrix}

\begin{lem}[CCA HDLSS Asymptotic lemma 12.]
\label{lem12}
For $\alpha < 1$, the magnitudes of the sample singular values $\hat{\lambda}_{XYi}^{(d)}$ for $i=1,2,..,n$ are of $O_{p}(d)$ as $d \to \infty$.
\end{lem}
\begin{proof}
Using the Cauchy-Schwarz inequality,
\begin{align}
\label{sec3:equ30}
\begin{aligned}
\left\| \frac{n\hat{\mathbf{\Sigma}}^{(d)}_{XY}}{d}\right\|_{F}^{2}
=&\frac{\langle c_{1}, d_{1} \rangle^{2}}{d^{2-2\alpha}d^{2\alpha}}
+\frac{\langle c_{1}, d_{2} \rangle^{2}}{d^{2-\alpha}d^{\alpha}}
+\frac{\langle c_{2}, d_{1} \rangle^{2}}{d^{2-\alpha}d^{\alpha}}
+\frac{\langle c_{2}, d_{2} \rangle^{2}}{d^{2}}\\
&+\tau_{Y}^{2}\sum_{i=3}^{d} \frac{\langle c_{1}, z_{Yi\bullet} \rangle^{2}}{d^{2-\alpha}d^{\alpha}}
+\tau_{Y}^{2}\sum_{i=3}^{d} \frac{\langle c_{2}, z_{Yi\bullet} \rangle^{2}}{d^{2}}
+\tau_{X}^{2}\sum_{i=3}^{d} \frac{\langle z_{Xi\bullet}, d_{1} \rangle^{2}}{d^{2-\alpha}d^{\alpha}}\\
&+\tau_{X}^{2}\sum_{i=3}^{d} \frac{\langle z_{Xi\bullet}, d_{2} \rangle^{2}}{d^{2}}
+\tau_{X}^{2}\tau_{Y}^{2}\sum_{i=3}^{d}\sum_{j=3}^{d} \frac{\langle z_{Xi\bullet}, z_{Yj\bullet} \rangle^{2}}{d^{2}}\\
\le &\frac{\langle c_{1}, d_{1} \rangle^{2}}{d^{2-2\alpha}d^{2\alpha}}
+\frac{\langle c_{1}, d_{2} \rangle^{2}}{d^{2-\alpha}d^{\alpha}}
+\frac{\langle c_{2}, d_{1} \rangle^{2}}{d^{2-\alpha}d^{\alpha}}
+\frac{\langle c_{2}, d_{2} \rangle^{2}}{d^{2}}\\
&+\tau_{Y}^{2}\sum_{i=3}^{d} \frac{\| c_{1} \|_{2}^{2} \| z_{Yi\bullet} \|_{2}^{2}}{d^{2-\alpha}d^{\alpha}}
+\tau_{Y}^{2}\sum_{i=3}^{d} \frac{\| c_{2} \|_{2}^{2} \| z_{Yi\bullet} \|_{2}^{2}}{d^{2}}
+\tau_{X}^{2}\sum_{i=3}^{d} \frac{\| z_{Xi\bullet} \|_{2}^{2} \| d_{1} \|_{2}^{2}}{d^{2-\alpha}d^{\alpha}}\\
&+\tau_{X}^{2}\sum_{i=3}^{d} \frac{\| z_{Xi\bullet}\|_{2}^{2} \|d_{2} \|_{2}^{2}}{d^{2}}
+\tau_{X}^{2}\tau_{Y}^{2}\sum_{i=3}^{d}\sum_{j=3}^{d} \frac{\| z_{Xi\bullet}\|_{2}^{2} \| z_{Yj\bullet} \|_{2}^{2}}{d^{2}}.
\end{aligned}
\end{align}
Since $\langle c_{1}, d_{1} \rangle^{2}$ is of $O_{P}(d^{\alpha})$, $\langle c_{1}, d_{2} \rangle^{2}$ and $\langle c_{2}, d_{1} \rangle^{2}$ are of $O_{P}(\sqrt{d^{\alpha}})$, and $\langle c_{2}, d_{2} \rangle^{2}$ is of $O_{P}(1)$, using $\alpha < 1$,
\begin{align*}
\frac{\langle c_{1}, d_{1} \rangle^{2}}{d^{2-2\alpha}d^{2\alpha}} \mathop{\longrightarrow}_{d \to \infty}^{p} 0, \,\, \frac{\langle c_{1}, d_{2} \rangle^{2}}{d^{2\alpha}} \mathop{\longrightarrow}_{d \to \infty}^{p} 0, \,\, \frac{\langle c_{2}, d_{1} \rangle^{2}}{d^{2\alpha}} \mathop{\longrightarrow}_{d \to \infty}^{p} 0, \,\, \frac{\langle c_{2}, d_{2} \rangle^{2}}{d^{2\alpha}} \mathop{\longrightarrow}_{d \to \infty}^{p} 0.
\end{align*}
It is not hard to see that $\| z_{Xi\bullet} \|_{2}^{2} \sim \chi_{n}^{2}$ and $\| z_{Yi\bullet} \|_{2}^{2} \sim \chi_{n}^{2}$. Note that $\| c_{1} \|_{2}^{2}$ and $\| d_{1} \|_{2}^{2}$ are of $O_{P}(d^{\alpha})$ and that $\| c_{2} \|_{2}^{2}$ and $\| d_{2} \|_{2}^{2}$ are of $O_{P}(1)$. With the law of large numbers and the condition that $\alpha < 1$,
\begin{align*}
\tau_{Y}^{2}\sum_{i=3}^{d} \frac{\| c_{1} \|_{2}^{2} \| z_{Yi\bullet} \|_{2}^{2}}{d^{2-\alpha}d^{\alpha}}
&=\frac{\tau_{Y}^{2}}{d^{1-\alpha}} \left \| \frac{c_{1}}{\sqrt{d^{\alpha}}} \right \|_{2}^{2} \sum_{i=3}^{d} \frac{\| z_{Yi\bullet} \|_{2}^{2}}{d}\\
&\mathop{\longrightarrow}_{d \to \infty}^{p} 0 \times O_{P}(1) \times E(\chi_{n}^{2}) = 0,\\
\tau_{Y}^{2}\sum_{i=3}^{d} \frac{\| c_{2} \|_{2}^{2} \| z_{Yi\bullet} \|_{2}^{2}}{d^{2}}
&=\frac{\tau_{Y}^{2} \| c_{2} \|_{2}^{2}}{d} \sum_{i=3}^{d} \frac{\| z_{Yi\bullet} \|_{2}^{2}}{d}\\
&\mathop{\longrightarrow}_{d \to \infty}^{p} 0 \times E(\chi_{n}^{2}) = 0,\\
\tau_{X}^{2}\sum_{i=3}^{d} \frac{\| d_{1} \|_{2}^{2} \| z_{Xi\bullet} \|_{2}^{2}}{d^{2-\alpha}d^{\alpha}}
&=\frac{\tau_{X}^{2}}{d^{1-\alpha}} \left \| \frac{d_{1}}{\sqrt{d^{\alpha}}} \right \|_{2}^{2} \sum_{i=3}^{d} \frac{\| z_{Xi\bullet} \|_{2}^{2}}{d}\\
&\mathop{\longrightarrow}_{d \to \infty}^{p} 0 \times O_{P}(1) \times E(\chi_{n}^{2}) = 0,\\
\tau_{X}^{2}\sum_{i=3}^{d} \frac{\| d_{2} \|_{2}^{2} \| z_{Xi\bullet} \|_{2}^{2}}{d^{2}}
&=\frac{\tau_{X}^{2} \| d_{2} \|_{2}^{2}}{d} \sum_{i=3}^{d} \frac{\| z_{Xi\bullet} \|_{2}^{2}}{d}\\
&\mathop{\longrightarrow}_{d \to \infty}^{p} 0 \times E(\chi_{n}^{2}) = 0.
\end{align*}
Using $\alpha > 1$ and the law of large numbers,
\begin{align*}
\tau_{X}^{2}\tau_{Y}^{2}\sum_{i=3}^{d}\sum_{j=3}^{d} \frac{\| z_{Xi\bullet}\|_{2}^{2} \| z_{Yj\bullet} \|_{2}^{2}}{d^{2}}&=\tau_{X}^{2}\tau_{Y}^{2} \sum_{i=3}^{d} \frac{\| z_{Xi\bullet} \|_{2}^{2}}{d} \sum_{j=3}^{d} \frac{\| z_{Yj\bullet} \|_{2}^{2}}{d}\\
&\mathop{\longrightarrow}_{d \to \infty}^{p} \tau_{X}^{2}\tau_{Y}^{2} \times E(\chi_{n}^{2}) \times E(\chi_{n}^{2})\\ 
&= \tau_{X}^{2}\tau_{Y}^{2}n^{2}.
\end{align*}
We have,
\begin{align*}
\left\| \frac{\hat{\mathbf{\Sigma}}^{(d)}_{XY}}{d}\right\|_{F}^{2} = \sum_{i=1}^{n} \left( \frac{\hat{\lambda}_{XYi}^{(d)}}{d} \right)^{2} \mathop{\longrightarrow}_{d \to \infty}^{p} \le \tau_{X}^{2}\tau_{Y}^{2}.
\end{align*}
Therefore, the magnitudes of the sample singular values $\hat{\lambda}_{XYi}^{(d)}$ for $i=1,2,..,n$ are of $O_{p}(d)$.
\end{proof}
\begin{lem}[CCA HDLSS Asymptotic lemma 13.]
\label{lem13}
This lemma improves Lemma~\ref{lem12}. For $\alpha < 1$, the sample singular values $n\hat{\lambda}_{XYi}^{(d)}$ for $i=1,2,..,n$ are of $(\asymp O_{p}(d))$ with the following limiting quantity as $d \to \infty$,
\begin{align*}
\frac{\hat{\lambda}_{XYi}^{(d)}}{d} \mathop{\longrightarrow}_{d \to \infty}^{P} = \frac{\tau_{X}\tau_{Y}}{n}, \,\, i=1,2,..,n.
\end{align*}
\end{lem}
\begin{proof}
First we calculate the exact limiting value of the Frobenius norm of $\hat{\mathbf{\Sigma}}_{XY}^{(d)}/d$ as $d \to \infty$. In~(\ref{sec3:equ30}), the first 8 terms converges in probability to 0 as $d \to \infty$. The last term of~(\ref{sec3:equ30}) expands as,
\begin{align*}
\sum_{i=3}^{d}\sum_{j=3}^{d} \frac{\langle z_{Xi\bullet}, z_{Yj\bullet} \rangle^{2}}{d^{2}} &=\sum_{i=3}^{d} \sum_{j=3}^{d} \left( \sum_{k=1}^{n} \frac{ z_{Xik} z_{Yjk} }{d} \right)^{2}\\
&=\sum_{k=1}^{n} \left( \sum_{i=3}^{d} \frac{z_{Xik}^{2}}{d} \right) \left(\sum_{i=3}^{d} \frac{z_{Yjk}^{2}}{d} \right)\\
&\quad +\sum_{k \not= k'} \left( \sum_{i=3}^{d} \frac{z_{Xik} z_{Yik'}}{d} \right) \left(\sum_{j=3}^{d} \frac{z_{Xjk} z_{Yjk'}}{d} \right).
\end{align*}
Since $E(z_{yik} z_{xik'}) = 0$ for $k \not= k'$, by the law of large numberss,
\begin{align*}
\sum_{k \not= k'} \left( \sum_{i=3}^{d} \frac{z_{yik} z_{xik'}}{d} \right) \left(\sum_{j=3}^{d} \frac{z_{yjk} z_{xjk'}}{d} \right) \mathop{\longrightarrow}_{d \to \infty}^{P} 0.
\end{align*}
Since $z_{yik}^{2}$ and $z_{xik}^{2}$ follow a Chi-squared distribution with a degree of freedom of 1, by the law of large numberss,
\begin{align*}
\sum_{k=1}^{n} \left( \sum_{i=3}^{d} \frac{z_{yik}^{2}}{d} \right) \left(\sum_{i=3}^{d} \frac{z_{xjk}^{2}}{d} \right) \mathop{\longrightarrow}_{d \to \infty}^{P} n \times 1 \times 1 = n.
\end{align*}
Hence,
\begin{align}
\label{sec3:equ32}
\left \| \frac{n\hat{\mathbf{\Sigma}}_{XY}^{(d)}}{d} \right\|_{F}^{2} \mathop{\longrightarrow}_{d \to \infty}^{P} n\tau_{X}^{2} \tau_{Y}^{2}.
\end{align}
Now we set a bound for each sample singular value $\hat{\lambda}_{XYi}^{(d)}/d$ for $i=1,2,..,n$ as $d \to \infty$. Consider the data matrix $\mathbf{X}^{(d)}$ and $\mathbf{Y}^{(d)}$ in~(\ref{sec3:equ17}). Write the singular vecotrs $\hat{\eta}_{Xi}^{(d)}$ and $\hat{\eta}_{Yi}^{(d)}$ of the sample cross-covariance matrix $\hat{\mathbf{\Sigma}}^{(d)}_{XY}$ of~(\ref{sec3:equ19}) by the linear combinations of the sample eigenvectors $\hat{\xi}_{Xj}^{(d)}$ and $\hat{\xi}_{Yj}^{(d)}$,
\begin{align*}
\hat{\eta}_{Xi}^{(d)} = \sum_{j=1}^{n} a_{j} \hat{\xi}_{Xj}^{(d)}, \,\, \hat{\eta}_{Yi}^{(d)} = \sum_{j=1}^{n} b_{j} \hat{\xi}_{Yj}^{(d)}, \,\, i=1,2,..,n,
\end{align*}
where $a_{1}^{2}+a_{2}^{2}+\dots+a_{n}^{2} = b_{1}^{2}+b_{2}^{2}+\dots+b_{n}^{2} = 1$ to ensure a unit length. Using $Cov^{2}(X,Y) \le Var(X)Var(Y)$,
\begin{align*}
\left( \hat{\lambda}_{XYi}^{(d)} \right)^{2} &= \frac{1}{n^{2}}\left\langle \left( \mathbf{X}^{(d)} \right)^{T} \hat{\eta}_{Xi}^{(d)}, \left( \mathbf{Y}^{(d)} \right)^{T} \hat{\eta}_{Yi}^{(d)} \right\rangle^{2}\\
&\le \frac{1}{n^{2}}\left\langle \left( \mathbf{X}^{(d)} \right)^{T} \hat{\eta}_{Xi}^{(d)}, \left( \mathbf{X}^{(d)} \right)^{T} \hat{\eta}_{Xi}^{(d)} \right\rangle \left\langle \left( \mathbf{Y}^{(d)} \right)^{T} \hat{\eta}_{Yi}^{(d)}, \left( \mathbf{Y}^{(d)} \right)^{T} \hat{\eta}_{Yi}^{(d)} \right\rangle.
\end{align*}
Using the fact that the covariance between two sets of scores of the projections of the observations onto any two different eigenvectors is 0,
\begin{align*}
\frac{1}{n}\left\langle \left( \mathbf{X}^{(d)} \right)^{T} \hat{\eta}_{Xi}^{(d)}, \left( \mathbf{X}^{(d)} \right)^{T} \hat{\eta}_{Xi}^{(d)} \right\rangle = &\frac{1}{n}\left\langle \left( \mathbf{X}^{(d)} \right)^{T} \sum_{j=1}^{n} a_{j} \hat{\xi}_{Xj}^{(d)}, \left( \mathbf{X}^{(d)} \right)^{T} \sum_{j=1}^{n} a_{j} \hat{\xi}_{Xj}^{(d)} \right\rangle\\
= &\sum_{j = j'} \frac{1}{n}\left\langle \left( \mathbf{X}^{(d)} \right)^{T} a_{j} \hat{\xi}_{Xj}^{(d)}, \left( \mathbf{X}^{(d)} \right)^{T} a_{j'} \hat{\xi}_{Xj'}^{(d)} \right\rangle\\ 
&+ \sum_{j \not= j'} \frac{1}{n}\left\langle \left( \mathbf{X}^{(d)} \right)^{T} a_{j} \hat{\xi}_{Xj}^{(d)}, \left( \mathbf{X}^{(d)} \right)^{T} a_{j'} \hat{\xi}_{Xj'}^{(d)} \right\rangle\\
= &\sum_{j = j'} \frac{1}{n}\left\langle \left( \mathbf{X}^{(d)} \right)^{T} a_{j} \hat{\xi}_{Xj}^{(d)}, \left( \mathbf{X}^{(d)} \right)^{T} a_{j'} \hat{\xi}_{Xj'}^{(d)} \right\rangle\\ 
= &\sum_{j=1}^{n} \frac{a_{j}^{2}}{n} \left\langle \left( \mathbf{X}^{(d)} \right)^{T} \hat{\xi}_{Xj}^{(d)}, \left( \mathbf{X}^{(d)} \right)^{T} \hat{\xi}_{Xj}^{(d)} \right\rangle\\
= &\sum_{j=1}^{n} a_{j}^{2} \hat{\lambda}_{Xj}^{(d)}.
\end{align*}
Similarly,
\begin{align*}
\frac{1}{n}\left\langle \left( \mathbf{Y}^{(d)} \right)^{T} \hat{\eta}_{Yi}^{(d)}, \left( \mathbf{Y}^{(d)} \right)^{T} \hat{\eta}_{Yi}^{(d)} \right\rangle = \sum_{k=1}^{n} b_{k}^{2} \hat{\lambda}_{Yk}^{(d)}.
\end{align*}
By Lemma~\ref{lem10},
\begin{align*}
\left( \frac{\hat{\lambda}_{XYi}^{(d)}}{d} \right)^{2} &\le \frac{1}{n^{2}}\frac{\left\langle \left( \mathbf{X}^{(d)} \right)^{T} \hat{\eta}_{Xi}^{(d)}, \left( \mathbf{X}^{(d)} \right)^{T} \hat{\eta}_{Xi}^{(d)} \right\rangle}{d} \frac{\left\langle \left( \mathbf{Y}^{(d)} \right)^{T} \hat{\eta}_{Yi}^{(d)}, \left( \mathbf{Y}^{(d)} \right)^{T} \hat{\eta}_{Yi}^{(d)} \right\rangle}{d}\\
&=\left( \sum_{j=1}^{n} a_{j}^{2} \frac{\hat{\lambda}_{Xj}^{(d)}}{d} \right) \left( \sum_{k=1}^{n} b_{k}^{2} \frac{\hat{\lambda}_{Yk}^{(d)}}{d} \right)\\
& \mathop{\longrightarrow}_{d \to \infty}^{P} \frac{\tau_{X}^{2}}{n}\frac{\tau_{Y}^{2}}{n} \left( \sum_{j=1}^{n} a_{j}^{2} \right) \left( \sum_{k=1}^{n} b_{k}^{2} \right)\\ 
&= \frac{\tau_{X}^{2}\tau_{Y}^{2}}{n^{2}}, \,\, i=1,2,..,n.
\end{align*}
The constraints above and the proven fact of~(\ref{sec3:equ32}),
\begin{align*}
\left \| \frac{\hat{\mathbf{\Sigma}}_{XY}^{(d)}}{d} \right\|_{F}^{2} = \sum_{i=1}^{n} \left( \frac{\hat{\lambda}_{XYi}^{(d)}}{d} \right)^{2} \mathop{\longrightarrow}_{d \to \infty}^{P} \frac{1}{n}\tau_{X}^{2}\tau_{Y}^{2},
\end{align*}
implies that,
\begin{align*}
\frac{\hat{\lambda}_{XYi}^{(d)}}{d} \mathop{\longrightarrow}_{d \to \infty}^{P} \frac{\tau_{X}\tau_{Y}}{n}, \,\, i=1,2,..,n.
\end{align*}
\end{proof}

\subsubsection{Behavior of the sample covariance matrices}

Here, we only include the result about the sample covariance matrix $\hat{\mathbf{\Sigma}}^{(d)}_{X}$. That of $\hat{\mathbf{\Sigma}}^{(d)}_{Y}$ is similar.

\begin{lem}[CCA HDLSS Asymptotic lemma 14.]
\label{lem14}
The sample eigenvalue $\hat{\lambda}_{Xi}^{(d)}$ and eigenvector $\hat{\xi}_{Xi}^{(d)}$, $i=2,3,..,d$, converge in probability to the following quantities as $d \to \infty$,
\begin{align*}
\frac{\hat{\lambda}_{Xi}^{(d)}}{d} \mathop{\longrightarrow}_{d \to \infty}^{P} \frac{\tau_{X}^{2}}{n}, \,\, \langle \hat{\xi}_{Xi}^{(d)}, \xi^{(d)} \rangle \mathop{\longrightarrow}_{d \to \infty}^{P} 0, \,\, i=1,2,\dots,n,
\end{align*}
where $\xi^{(d)}$ is an any given vector in $R^{d}$.
\end{lem}
\begin{proof}
Recall that the underlying random variable $X^{(d)}$ of the sample covariance matrix $\hat{\mathbf{\Sigma}}_{X}^{(d)}$ has a simple spiked covariance structure of~(\ref{sec3:equ8}). Then, the asymptotic behavior of a sample eigenvalue is a direct consequence of Theorem~\ref{thm1} for $\alpha < 1$. The result on the behavior of $\langle \hat{\xi}_{Xi}^{(d)}, \xi^{(d)} \rangle$ is indicated in the proof of Theorem~\ref{thm1} given in~\cite{38}.
\end{proof}

\subsubsection{Behavior of the matrix \texorpdfstring{$\hat{\mathbf{R}}^{(d)}$}{}}

Lemma~\ref{lem13} states that the the singular values of the sample cross-covariance matrix $\hat{\mathbf{\Sigma}}^{(d)}_{XY}$ become indistinguishable in the limit of $d \to \infty$, in which situation there exist an infinite number of choices for a singular vector set. We choose its singular vectors such that the sample eigenvectors of the sample covariance matrices $\hat{\mathbf{\Sigma}}^{(d)}_{X}$ and $\hat{\mathbf{\Sigma}}^{(d)}_{X}$ are their limiting quantities,
\begin{align*}
\left\| \hat{\eta}_{Xi}^{(d)} - \hat{\xi}_{Xi}^{(d)} \right\|_{2}^{2} \mathop{\longrightarrow}_{d \to \infty}^{P} 0, \,\, \left\| \hat{\eta}_{Yj}^{(d)} - \hat{\xi}_{Yj}^{(d)} \right\|_{2}^{2} \mathop{\longrightarrow}_{d \to \infty}^{P} 0, \,\, i,j=1,2,\dots, n.
\end{align*}
Then, it is easily shown that the matrix $\hat{\mathbf{R}}^{(d)}$ defined in~(\ref{sec3:equ22}), reduces to the following limiting form, using the limiting values of the sample singular vectors and eigenvectors given in Lemma~\ref{lem13} and~\ref{lem14},
\begin{align}
\label{sec3:equ33}
\left\| \hat{\mathbf{R}}^{(d)}-
\begin{bmatrix}
\hat{\xi}_{X1}^{(d)} &\hat{\xi}_{X2}^{(d)} &\ldots &\hat{\xi}_{Xn}^{(d)}
\end{bmatrix}
\begin{bmatrix}
1 &0 &\ldots &0\\
0 &1 &\ldots &0\\
\vdots &\vdots &\ddots &\vdots\\
0 &0 &\ldots &1
\end{bmatrix}
\begin{bmatrix}
\left( \hat{\xi}_{Y1}^{(d)} \right)^{T}\\
\left( \hat{\xi}_{Y2}^{(d)} \right)^{T}\\
\vdots\\
\left( \hat{\xi}_{Yn}^{(d)} \right)^{T}
\end{bmatrix}\right\|_{F}^{2}  \mathop{\longrightarrow}_{d \to \infty}^{P} 0.
\end{align}

\subsubsection{Behavior of sample canonical weight vectors}

Use the equation~(\ref{sec3:equ28}) for a sample canonical weight vector $\hat{\psi}_{Xi}^{(d)}$ represented as in~(\ref{sec3:equ26}) for a given $i$. We know from Lemma~\ref{lem14} that the magnitude of $d/\sqrt{\hat{\lambda}_{Xi}^{(d)}}$, for $i=1,2,\dots,n$, is of $(\asymp O_{P}(1))$ and that $\langle e_{1}^{(d)}, \hat{\xi}_{Xi}^{(d)} \rangle$ and $\langle e_{2}^{(d)}, \hat{\xi}_{Xi}^{(d)} \rangle$, for $i=1,2,\dots,n$, converge in probability to 0 as $d \to \infty$, which leads to,
\begin{align*}
\left( \langle \hat{\psi}_{Xi}^{(d)}, \psi_{X1}^{(d)} \rangle - \frac{\cos \theta_{X} \left( \sum\limits_{j=1}^{n} \frac{da_{j}^{(d)}}{\sqrt{\hat{\lambda}_{Xj}^{(d)}}} \langle e_{1}^{(d)}, \hat{\xi}_{Xj}^{(d)} \rangle \right)+\sin \theta_{X} \left( \sum\limits_{j=1}^{n} \frac{da_{j}^{(d)}}{\sqrt{\hat{\lambda}_{Xj}^{(d)}}} \langle e_{1}^{(d)}, \hat{\xi}_{Xj}^{(d)} \rangle \right)}{\left \| \sum\limits_{j=1}^{n} \frac{da_{j}^{(d)}}{\sqrt{\hat{\lambda}_{Xj}^{(d)}}} \hat{\xi}_{Xj}^{(d)} \right\|_{2}} \right)^{2} \mathop{\longrightarrow}_{d \to \infty}^{P} 0.
\end{align*}
Hence, we have,
\begin{align*}
\left( \langle \hat{\psi}_{Xi}^{(d)}, \psi_{X1}^{(d)} \rangle - 0 \right)^{2} \mathop{\longrightarrow}_{d \to \infty}^{P} 0, \,\, i=1,2,\dots,n.
\end{align*}
Similarly,
\begin{align*}
\left( \langle \hat{\psi}_{Yi}^{(d)}, \psi_{Y1}^{(d)} \rangle - 0 \right)^{2} \mathop{\longrightarrow}_{d \to \infty}^{P} 0, \,\, i=1,2,\dots,n.
\end{align*}

\subsubsection{Behavior of sample canonical correlation coefficients}

The asymptotic $i$th sample canonical correlation coefficients is found as the $i$th singular value of the matrix~(\ref{sec3:equ33}) under the limiting operation of $d \to \infty$,
\begin{align*}
\hat{\rho}_{i}\mathop{\longrightarrow}_{d \to \infty}^{P}1, \,\, i=1,2,\dots,n.
\end{align*}

\section{Simulation}

Simulation study in this section aims at verifing the asymptotic behavior of sample canonical correlation coefficients and their corresponding weight vectors given in the main theorom~\ref{thm1} as dimension $d$ grows with sample size $n$ fixed. We first state the parameter settings to be used. For the spiked covariance structures of the random variables $X^{(d)}$ and $Y^{(d)}$ described in~(\ref{sec3:equ6}) and~(\ref{sec3:equ7}), we set $\sigma_{X}^{2}=\tau_{X}^{(d)}=\sigma_{Y}^{2}=\tau_{Y}^{(d)}=1$. The population caconical weight vectors described in~(\ref{sec3:equ13}) and population caconical correlation coefficient are set to be,
\begin{figure}[t!]
\centering
\includegraphics[width=14cm, height=16cm]{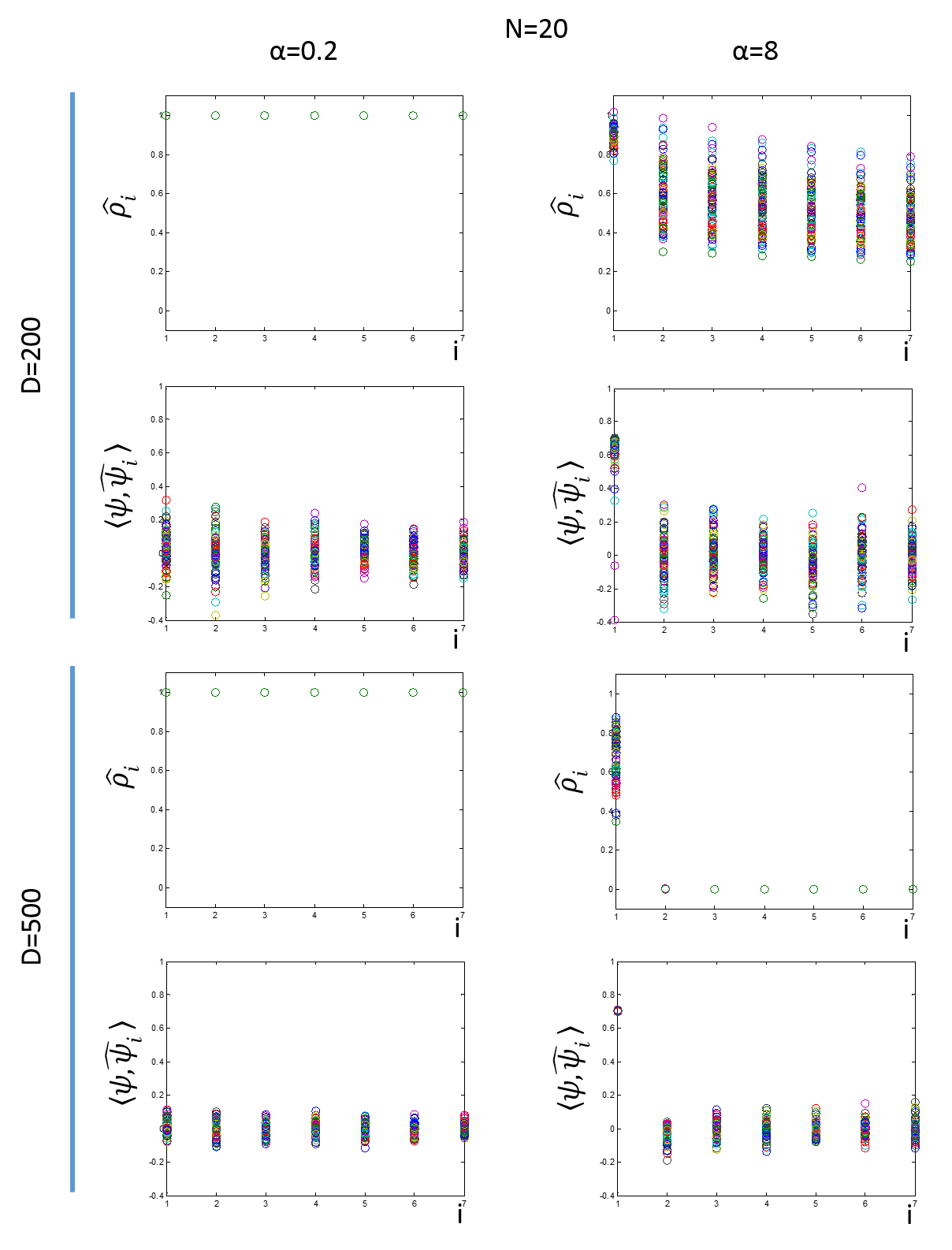}
\caption{Estimated sample canonical correlation coefficients $\hat{\rho}_{i}^{(d)}$ and inner products of the sample left canonical weight vectors $\hat{\psi}_{Xi}^{(d)}$ and the population canonical weight vector $\psi_{Xi}^{(d)}$, for $i=1,2,\dots,5$, obtained from 100 repetitions of simulations for different settings of dimension $d$ and exponent $\alpha$ with a sample size of $n=20$.}
\label{sec3:fig1}
\end{figure}\begin{figure}[t!]
\centering
\includegraphics[width=14cm, height=16cm]{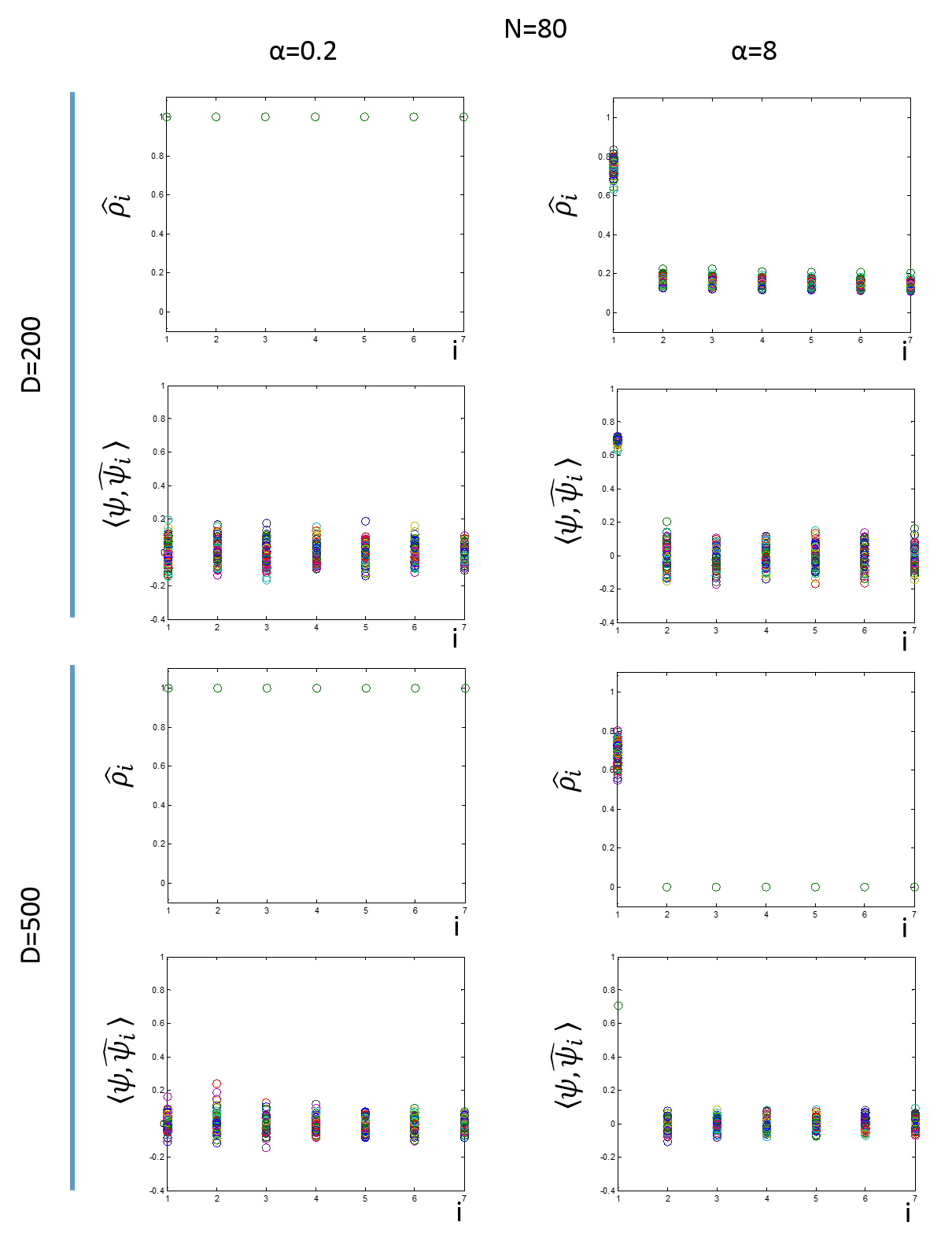}
\caption{100 estimated sample canonical correlation coefficients $\hat{\rho}_{i}^{(d)}$ and inner products of the sample left canonical weight vectors $\hat{\psi}_{Xi}^{(d)}$ and the population canonical weight vector $\psi_{Xi}^{(d)}$, for $i=1,2,\dots,5$, obtained from 100 repetitions of simulations for different settings of dimension $d$ and exponent $\alpha$ with a sample size of $n=80$.}
\label{sec3:fig2}
\end{figure}
\begin{align*}
\psi_{X}^{(d)}=(\cos 0.75\pi) e_{1}^{(d)} + (\sin 0.75\pi) e_{2}^{(d)}, \,\, \psi_{Y}^{(d)}=(\cos 0.75\pi) e_{1}^{(d)} + (\sin 0.75\pi) e_{2}^{(d)}, \,\, \rho=0.7.
\end{align*}
Note that $\langle \psi_{X}^{(d)}, e_{1}^{(d)} \rangle=\langle \psi_{Y}^{(d)}, e_{1}^{(d)} \rangle=\cos 0.75\pi=0.7071$, which implies that the angle between $\psi_{X}^{(d)}$ and $e_{1}^{(d)}$ is $135^{\circ}$. The population cross-covariance structure of $X^{(d)}$ and $Y^{(d)}$ can be accordingly defined as in~(\ref{sec3:equ8}). We perform 100 runs of simulations for each combination of different values of the following three sets,
\begin{itemize}
\item Sample size $n \in \{20,80\}$,
\item Dimension $d \in \{200,500\}$,
\item Exponent $\alpha \in \{0.2, 8\}$.
\end{itemize}
Each case, estimates of the first 5 canonical correlation coefficients $\hat{\rho}_{i}^{(d)}$ and their corresponding canonical weight vectors $\hat{\psi}_{Xi}^{(d)}$ and $\hat{\psi}_{Yi}^{(d)}$ are obtained. The estimated vectors $\hat{\psi}_{Xi}^{(d)}$ and $\hat{\psi}_{Yi}^{(d)}$, for $i=1,2,\dots,5$, are compared to the population canonical weight vector $\psi_{X}^{(d)}$ using their inner product. Here, we do not include results of $\hat{\psi}_{Yi}^{(d)}$ as they are similar as those of $\hat{\psi}_{Xi}^{(d)}$.

Figure~\ref{sec3:fig1} presents the simulation results for a small sample size of $n=20$. For $\alpha=0.2$, sample coefficients and vectors are almost of no use as the estimated vectors tend to be as far away as possible from the popultion direction (implied in the inner products of 0) with always perfect correlation. When $\alpha$ increases to a high strength of 8, the first sample coefficient $\hat{\rho}_{1}^{(d)}$ approachs to the population direction whereas the rest degenerate to 0 as $d \to \infty$. The first left sample canonical weight vector $\hat{\psi}_{X1}^{(d)}$ converges to the direction $e_{1}^{(d)}$ (implied in the inner products of $\cos 0.75\pi$) containing dominant variability as $d \to \infty$ and the rest carry no information on the population direction with tending to deviate from it by a highest degree of $90^{\circ}$.
Figure~\ref{sec3:fig2} illustrates the results for a larger sample size of $n=80$. For the case of $\alpha=0.2$, the behavior of $\hat{\rho}_{i}^{(d)}$ and $\hat{\psi}_{Xi}^{(d)}$ is similar as that in a small sample size case. However, for $\alpha=8$, we see a noticeable decrease in variability of the first sample canonical correlation coefficient $\hat{\rho}_{1}^{(d)}$ around a true value of 0.7 and of the rest of $\hat{\rho}_{i}^{(d)}$ around 0. This implies that the usual large sample theory works for $\hat{\rho}_{1}^{(d)}$. Diminishing variability is also observed for the sample canonical weight vectors $\hat{\psi}_{Xi}^{(d)}$, where the first sample vector $\hat{\psi}_{Xi}^{(d)}$ becomes almost identical to the largest variance direction $e_{1}^{(d)}$ and the rest diverge from the population canonical direction $\psi_{X}^{(d)}$.

\section{Discussion}

A natural question arises about what asymptotic behavior of sample canonical weight vectors we can expect at the boundary case of $\alpha=1$ as $d \to \infty$ with a sample size fixed at $n$ ? When $\alpha > 1$, we saw that the angle between the first sample canonical weight vector and its population counterpart degenerates to 0 and when $\alpha < 1$, the angle between them diverges by as large as $\pi/2$. We conjecture that, with $\alpha=1$, the angle formed by the first sample canonical weight vector and its population counterpart converges weakly to some random variable on the support of $[0,\pi/2]$. We leave an investigation of this conjecture for future work.  

\bibliography{mybibfile}
\end{document}